\pdfoutput=1
\documentclass[11pt,UKenglish]{amsart}
\usepackage{babel}
\usepackage[T1]{fontenc}
\usepackage[utf8]{inputenc}
\usepackage[babel]{microtype}
\usepackage{csquotes}
\usepackage{amscd, color}
\usepackage{amsthm}
\usepackage{caption}
\usepackage{subcaption}
\usepackage{thmtools}
\usepackage{thm-restate}
\usepackage{graphicx}
\usepackage{hyperref}
\usepackage{longtable}
\usepackage[mathscr]{euscript}
\usepackage{etoc}

\hypersetup{pdfencoding=auto,bookmarksnumbered=true,hypertexnames=false,pdftitle={Homogenization of Hamilton-Jacobi equations on networks},pdfauthor={Marco Pozza, Antonio Siconolfi, and Alfonso Sorrentino},pdfsubject={MSC(2020): 35B27, 35R02, 35F21, 37J51, 49L25}, pdfkeywords={homogenization;Hamilton-Jacobi equations;periodic networks;Aubry-Mather theory;weak KAM theory;topological crystals}}

\etocsettocstyle{\section*{\contentsname}\hypersetup{hidelinks}}{}
\etocsetstyle{section}{}{\noindent}{\textbf{\etocname{}}\dotfill\etocpage\smallskip\par}{}
\etocsetstyle{subsection}{\etocskipfirstprefix\leftskip2em\rightskip1.3em\noindent}{, }{\textit{\footnotesize\etocname{} (p. \etocpage)}}{.\par\leftskip0cm\rightskip0cm}

\DeclareCaptionLabelFormat{bf-parens}{\textbf{(#2)}}

\theoremstyle{plain}
\newtheorem*{MainThm}{Main Theorem}
\newtheorem{Thm}{Theorem}[section]
\newtheorem{Lemma}[Thm]{\bf Lemma}
\newtheorem{Corollary}[Thm]{\bf Corollary}
\newtheorem{Theorem}[Thm]{\bf Theorem}
\newtheorem{Proposition}[Thm]{\bf Proposition}
\newtheorem{Notation}[Thm]{\bf Notation}
\theoremstyle{definition}
\newtheorem{Definition}[Thm]{\bf Definition}
\theoremstyle{remark}
\newtheorem{Remark}[Thm]{\bf Remark}
\newtheorem{Example}[Thm]{\bf Example}
\newtheoremstyle{Cl}{5pt}{3pt}{\sl}{}{\it}{:}{.5em}{}
\theoremstyle{Cl}

\graphicspath{{img/}}

\DeclareMathOperator{\spt}{supp \,}
\DeclareMathOperator{\supp}{supp \,}

\def\al{\alpha}
\def\be{\beta}
\def\de{\delta}

\def\eps{\varepsilon}
\def\e{\varepsilon}

\def\ga{\gamma}
\def\G{\Gamma}

\def\wha{\widehat}

\def\la{\lambda}

\def\si{\sigma}
\def\om{\omega}

\def\EE{{\mathbf E}}
\def\VV{{\mathbf V}}
\def\TT{{\mathbf T}}
\def\EN{{\mathcal E}}
\def\VN{{\mathcal V}}

\def\G{\Gamma}

\def\bM{\mathbb M}
\def\bP{\mathbb P}
\def\R{{\mathbb R}}
\def\Z{{\mathbb Z}}

\def\N{{\mathbb N}}

\def\A{{\mathcal A}}

\def\FN{{\mathcal F}}
\def\LL{{\mathcal L}}
\def\HH{{\mathcal H}}
\def\NN{{\mathcal N}}

\def\M{{\mathbb M}}
\def\ov{\overline}

\def\F{\mathcal F}
\def\Cf{\mathfrak C}

\def\oo{\mathrm o}
\def\tt{{\mathrm t}}

\def\F{{\mathcal F}}

\def\wtd{\widetilde}

\def \txt{\qquad\text}

\def\begincproof{

    \begin{proof}
    }
    \def\endcproof{
    \end{proof}

}

\DeclareRobustCommand{\SkipTocEntry}[5]{}

\newcommand{\add}[1]{\textcolor{blue}{#1}}

\swapnumbers

\title[Homogenization of HJ equations on networks]{Homogenization of Hamilton--Jacobi equations on networks}

\author{Marco Pozza}
\address{Link Campus University, Italy.}
\email{m.pozza@unilink.it}

\author{Antonio Siconolfi}
\address{Dipartimento di Matematica, Sapienza Universit\`a di Roma, Italy.}
\email{siconolfi@mat.uniroma1.it}

\author{Alfonso Sorrentino}
\address{Dipartimento di Matematica, Universit\`a degli Studi di Roma ``Tor Vergata'', Rome, Italy.}
\email{sorrentino@mat.uniroma2.it}

\subjclass[2020]{35B27, 35R02, 35F21, 37J51, 49L25}

\keywords{homogenization, Hamilton--Jacobi equations, periodic networks, Aubry--Mather theory, weak KAM theory, topological crystals}

\begin{document}

    \begin{abstract}
        We prove a homogenization result for a family of time-dependent Hamilton--Jacobi equations, rescaled by a parameter $\eps$ tending to zero, posed on a periodic network, with a suitable notion of periodicity that will be defined. As $\eps$ becomes infinitesimal, we derive a limiting Hamilton--Jacobi equation in a Euclidean space, whose dimension is determined by the topological complexity of the network and is independent of the ambient space in which the network is embedded. Among the key contributions of our analysis, we extend to the setting of networks and graphs Mather's result on the asymptotic behavior of the average minimal action functional, as time tends to infinity. Additionally, we establish the well-posedness of the approximating problems, representing a nontrivial generalization of existing results for finite networks to a non-compact setting.
    \end{abstract}

    \maketitle

    \bigskip

    \section{Introduction}

    Starting from the pioneering contributions in \cite{SchCam,BreH}, the literature on Hamilton--Jacobi equations and control models on graphs and networks, as well as on stratified domains, now constitutes a significant body of results that continues to fuel a rather popular line of research.

    By {\it network} we simply mean a subset $\NN \subset \R^N$ of the form
    \begin{equation*}
        \NN := \bigcup_{\gamma \in \EN} \, \gamma([0,1]) \subset \R^N,
    \end{equation*}
    where $\EN$ is a collection of simple $C^1$ regular ({\it i.e.}, with non-vanishing derivative) oriented curves $\gamma: [0,1] \to \R^N$, called {\it arcs} of the network. The initial and final point of any arc $\ga$ have a special status and are called {\it vertices} of the network: if we look at a network as a piecewise regular manifold, vertices are the points where regularity fails. \\
    In the following we assume that the networks are connected. We refer to Subsection \ref{sectionnetwork} for a more comprehensive presentation of network and their properties.

    The analysis of differential problems on networks involves a number of subtle theoretical issues and has a significant impact on applications in various fields, such as traffic models, data transmission, and robotics; in these contexts, the discontinuities in the differential problems are related to some specific structures in the environment. Valuable contributions can be found, for instance, in \cite{AchdouCamilliCutriTchou12,AchdouOudetTchou16,AchdouTchou15, BarlesBrianiChasseigne14, BarlesBrianiChasseigne13,CamilliMarchi13,ImbMon, ImbMonZid, LionsSouganidis17, LionsSouganidis16,Morfe20,PozzaLTB, PozzaSiconolfi21, RaoSiconolfiZidani14}, and, last but not least, in the recent comprehensive monograph \cite{BarlesChasseigne}.

    \medskip

    The main aim of this paper is to prove a homogenization result for a family of oscillating time-dependent Hamilton–Jacobi equations posed on \textit{periodic} networks. The notion of periodicity for networks, a cornerstone of our discussion, will be described more precisely in Subsection \ref{secunderlyinggraph}; in particular our definition implies that the local network is \emph{locally finite}, \textit{i.e.}, there are finitely many arcs departing or arriving at each given vertex.

    However, to provide a preliminary rough idea for the sake of this introduction, it can be expressed as the invariance of the network under the action of a free Abelian group of symmetries ({\it i.e.}, isomorphic to $\Z^b$, for some $b \in \N$), such that the {\it fundamental cell} of this action is a finite graph ({\it i.e.}, with finitely many vertices and arcs). In other words, the network can be viewed as a {\it repetition} of copies of a finite network {\it patched} together through the action of the group of symmetries. We show in Appendix \ref{appexnet} (see also Section \ref{construens}) how to construct examples of periodic networks starting from any choice of finite graph as fundamental cell.

    As usual in an homogenization result, we prove that the approximating equations on the periodic network converge, as the rescaling parameter goes to 0, to a limit Hamilton--Jacobi posed on a Euclidean space, whose dimension is related to the complexity of the periodic network, where an \textit{effective} Hamiltonian $\bar H$, independent of the state variable, appears.

    Some remarks are now in order to corroborate our exposition.

    \medskip

    \begin{itemize}
        \item{\it Hamiltonians}: A {Hamiltonian} on the network is a collection of functions of local Hamiltonians $H:=\{H_\gamma\}_{\gamma\in\EN}$, $H_\gamma:[0,1]\times\R\to\R$, defined on each arc. Since our approach is variational and grounded in the use of the Lagrangian action, we assume each $\{H_\gamma\}_{\gamma\in\EN}$ to satisfy some conditions related to strict convexity and superlinearity on the fibers (see (H$_\ga$1)--(H$_\ga$4) in Subsection \ref{secHam}). We emphasize that, apart from some obvious compatibility conditions related to changing the orientation of an arc (see \eqref{ovgamma}), the Hamiltonians $H_\ga$ are completely unrelated to one another.

        \smallskip

        \item{\it Notion of solution}: The fact that a Hamiltonian on a network is in general discontinuous at the vertices requires an adapted notion of solution. In this paper, we adopt the one introduced in \cite{Sic}, coupled with the choice of the so-called {\it maximal flux limiter} (see Remark \ref{remarkintroflux} and Subsection \ref{secsolutionHJ}).

        \smallskip

        \item{\it Limit space:} An interesting aspect of our result is that the homogenization process leads to a limiting Hamilton--Jacobi equation posed in a Euclidean space of dimension $b$, which is determined by the topological complexity of the network, particularly the rank of its maximal group of symmetries. This result is noteworthy because it demonstrates that the dimension of the ambient space in which the network is embedded does not affect the homogenization outcome. We also note that $b$ can be characterized in terms of the number of edges and vertices of what we call the {\it base graph}, namely a finite abstract graph underlying the structure of the periodic network playing the role of the fundamental cell of the periodicity. It can be shown (see Remark \ref{Euler}) that
        \begin{equation*}
            b = \frac12 \text{(numer or edges)} - \text{(number of vertices) + 1}.
        \end{equation*}
        Intuitively, the topology becomes more complicated (and hence $b$ increases) if the number of vertices decreases and the number of connections between them increases, whereas it becomes simpler if the converse happens.

        \smallskip

        \item{\it Convergence:} The fact that time-oscillating problems and the limit one are posed on different spaces requires to introduce a suitable the notion of convergence. In this regard, we found beneficial the concept of {\it asymptotic cone} in the Gromov--Hausdorff sense (see Section \ref{conogelato}), which allows us not only to relate the rescaled networks (as metric spaces) to the limit space, but also to deal with the convergence of functions defined on these different spaces (see Definitions \ref{semi} and \ref{semibis}). The main idea is as follows: when we rescale the metric of the periodic network and let the scale parameter tend to $0$, we obtain the so-called asymptotic cone of the network, which is the metric space that best approximates, in the limit, the network in the Gromov–Hausdorff sense. Visually, instead of reducing the size of the network, we simply `observe' it from increasing and increasing distance.
    \end{itemize}

    \medskip

    \subsection{Main novelties and strategy}

    There are two possible approaches to the homogenization problem. In \cite{LPV}, the homogenization problem is addressed using purely PDE techniques. In this approach, the Hamiltonian is assumed to be merely continuous in both arguments and coercive in the variable $p$, without any convexity assumptions on the momentum variable. However, under the assumption that the Hamiltonian is convex and superlinear in the second argument, a variational approach is also possible. This approach relies on the fact that, under these conditions, a corresponding Lagrangian can be defined. The variational approach is based on the Lax–Oleinik formula, which expresses the solution of a time-dependent Hamilton–Jacobi equation with an initial condition in terms of the initial datum and the minimal action functional.

    \medskip

    Our method is rooted in this latter variational approach, drawing inspiration from \cite{ConItSic} and using Lax–Oleinik representations, but adapted to the specific framework of networks. In particular, two keys ideas in our strategy consists in:

    \begin{itemize}
        \item{\it Discretizing} the problem, as already successfully implemented in \cite{PozSic, Sic, SicSor1, SicSor}. We associate to the network underlying graphs, encoding all of the information on the complexity of the network and its periodicity, and we transfer the Hamiltonians and the Lagrangians to discrete objects on the graphs, extrapolating all the information that are relevant to our asymptotic analysis.

        \item Using {\it homological techniques}, similar to those used in \cite{ConItSic}, where the homogenization result is developed on arbitrary compact manifolds by lifting the Hamiltonian to a suitable covering manifold that introduces generalized periodicity. We adapt these ideas to the framework of networks, which presents non-trivial and novel challenges.
    \end{itemize}

    Among the most significant results that are auxiliary to the proof of the main homogenization theorem, we highlight:

    \smallskip

    \begin{itemize}
        \item \textit{Mather's asymptotic result for networks and graphs}: We leverage results from Aubry--Mather and weak KAM theory on graphs \cite{SicSor1, SicSor}, to extend Mather asymptotic result \cite{Mather91}, with suitable adaptations, to the framework of networks and graphs, see Theorems \ref{new} and \ref{melania}. More specifically, we prove the uniform convergence, as time goes to infinity, of the average minimal action on the network/graph to the value function of a relaxed variational problem defined on a suitable space of measures. As a matter of fact, the asymptotic, as $\eps$ goes to $0$, of the solutions to the approximated problems corresponds, through Lax--Oleinik formula, to the behaviour of the average action functional for large times. In this way, we also obtain a characterization of the effective Hamiltonian appearing in the limit equation in terms of stationary cell problems on the base graph, which generalizes what happens in the case of the torus or arbitrary compact manifolds.

        \smallskip

        \item \textit{Lax--Oleinik representation formula}: We prove the validity of this formula in the context of periodic networks, generalizing previous results for junctions \cite{ImbMonZid} and finite networks \cite{PozSic}. This is part of an extensive analysis of time-dependent Hamilton--Jacobi equation on periodic networks, carried out in Section \ref{equation}, which also includes a novel comparison theorem (Theorem \ref{compa}), requiring nontrivial modifications compared to the finite network case treated in \cite{Sic}.
    \end{itemize}

    \medskip

    \begin{Remark}\label{remarkintroflux}
        Finally, we note that the solutions of time-dependent Hamilton--Jacobi equations on networks are influenced not only by the initial data but also by the choice of a {\it flux limiter}, following the theory initiated in \cite{ImbMon}. Throughout this paper, we fix the flux limiter to be the maximal one. However, we acknowledge that if the flux limiter is varied, the homogenization result still holds, but with distinct features, both in terms of the limit space and the limit equation (see also Remark \ref{fluxlimiter}). Given the length of this paper, we have chosen not to address this aspect here, leaving it for future investigation.
    \end{Remark}

    \medskip

    \subsection{Comparison with previous literature}

    Since the seminal work by Lions, Papanicolaou, and Varadhan \cite{LPV}, an extensive body of literature on the homogenization of Hamilton–Jacobi equations has emerged, with significant applications to models characterized by the coexistence of phenomena at different scales and with varying levels of complexity.

    In our framework, certain recent papers on the so-called {\it specified homogenization} in the context of junctions can be considered related, see \cite{ForcadelSalazar20, ForcadelSalazarZaydan18, GaliseImbertMonneau15}. These works deal with graphs that possess a single vertex from which branches of infinite length emanate. The primary aim of these studies is to transition from microscopic to macroscopic scales in traffic analysis when bifurcations or local perturbations occur. Despite this broad similarity, the mathematical models in these works differ substantially from ours.

    Regarding homogenization results in non-standard settings, it is also worth mentioning \cite{AchdouLeBris23}, which investigates the problem for a periodic Hamiltonian perturbed near the origin, and \cite{BarlesBrianiChasseigneTchou15}, which addresses homogenization in a control model with distinct dynamics and costs in complementary domains of Euclidean space.

    As for the full homogenization, so to say, of Hamilton--Jacobi equations on networks, to the best of our knowledge, previous contributions can only be found in the comprehensive work \cite{ImbMon} and in \cite{CamMar}, where in particular the rate of convergence of the approximating equations is studied in the presence of a second-order viscosity term. However, the results in these papers are limited to networks with base graphs of {\em bouquet type}, {\it i.e.}, with a single vertex and several edges that start and end at this vertex (see for example figure \ref{subfig:2bouquet}); in this case $b$ corresponds to the number of self-loops. In contrast, our analysis allows for arbitrary geometries of the periodic network.

    \medskip

    \subsection{Organization of the article}

    To help the reader navigate this lengthy text, we would like to provide a map of its organization. It is essentially structured into four main segments, alongside the introduction and the appendix.

    The first one corresponds to Section \ref{periodic}, where we specify the setting and state the Homogenization result (Main Theorem).

    The second part coincides with Section \ref{equation} and contains a PDE analysis of time-dependent Hamilton--Jacobi equations posed on the periodic networks. The novelty of this part, compared to the existing literature, lies in the fact that the problem is posed on a periodic network with infinitely many vertices and arcs. We remark that our approach is generally applicable to any infinite network, not necessarily periodic, satisfying the conditions listed in \textbf{(P1)}--\textbf{(P7)} (see Subsection \ref{subsecexistenceuni}).

    The third part is the most conspicuous one and consists of Sections \ref{construens} to \ref{secrelaxed}. Inspired by \cite{PozSic, Sic, SicSor1, SicSor}, one of the main ideas behind our approach consists in discretizing the problem. Namely, we associate to the network an underlying abstract graph, encoding all of the information on the complexity of the network and its periodicity, and we transfer the Hamiltonians and the Lagrangians to discrete objects on the graphs.

    In the fourth and final part, Section \ref{epi}, the output of the previous sections are exploited to carry out the asymptotic analysis and provide a proof of the Main Theorem.

    To avoid weighing down the presentation, we have decided to postpone some of the more technical and lengthy proofs to Section \ref{proofs}.

    Finally, to corroborate the generality of our setting, in Appendix \ref{appexnet} we will discuss how to explicitly construct a periodic network starting from any given finite graph as base graph. We believe this construction could be useful for computational or numerical purposes, see \cite{CCPS}.

    \bigskip

    \subsection*{Acknowledgements}

    Antonio Siconolfi and Alfonso Sorrentino would like to thank the Mathematical Sciences Research Institute (MSRI) in Berkeley for its hospitality during the Fall 2018 semester, as part of the program: {\it Hamiltonian systems, from topology to applications through analysis}. The stimulating environment and support provided by MSRI greatly contributed to the development of this work.\\
    Alfonso Sorrentino and Marco Pozza acknowledge the support of the Italian Ministry of University and Research’s PRIN 2022 grant ``{\it Stability in Hamiltonian dynamics and beyond}''. Antonio Siconolfi acknowledges the support of KAUST through the subaward agreement ORA-2021-CRG10-4674.6. Alfonso Sorrentino also acknowledges the support of the Department of Excellence grant MatMod@TOV (2023-27) awarded to the Department of Mathematics of University of Rome Tor Vergata. \\
    The authors are members of the INdAM research group GNAMPA.

    \newpage

    \addtocontents{toc}{\SkipTocEntry}

    \tableofcontents

    \newpage

    \section*{Some notations and terminology}

    In this section, we gather some of the notations, terminology and symbols that will be used throughout the paper to assist the reader in navigating the article.
    \begin{itemize}
        \item By \emph{curve} we mean an \emph{absolutely continuous} curve with support contained in a Euclidean space.

        \item Given an open set $\mathcal O$ and a continuous function $u:\overline{\mathcal O}\to\R$, we call \emph{supertangents} (resp.\ \emph{subtangents}) to $u$ at $x\in\mathcal O$ the viscosity test functions from above (resp.\ below). If needed, we take, without explicitly mentioning, $u$ and test function coinciding at $x$ and test function strictly greater (resp.\ less) than $u$ in a punctured neighbourhood of $x$. \\
        We say that a subtangent $\varphi$ to $u$ at $x\in\partial\mathcal O$ is \emph{constrained to $\overline{\mathcal O}$} if $x$ is a minimizer of $u-\varphi$ in a neighbourhood of $x$ intersected with $\overline{\mathcal O}$.

        \item List of symbols:
    \end{itemize}

    \begin{longtable}{rl}
        $|\xi|$ & Euclidean length of a curve $\xi$\\
        $\G$ & a maximal topological crystal\\
        $\EE$ & set of the edges of $\G$\\
        $\VV$ & set of the vertices of $\G$\\
        $\oo(e)$ & initial vertex of the edge $e$\\
        $\tt(e)$ & terminal vertex of the edge $e$\\
        $\oo(\xi)$ & initial vertex of the path $\xi$\\
        $\tt(\xi)$ & terminal vertex of the path $\xi$\\
        $\EE_x$ & set of the $e\in\EE$ with $\oo(e)=x$\\
        $d_\VV$ & metric on $\VV$ associating to any pair of vertices the minimal length\\
        & of a path linking them\\
        $\G_0$ & quotient graph of $\G$\\
        $\EE_0$ & set of the edges of $\G_0$\\
        $\VV_0$ & set of the vertices of $\G_0$\\
        $\EE_0^+$ & an orientation of $\G_0$\\
        $\Cf_0(\G_0,\R)$ & the vector space over $\R$ spanned by $\VV_0$\\
        $\Cf_1(\G_0,\R)$ & the vector space over $\R$ spanned by $\EE_0$\\
        $H_1(\G_0,\R)$ & the first homological group of $\G_0$ with real coefficients\\
        $H^1(\G_0,\R)$ & the first cohomological group of $\G_0$ with real coefficients\\
        $T^+\G_0$ & tangent cone of $\G_0$\\
        $\TT$ & a spanning tree in $\G_0$\\
        $\EE_\TT$ & set of the edges of $\TT$\\
        $\VV_\TT$ & set of the vertices of $\TT$\\
        $\NN$ & periodic network over the base graph $\G_0$\\
        $\EN$ & set of the arcs of $\NN$\\
        $\VN$ & set of the vertices of $\NN$\\
        $\EN^z$ & set of the $\ga\in\EN$ with $\ga(1)=z$\\
        $d_\NN$ & distance on $\NN$ induced by the Euclidean one in $\R^N$\\
        $d_\VN$ & lift of $d_\VV$ from $\VV$ to $\VN$\\
        $\pi_1(z)$ & projection of $z\in\VN$ into $\VV_0$\\
        $\pi_2(z)$ & projection of $z\in\VN$ into $\Z^{b(\G_0)}$\\
        $T\NN$ & tangent bundle of $\NN$\\
        $-e$ & reversed edge of the edge $e$\\
        $\wtd\ga$ & reversed arc of the arc $\gamma$\\
        $b(\G_0)$ & first Betti number of $\G_0$\\
        $\#_\xi(e)$ & multiplicity of an edge $e$ in a path $\xi$\\
        $[\xi]$ & incidence vector of a path $\xi$\\
        $\theta([\xi])$ & rotation vector of a path $\xi$\\
        $\spt\xi$ & set of the edges $e$ in a path $\xi$\\
        $H$ & collection of Hamiltonians $H_\gamma$ indexed by the $\gamma\in\EN$\\
        $L_\ga$ & the Lagrangian obtained from $H_\ga$ via Fenchel--Legendre duality\\
        $\wha c_z$ & maximal flux limiter\\
        $L$ & Lagrangian defined on $T\NN$ suitably gluing together the $L_\ga$'s\\
        $A_L(\xi)$ & action of the curve $\xi$ on $\NN$\\
        $\Phi_L(z_1,z_2,T)$ & minimal action functional between $z_1,z_2\in\NN$ in a given time $T$\\
        $\si^+_a(e,s)$ & $\max\{\rho \mid H_e(s,\rho)=a\}$\\
        $\si(e,a)$ & $\int_0^1\si^+_{a}(e,s) \, ds$\\
        $\HH(e,\cdot)$ & inverse with respect to the composition of $a \mapsto \si(e,a)$\\
        $\LL$ & Fenchel transform of $\HH$\\
        $A_\LL(\xi)$ & action of the path $\xi$ on $\G_0$\\
        $\Phi_\LL(x,y,T)$ & minimal action functional between $zx,y\in\G_0$ in a given time $T$\\
        $\mathbb P$ & family of Borel probability measures on $T^+ \G_0$\\
        $\supp_{\EE_0}\mu$ & set of the $e\in\EE_0$ such that $\mu_e(\R^+)>0$\\
        $\delta(e,q)$ & Dirac delta concentrated on the point $(e,q)\in T^+\G_0$\\
        $\bM$ & set of the closed measures in $\mathbb P$\\
        $\rho(\mu)$ & rotation vector of $\mu\in\bM$\\
        $A(\mu)$ & action of $\mu\in\bM$\\
        $\al$ & Mather's $\al$--function\\
        $\be$ & Mather's $\be$--function\\
        $\bM_p$ & set of the minimizer of $\al(p)$\\
        $\bM^h$ & set of the minimizer of $\be(h)$\\
        $\bar H$ & effective Hamiltonian
    \end{longtable}

    \newpage

    \section{Periodic networks and main theorem}\label{periodic}

    Before stating our main result, let us describe the setting, namely the concepts of periodic network, underlying abstract graph (which is auxiliary to definition of periodic networks and the discretization of the problem), and Hamiltonian on the network.

    \medskip

    \subsection{Networks}\label{sectionnetwork}

    We recall that by {\it network} we mean a subset $\wha\NN \subset \R^N$, of the form
    \begin{equation*}
        \wha\NN := \bigcup_{\gamma \in \wha\EN} \, \gamma([0,1]) \subset \R^N,
    \end{equation*}
    where $\wha\EN$ is a collection of simple $C^1$ regular ({\it i.e.}, with non-vanishing derivative) oriented curves $\gamma: [0,1] \to \R^N$, called {\it arcs} of the network.

    For each arc $\gamma\in \wha\EN$, we denote by $\wtd{\gamma}$ the same arc with opposite orientation ({\it reversed arc}) ({\it i.e.}, $\wtd{\gamma}(s)=\gamma(1-s)$ for every $s\in [0,1]$), which we assume to also be in $\wha\EN$. Observe that the map
    \begin{eqnarray*}
        \wtd{\phantom{o}}: \wha\EN &\longrightarrow& \wha\EN \\
        \gamma &\longmapsto& \wtd \gamma,
    \end{eqnarray*}
    is a fixed-point-free involution, namely
    \begin{equation*}
        \wtd \gamma \ne \gamma \qquad\text{and} \qquad \wtd{\wtd \gamma}= \gamma \qquad\text{for any $\gamma \in \EN$.}
    \end{equation*}

    The initial and final point of any arc $\ga$, namely $\ga(0)$ and $\ga(1)$, have a special status: they are called {\it vertices of the network}. The set of vertices is denoted by $\wha\VN$. If we look at a network as a piecewise regular manifold, vertices are the points where regularity fails.

    We assume the following key disjointness condition: arcs with different support can intersect only at the vertices, namely
    \begin{equation*}
        \ga((0,1)) \cap \ga'((0,1)) = \emptyset \txt{provided that $\ga' \ne \ga, \, \wtd \ga$.}
    \end{equation*}
    Given a vertex $z$ in $\wha\NN$, we set
    \begin{equation*}
        \wha\EN^z:= \{ \ga \in \wha\EN \mid \ga(1)=z\}
    \end{equation*}
    We say that a network $\wha\NN$ is {\it finite/infinite} if the set of edges is finite/infinite, we say that it is {\it locally finite} if $\wha\EN^z$ is finite for any $z \in \wha\VN$.

    We say that a network is {\em connected} if any pair of points can be linked by a curve contained in $\wha\NN$, namely an absolutely continuous curve of $\R^N$ with support contained in $\wha\NN$. {\it All the networks we consider throughout the paper are understood to be connected}.

    We consider on the network the distance $d_{\wha\NN}$ induced by the Euclidean one in $\R^N$, namely we take as distance of two points in $\wha\NN$ the minimal Euclidean length of curves contained in $\wha\NN$ linking them.

    \medskip

    \begin{Remark}
        Our results can be extended to the case in which $\wha\NN$ is embedded in a Riemannian manifold, for example by means of Nash embedding theorem \cite{Nash}. Moreover, the results are independent of the chosen parameterizations of the arcs. In this regard, one could also choose a more intrinsic approach and consider arcs as $1$-dimensional submanifolds, and the whole network as a stratified space. Hereafter, we preferred not to adopt this point of view to simplify the presentation.
    \end{Remark}

    \medskip

    \subsection{Abstract graphs and topological crystals}

    A {\it graph} $\wha\G=(\wha\VV ,\wha\EE) $ is an ordered pair of disjoint sets $\wha\VV$ and $\wha\EE$, with the latter possibly empty, which are called, respectively, {\it vertices} and {\it edges}, plus two functions:
    \begin{equation*}
        \oo: \wha\EE \longrightarrow \wha \VV
    \end{equation*}
    and
    \begin{eqnarray*}
        -: \wha\EE &\longrightarrow&\wha \EE \\
        e &\longmapsto& - e,
    \end{eqnarray*}
    with the latter assumed to be a fixed-point-free involution, namely satisfying
    \begin{equation*}
        -e \ne e \qquad\text{and} \qquad -(- e)= e \qquad\text{for any $e \in \wha\EE$.}
    \end{equation*}
    We further set
    \begin{equation*}
        \tt(e)= \oo(-e) \qquad\text{for any $e \in \wha\EE$.}
    \end{equation*}
    We give the following geometric picture of the setting: $\oo(e)$, $\tt(e)$ are the {\it initial} and {\it terminal} vertex of $e$, respectively, and $-e$ the {\it reversed} edge, namely the same edge with the opposite orientation, whose initial vertex is clearly the terminal vertices of $e$.

    \smallskip

    We call {\it path} any finite sequence $e_1, \cdots, e_m$ of {\it concatenated} edges, namely satisfying
    \begin{equation*}
        \oo(e_i)= \tt(e_{i-1}) \txt{for $i= 2, \cdots, m$.}
    \end{equation*}
    Given a path $\xi:=(e_1, \cdots, e_m)$, we set $\oo (\xi):= \oo (e_1)$, $\tt (\xi):= \tt (e_m)$. Given two paths $\xi$, $\eta$ with $\tt(\xi)=\oo(\eta)$ we denote by $\xi \cup \eta$ the path obtained by concatenating them.

    \begin{Definition}\label{loop}
        An edge $e$ (resp.\ an arc $\ga$) in a graph (resp.\ network) is called {\it self-loop} if
        \begin{equation*}
            \oo(e)=\tt(e) \txt{(resp $\ga(0)=\ga(1)$).}
        \end{equation*}
    \end{Definition}

    \smallskip

    We call {\it length} of a path the number of edges in it. We say that a graph is {\it connected} if any pair of vertices is linked by some path. {\it All the graphs we consider throughout the paper are understood to be connected}. We can therefore define a metric $d_{ \wha \VV}$ on $\wha\VV$ associating to any pair of vertices the minimal length of a path linking them.

    For any vertex $x\in\wha\VV$, we define
    \begin{equation*}
        \wha \EE_x :=\{ e\in \wha\EE|\; \oo(e)=x\}
    \end{equation*}
    and call it the {\it star centered at $x$}.

    We say that a graph is {\it finite/infinite} if it possesses finitely/infinitely many edges. If it is finite we denote by $|\wha\EE|$, $|\wha\VV|$ the number of its edges and vertices respectively. We say that a graph is {\it locally finite} if every star has a finite number of elements.

    \smallskip

    \smallskip

    Let us now introduce some further notions that are relevant to our scope of defining the notion of periodicity.

    \begin{Definition}\label{morph}
        A {\it morphism} between two graphs $\wha\G=(\wha\VV,\wha\EE)$, $\G'=(\VV',\EE')$ is a pair of maps $(\FN_{\VV},\FN_{\EE})$ where
        \begin{eqnarray*}
            \FN_{\VV}: \wha\VV \to\VV' \qquad{\rm and} \qquad \FN_\EE: \wha\EE \to \EE'
        \end{eqnarray*}
        satisfying the compatibility/adjacency relations:
        \begin{eqnarray*}
            \oo(\FN_\EE(e)) = \FN_\VV (\oo(e)) \quad{\rm and} \quad \FN_\EE(-e ) = -\FN_\EE(e) \txt{for any $e \in \wha \EE$}.
        \end{eqnarray*}
        If both $\F_\EE$, $\F_\VV$ are bijections then the morphism is called {\it isomorphism} and if, in addition, $\wha\G=\G'$ {\it automorphism} or {\it symmetry}. We denote by ${\mathrm{Aut}}(\wha\G)$ the set of automorphisms of $\wha\Gamma$, which is clearly a group under the composition.
    \end{Definition}

    \medskip

    \begin{Definition}
        Given a subgroup ${G} \le{\rm Aut}(\wha\G)$, we say that $G$ acts {\it freely} on $\wha\G$ if for any $(\FN_\VV,\FN_\EE) \in G$:
        \begin{itemize}
            \item[{\bf (i)}] $\FN_\EE (e) \ne -{e}$ for every $e \in \wha\EE$ ({\it i.e.}, it acts without inversion),
            \item[{\bf (ii)}] if $\FN_\VV(x)=x$ for some $x\in \wha\VV$, then $(\FN_\VV,\FN_\EE)$ is the identity (both on vertices and edges). Clearly, this also implies that if $(\FN_\VV,\FN_\EE) \in G$ is not the identity, then $\FN_\EE(e)\ne e $ for every $e\in \wha\EE$.
        \end{itemize}
    \end{Definition}

    If $G \le{\rm Aut}(\G)$ acts {\it freely} on $\G$, then the quotient of $\G$ with respect to the action of $G$ has a natural structure of graph, that we call {\it quotient graph} and denote by $ \G_0 = ( \VV_0, \EE_0)$, see \cite[Theorem 3.1]{Sunada}.

    \begin{Remark}
        To retrieve analogies with the Euclidean setting, we recall that the $\Z^N$--translation make up a group of automorphisms of $\R^N$, and the action of $(\Z^N,+)$ on $\R^N$ is free according to the above definition, up to some trivial adaptations.
    \end{Remark}

    In the next definition we prescribe further conditions on the group $G$ freely acting on $\G$.

    \medskip

    \begin{Definition}\label{deftopcrystruct}
        A graph $\G= (\VV, \EE)$ is called a {\em (maximal) topological crystal over $\Gamma_0$} if it admits a group $G \le{\mathrm{Aut}}(\G)$ such that:
        \begin{itemize}
            \item[{\bf (i)}] $G$ acts freely on $\G$,
            \item[{\bf (ii)}] the quotient graph $ \G_0 = ( \VV_0, \ \EE_0)$ is a finite graph, and $G \simeq{(\Z^{b(\G_0)},+)}$, where
            \begin{equation*}
                b(\G_0)= \frac{1}{2}| \EE_0| - | \VV_0| + 1.
            \end{equation*}
        \end{itemize}
    \end{Definition}

    \begin{Remark}
        As observed by crystallographers, the associated with a crystals is not just an infinite graph realized in space, but a graph with an acting group of symmetries, which becomes a finite graph when factored out. In crystallography the group of symmetries leaving the crystal invariant (the group $G$ in Definition \ref{deftopcrystruct}) is called {\it Bravais lattice}.
    \end{Remark}

    \smallskip

    \begin{Remark}
        From a topological point of views, $\G$ is an Abelian covering graph of $\G_0$ (see \cite[Theorem 5.2]{Sunada}). In particular condition {\bf(ii)} ensures that $\G$ is the {\it maximal} Abelian covering of $ \G_0$ (roughly speaking, one cannot obtain $\G_0$ as quotient of a graph with a topological crystal structure given by a group of rank larger than $b(\G_0)$). In fact, $b(\G_0)$ is the {\it first Betti number} of $\G_0$, namely it coincides with the rank of its first homology group, (see {Corollary \ref{incide}}), and this guarantees the maximality (see \cite[Theorem 6.1]{Sunada}). See also Remark \ref{remaftermainthm} {\bf (iii)}.\\
        In particular, it follows that the maximal topological crystal over a given finite graph is unique, up to isomorphism.
    \end{Remark}

    \medskip

    \subsection{Underlying abstract graph of a network and periodic networks}\label{secunderlyinggraph}

    Given an infinite network $\wha\NN=(\wha\VN, \wha\EN)$ with a set of edges $\EN$ and a set of vertices $\VN$, we can associate with it an abstract graph $\wha\G= (\wha\VV, \wha\EE)$ in the following way. In other words, we view each edge as a single element of an abstract set rather than as an arc in $\R^N$.

    \begin{Definition}\label{sottostante}
        Given any network $\wha\NN \subset \R^N$, an {\it underlying abstract graph} $\wha\Gamma= (\wha\VV, \wha\EE)$ is defined by considering two bijections
        \begin{equation*}
            \Psi_{\wha\VN}: \wha\VN \to \wha\VV \qquad \Psi_{\wha\EN}: \wha\EN \to \wha\EE,
        \end{equation*}
        where $\wha\VV$ and $ \wha\EE$ are abstract sets playing the role of vertices and edges, respectively, of $\wha\G$. Maps $\oo: \wha\EE \to \wha\VV $ and $- : \wha\EE \to \wha\EE$ are defined via
        \begin{equation*}
            \oo(\Psi_{ \wha\EN}(\ga))= \Psi_{\wha\VN}(\ga(0)) \txt{and} \qquad -\Psi_{\wha\EN}(\ga) = \Psi_{\wha\EN}(\wtd \ga) \txt{for any $\ga \in \wha\EN$.}
        \end{equation*}
        The above rules actually give to $\wha\G=(\wha\VV,\wha\EE)$ the structure of graph.
    \end{Definition}

    \medskip

    \begin{Remark}\label{remarkunicityuptoiso}
        \hfill

        \noindent{\bf (i)} An underlying graph inherits from the associated network the properties of being connected, finite, infinite or locally finite. Conversely, properties of the underlying graph translates into analogue properties of the above network. For example, self-loops of the underlying graph (see Definition \ref{loop}) are in correspondence with arcs $\gamma$ in $\wha\NN$ such that $\gamma(0)=\gamma(1)$.\\
        {\bf (ii)} We denote by $d_{\VN}$ the lift of $d_{\wha\VV}$ from $\wha\VV$ to $\wha\VN$. \\
        {\bf (iii)} It is easy to prove that the underlying graph associated to a given network is unique, up to isomorphisms.
    \end{Remark}

    \medskip

    We can now introduce the notion of {\it periodic network}.

    \begin{Definition}\label{vetro}
        Given a finite graph $\G_0 =(\VV_0, \EE_0)$, we call {\em periodic network} over the {\em base graph} $\G_0$, a network $ \NN= ( \EN, \VN)$, embedded in some Euclidean space, such that

        \begin{itemize}
            \item[{\bf (i)}] The underlying graph $\G$ is a maximal topological crystal with quotient graph $\G_0$;
            \item[{\bf (ii)}] The Euclidean lengths of the arcs of $\NN$ are bounded from above and possess a positive infimum.
        \end{itemize}
    \end{Definition}

    \smallskip

    We remark that one could give a more general notion of periodicity without requiring the underlying graph to be a maximal topological crystal.

    \medskip

    One directly derive from the above condition {\bf (ii)} the following

    \begin{Corollary}
        The two metrics $d_\VN$ and $d_\NN$, restricted to $\VN$, are equivalent.
    \end{Corollary}

    \smallskip

    \begin{Notation}
        From now on, we indicate by $\NN =(\VN,\EN)$ an arbitrary periodic network and by $\G=(\VV,\EE)$, $G$, $\G_0=(\VV_0,\EE_0)$ its underlying graph, Bravais lattice and base graph, respectively.
    \end{Notation}

    \smallskip

    \begin{Remark}
        According to \cite[Theorem 5.2]{Sunada}, the canonical projection from $ \G$ to $\G_0$ is a graph morphism, which we denote by $(\FN^0_\VV,\FN^0_\EE)$, satisfying
        \begin{itemize}
            \item[--] $\FN^0_\VV: \VV \to \VV_0$ is surjective;
            \item[--] $\FN^0_\EE$ is a bijection from $\EE_z$ to ${(\EE_0)}_{\FN^0_\VV(z)}$ for any vertex $z$ in $\VV$.
        \end{itemize}
        Any maximal topological crystal is consequently {\it locally finite}, and the same property consequently holds for periodic networks, see Definition \ref{sottostante}.
    \end{Remark}

    \medskip

    \medskip

    \begin{Example}\label{bouquetexe}
        Homogenization in the previous literature, \cite{ImbMon,CamMar}, is limited to periodic networks generated by $\Z^N$ with straight arcs. In our setting the base graph of such periodic networks is the so called {\it $N$--bouquet}, namely a graph made up of a single vertex and $N$ self-loops (see Definition \ref{loop}). In figure \ref{fig:bouquetcov} we see a periodic network generated by $\Z^2$, with straight arcs, and its base graph ({\it i.e.}, a $2$-bouquet).
    \end{Example}

    \begin{figure}[ht]
        \begin{subcaptiongroup}
            \phantomcaption\label{subfig:2bouquetcry}
            \raisebox{-114.5946pt}{\includegraphics{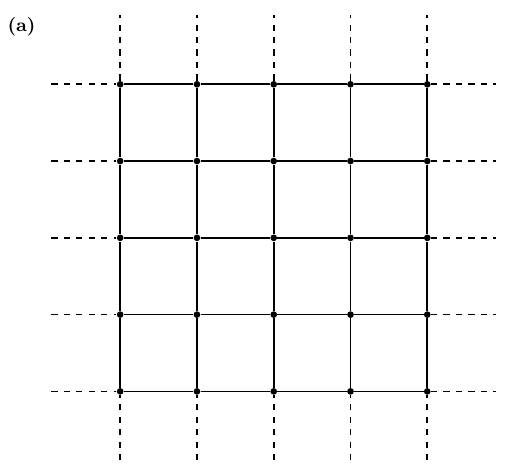}}
            \hspace{7mm}
            \phantomcaption\label{subfig:2bouquet}
            \raisebox{-44.01863pt}{\includegraphics{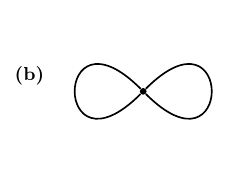}}
        \end{subcaptiongroup}
        \caption{A periodic network \subref{subfig:2bouquetcry} over the 2--bouquet \subref{subfig:2bouquet}.}\label{fig:bouquetcov}
    \end{figure}

    \begin{Example}\label{graphenecryexe}
        We present in figure \ref{fig:graphenecry} another example of periodic network. This {\it honeycomb} network appears as the crystal structure of the graphene and is well known in crystallography.
    \end{Example}

    \begin{figure}
        \begin{subcaptiongroup}
            \phantomcaption\label{subfig:graphenecry}
            \raisebox{-83.2962pt}{\includegraphics{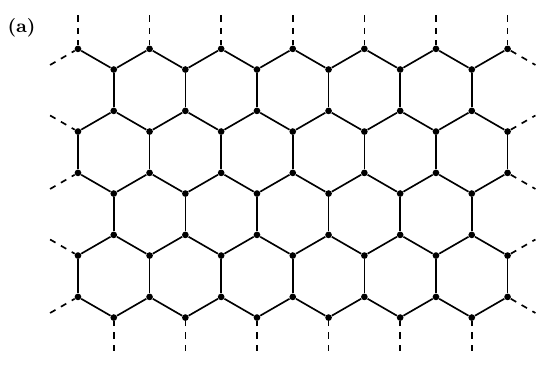}}
            \hspace{7mm}
            \phantomcaption\label{subfig:graphenebasecry}
            \raisebox{-38.00491pt}{\includegraphics{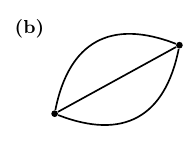}}
        \end{subcaptiongroup}
        \caption{The honeycomb network in \subref{subfig:graphenecry} is a periodic graph over the base graph \subref{subfig:graphenebasecry}.}\label{fig:graphenecry}
    \end{figure}

    \medskip

    For more examples of periodic networks, see Appendix \ref{appexnet}, where we describe how to construct periodic networks starting from any finite graph as base graph.

    \medskip

    \subsection{Hamiltonians on networks and statement of the Main Theorem}\label{secHam}

    \begin{Definition}
        A {\it Hamiltonian} on $\NN$ is a collection of functions $H:=\{H_\gamma\}_{\gamma\in\EN}$, where

        \begin{eqnarray*}
            H_\ga: [0,1] \times \R &\longrightarrow& \R\\
            (s,\rho) &\longmapsto& H_\ga(s,\rho)
        \end{eqnarray*}
        and such that the following compatibility condition is satisfied:
        \begin{equation}\label{ovgamma}
            H_{\wtd\ga}(s,\rho) = H_{\ga}(1-s,-\rho) \qquad\text{for any $\ga \in \EN$}.
        \end{equation}
    \end{Definition}

    \begin{Remark}
        We stress that, apart from condition \eqref{ovgamma}, the Hamiltonians $H_\ga$ on each arc are totally unrelated.
    \end{Remark}

    In the following, we further require our Hamiltonian $H=\{H_\gamma\}_{\gamma\in\EN}$ to satisfy some extra conditions. Namely, we require any $H_\ga$ to be:
    \begin{itemize}
        \item[{\bf (H$_\ga$1)}] continuous in $s$ and continuously differentiable in $\rho$;
        \item[{\bf (H$_\ga$2)}] {strictly} convex in $\rho$;
        \item[{\bf (H$_\ga$3)}] superlinear in $\rho$, {\it i.e.}, $\lim_{\rho\rightarrow +\infty} \frac{H_\gamma(s,\rho)}{\rho}=+\infty$, uniformly in $s\in [0,1]$.
    \end{itemize}

    We further require $H$ to be {\it $G$--periodic}, namely

    \begin{itemize}
        \item[{\bf (H$_\ga$4)}] for every $(\FN_\VV, \FN_\EE) \in G$ and every $\gamma\in \EN$
        \begin{equation*}
            H_{(\Psi^{-1}_\EN \circ \FN_\EE \circ \Psi_\EN)(\gamma)}=H_\gamma,
        \end{equation*}
        where $\Psi_\EN: \EN \to \EE$ is one of the bijection relating $\NN$ to its underlying graph $\G$, see Definition \ref{sottostante}.
    \end{itemize}

    \bigskip

    We consider the family of equations
    \begin{equation*}
        \tag{HJ$_\ga$} u_t + H_\ga(s,u')=0 \qquad\text{on $ (0,1) \times (0,+\infty)$,}
    \end{equation*}
    for $\ga$ varying in $\EN$, with a uniformly continuous initial datum $g: \NN \to \R$ at $t=0$

    We synthetically denote the problem we are interested on by
    \begin{equation*}\tag{HJ}
        \left\{
        \begin{aligned}
            &\partial_tu(z,t)+H(x,\partial_zu(z,t))=0&&\text{on }\NN\times(0,\infty),\\
            &u(z,0)=g(z)&&\text{on }\NN.
        \end{aligned}
        \right.
    \end{equation*}
    {See Section \ref{equation} for more details and for the precise definition of solution on $\NN$, which, in particular, involves the introduction of the so-called {\it maximal flux limiter} $\wha c_z$, see Definition \ref{defsol}}.

    \bigskip

    We can now state our Main Theorem.

    \begin{MainThm}
        {Let $\NN =(\VN,\EN)$ a periodic network and let $\G=(\VV,\EE)$ denote its underlying abstract graph, with base graph $\Gamma_0=(\VV_0,\EE_0)$ and Bravais lattice $G\simeq \Z^{b(\G_0)}$}. Let ${H=}\{H_\gamma\}_{\gamma\in\EN}$ be a Hamiltonian on $\NN$, satisfying conditions {\bf(H$_\ga$1)--(H$_\ga$4)}. Let $u_\eps$ be the solutions to the rescaled time-dependent Hamilton--Jacobi equations:
        \begin{equation}\tag{HJ$^\eps$}\label{HJeps}
            \left\{
            \begin{array}{lll}
                \partial_t u_\eps(x,t) + H\left(x, \frac{1}{\eps} \partial_x u_\eps(x,t)\right) = 0 && x\in \NN,\; t>0\\
                u_\eps(x,0) = g_\e (x) && x\in \NN
            \end{array}
            \right.
        \end{equation}
        coupled with {maximal flux limiters} (see Definition \ref{defsol}) and initial data $g_\eps: \NN \longrightarrow \R$ at $t=0$ equi-uniformly continuous with respect to the metrics $\eps d_\NN$ (in the sense of Definition \ref{semiter}) and locally uniformly convergent (in the sense of Definition \ref{semi}), as $\eps$ goes to $0^+$, to a function $g: \R^{b(\G_0)} \longrightarrow \R$.\\
        \noindent Then, there exists a convex and superlinear effective Hamiltonian $\bar{H}: \R^{b(\G_0)} \rightarrow \R$, such that the family $ u_\eps$ locally uniformly converge (in the sense of Definition \ref{semibis}) to the unique viscosity solution $u: \R^{b(\G_0)} \times (0,+\infty) \longrightarrow \R$ to
        \begin{equation}\tag{HJ$^{\rm lim}$}\label{HJlim}
            \left\{
            \begin{array}{lll}
                \partial_t u ( h,t) + \bar{H}(\partial _{h} u(h,t)) = 0 && (h,t) \in \R^{b(\G_0)} \times (0,\infty)\\
                u(h,0) = g(h) && h\in \R^{b(\G_0)}.
            \end{array}
            \right.
        \end{equation}
    \end{MainThm}

    \bigskip

    \begin{Remark}
        Since the networks we consider are embedded in Euclidean spaces, the rescaling from $(x,p)$ to $(x/\eps,p)$ still makes sense and if $\NN$ is a network, the same holds true for $\eps \, \NN$. This approach to the homogenization problem entails that the approximated problems are defined on networks varying with $\eps$.

        Instead, we prefer to use the rescaling from $(x,p)$ to $(x,p/\eps)$ because it is more intrinsic and simpler to handle. It corresponds to rescaling by a factor $\eps$ the metric keeping the network fixed, which only affects the momenta.

        The two rescaling are actually equivalent in our case as well as in the Euclidean one: if $u_\eps(x,t)$, $v_\eps(x,t)$ are the solutions of the $\eps$--problems corresponding to the two above rescaling, respectively, then it is easy to check that the relation $v_\eps(x,t) = u_\eps(\eps x,t)$ holds true.
    \end{Remark}

    \bigskip

    \begin{Remark}\label{remaftermainthm}
        \hfill

        \noindent{\bf (i)} The effective Hamiltonian $\bar{H}: \R^{b(\G_0)}\longrightarrow \R$ coincides with the so-called Mather's $\alpha$-function associated to the Hamiltonian on the base graph $\G_0$, which was constructed in \cite{SicSor1} (see Definitions \ref{defHbar} \& \ref{defalpha}, and Proposition \ref{effetto}). We mention that this function is defined on the first cohomology group of $\G_0$, which under our assumptions (see \textbf{(ii)} in Definition \ref{deftopcrystruct}) is isomorphic to $\R^{b(\G_0)}$.\\
        {\bf (ii)} This result could be actually generalized to the case in which the topological structure on $\NN$ is {\it not maximal}, namely condition \textbf{(ii)} in Definition \ref{deftopcrystruct} is not imposed. In this case, the effective Hamiltonian will coincide with the restriction of the Mather's $\alpha$-function to some subspace of the first cohomology group. In order to avoid extra technicalities to our presentation, we prefer to restrict ourselves to the maximal case, although we think that our techniques extend to this other setting.\\
        {\bf (iii)} A more subtle issue is how to extend this result to the case in which the acting group $G \le{\rm Aut}(\NN)$ is not necessarily Abelian. We remark that in this case, the existence of an asymptotic cone of $\NN$ depends on the algebraic nature $G$, in particular on its rate of growth, which must be at most polynomial. We will not discuss this setting here. For Hamilton--Jacobi equations on manifolds and acting groups $G$ that are (virtually) nilpotent, this problem has been investigated in \cite{SorrentinoHomogenization}.
    \end{Remark}

    \bigskip

    \section{Time-dependent Hamilton--Jacobi equations on periodic networks}\label{equation}

    In this section we would like to discuss the existence and the properties of the solutions to the {\it time-dependent Hamilton--Jacobi equation} on the network $\NN=(\VN, \EN)$, which are the main objects in the homogenization procedure:
    \begin{equation}\tag{HJ}\label{eq:HJ}
        \left\{
        \begin{array}{l}
            \partial_t u (x,t) + H\left(x, \partial_x u(x,t)\right) = 0 \qquad x\in \NN,\; t>0\\
            u(x,0)= g( x),
        \end{array}
        \right.
    \end{equation}
    under suitable assumptions on $H=\{H_\gamma\}_{\gamma\in \EN}$ and $g: \NN \rightarrow \R$.

    This problem, in the case of finite networks $\NN$, has been thoroughly investigated in \cite{PozSic}. Although those results cannot be directly applied to our setting, being $\NN$ not finite, yet the periodicity assumptions on $\NN$ and $\{H_\gamma\}_{\gamma\in \EN}$ will allow us to extend them to our case.

    Let us first describe the meaning of \eqref{eq:HJ} and what we mean by being a solution to this problem.

    \subsection{Notions of solution, subsolution and supersolution}\label{secsolutionHJ}

    For any vertex $z$ in $ \NN$, we define:
    \begin{equation}\label{maxxi}
        \wha c_z := \min_{\ga \in \EN^z} \left [ -\max_{s \in [0,1]} \min_\rho H_\ga(s,\rho) \right ].
    \end{equation}

    See Remark \ref{fluxlimiter} for more details on the meaning of these quantities.

    We can now introduce the concept of being a solution (resp.\ sub/supersolution) to \eqref{eq:HJ}.

    \begin{Definition}\label{defsol}
        By {\it solution} to \eqref{eq:HJ} we mean a uniformly continuous function $ u: \NN \times [0,+ \infty) \to \R$ such that:
        \begin{itemize}
            \item[{\bf (i)}] For every $\gamma\in \EN$, the restriction $u_\gamma(s,t):=u(\ga(s),t)$ is a viscosity solution to
            \begin{equation}\label{HJg}
                \tag{HJ$_\ga$} \left\{
                \begin{array}{l}
                    \partial_t u_\gamma (s,t) + H_\gamma\left(s, \partial_s u_\gamma(s,t)\right) = 0 \qquad (s,t)\in (0,1) \times (0,+ \infty)\\
                    u_\gamma(s,0)= g(\gamma(s)) \hspace{4 cm} s\in [0,1].
                \end{array}
                \right.
            \end{equation}
            \item[{\bf (ii)}] For every $z \in \VN$ and $t_0 \in (0,+\infty)$:
            \begin{itemize}
                \item[{\bf (a)}] If $\psi$ is a $C^1$ subtangent to $u(z, \cdot)$ at $t_0>0$, such that
                \begin{equation*}
                    \frac d{dt} \psi (t_0) < \wha c_z
                \end{equation*}
                then there is an arc $\ga \in \EN^z$ such that all the $C^1$ subtangents $\varphi$, constrained to $[0,1] \times [0,+ \infty)$, to $v \circ \ga$ at $(1,t_0)$ satisfy
                \begin{equation*}
                    \varphi_t(1,t_0) + H_\ga(1, \varphi'(1,t_0)) \ge 0.
                \end{equation*}
                (Note that the arcs $\ga$ where this condition holds true, change in function of $t_0$).\\
                \item[{\bf (b)}] All $C^1$ supertangents $\psi$ to $u(z,\cdot)$ at $t_0>0$ satisfy
                \begin{equation*}
                    \frac d{dt} \psi(t_0) \le \wha c_z
                \end{equation*}
            \end{itemize}
        \end{itemize}

        \medskip

        We further say that a continuous function $ u$ is a {\it subsolution} (resp. {\it supersolution}) if for every $\gamma\in \EN$:
        \begin{equation*}
            \left\{
            \begin{array}{l}
                \partial_t u_\gamma (s,t) + H_\gamma\left(s, \partial_s u_\gamma(s,t)\right) \le 0 \qquad \text{(resp.\ $\ge 0$)} \qquad \qquad (s,t)\in (0,1) \times (0,+ \infty)\\
                u_\gamma(s,0)\le g(\gamma(s)) \hspace{3 cm} \text{(resp.\ $\ge g(\gamma(s))$} \hspace{2cm} s\in [0,1]
            \end{array}
            \right.
        \end{equation*}
        and condition \textbf{(ii-b)} (resp. \textbf{(ii-a)}) holds true at any vertex $z \in \VN$ and $t_0\in (0,+\infty)$.
    \end{Definition}

    \smallskip

    \begin{Remark}\label{fluxlimiter}
        It has been established in \cite{ImbMon} (in the case of junctions) and in \cite{Sic} (for general networks) that to get existence and uniqueness of solutions to the evolutive Hamilton-Jacobi equation, then it must be coupled not only with a continuous initial datum at time $0$, but also with the so-called {\it flux limiter}, that is a choice of appropriate constants
        \begin{equation*}
            c_z \le \wha c_z \qquad \forall\;z\in \VN;
        \end{equation*}
        for this reason, $ \{\wha c_z\}_{z\in \VN}$ is called {\it maximal flux limiter}.\\
        Flux limiters crucially appear in the conditions that a solution must satisfy on the interfaces and, among other things, bond from above the time derivatives of any subsolution on it. This is the meaning of conditions \textbf{(ii-a)} and \textbf{(ii-b)}. We remark that even if an initial datum is fixed, solutions can change accordingly to the choice of flux limiter and the homogenization procedure might have different features, in particular the limit Hamiltonian might not the same (although they are somehow related). In this article, we decided to focus on the case of maximal flux limiter at every vertex.
    \end{Remark}

    \subsection{Existence and uniqueness of solutions to the evolutive Hamilton-Jacobi equation}\label{subsecexistenceuni}

    We will prove in this section that the problem \eqref{eq:HJ} is well posed and the unique solution is given by an adapted Lax--Oleinik formula.

    In order to do this, let us introduce the {\it Lagrangian} associated to the Hamiltonian $H=\{H_\gamma\}_{\gamma\in \EN}$, namely a collection of functions $L:=\{L_\gamma\}_{\gamma \in\EN}$, defined by the Fenchel transform:
    \begin{equation*}
        L_\gamma(s,\lambda) := \sup_{\rho \in \R} \left( \rho \, \lambda - H_\gamma(s,\rho) \right).
    \end{equation*}
    It follows from the above properties of $H_\gamma$ that $L_\gamma$ satisfies the following conditions:
    \begin{itemize}
        \item[{\bf (L$_\gamma$1)}] $L_\gamma$ is continuous in $(s,\lambda)$ and continuously differentiable in $\lambda$ for any fixed $s$;
        \item[{\bf (L$_\gamma$2)}] $L_\gamma$ is strictly convex in $\lambda$;
        \item[{\bf (L$_\gamma$3)}] $L_\gamma$ is superlinear in $\lambda$ {uniformly in $s\in [0,1]$};
        \item[{\bf (L$_\ga$4)}] $G$--periodic.
    \end{itemize}

    \begin{Remark}
        It follows easily from \eqref{ovgamma} and the Fenchel-Legendre transform that
        \begin{equation*}
            L_{\wtd \gamma}(s,q) = L_\gamma(1-s, -q) \qquad \forall\; s\in [0,1],\; q\in \R.
        \end{equation*}
    \end{Remark}

    \begin{Remark}
        The maximal flux limiter introduced in \eqref{maxxi} can be expressed in terms of the Lagrangians as:
        \begin{equation*}
            \wha c_z = \min_{\ga \in \EN^z} \left [ -\max_{s \in [0,1]} \min_\rho H_\ga(s,\rho) \right ] = \min_{s, \, \ga \in \EN^z} L_\ga(s,0).
        \end{equation*}
    \end{Remark}

    \medskip

    Since our ambient space is a locally finite network, we need to require a number of adjustments to generalize the results holding in the case of finite networks for time-dependent Hamilton--Jacobi equations, see \cite{PozSic,Sic}.

    We emphasize that the results of this section hold for equations posed on any networks, not necessarily periodic, as long as the following conditions are fulfilled (which is the case for periodic networks):
    \begin{itemize}
        \item[{\bf (P1)}] The Euclidean length of arcs in $ \NN$ is bounded and admits a positive infimum.
        \item[{\bf (P2)}] The flux limiters $\wha c_z$ are bounded from below.
        \item[{\bf (P3)}] The $ L_\ga$'s are uniformly superlinear.
        \item[{\bf (P4)}] There exists a constant estimating from below the $L_\ga$'s in $[0,1] \times \R$ for any $\ga$.
        \item[\textbf{(P5)}] $\NN$ is locally finite.
        \item[\textbf{(P6)}] For any $r\ge0$ there is a constant $M_r$ such that
        \begin{equation*}
            \max_{\ga\in\EN,s\in[0,1],|q|\le r}L_\ga(s,q)\le M_r.
        \end{equation*}
        \item[{\bf (P7)}] The network does not have self-loops. See the forthcoming Proposition \ref{luppolo}
    \end{itemize}

    \medskip

    \smallskip

    \subsubsection{Overall Lagrangian on \texorpdfstring{$\NN$}{N}}

    The first step is to define an {\it overall Lagrangian} $L$ on $ \NN$ suitably gluing together the $L_\ga$'s. For this purpose, we introduce the {\em tangent bundle} of $ \NN$, $T \NN$ in symbols, made up by elements $(z, q) \in \NN \times \R^N$ with $q$ of the form
    \begin{equation*}
        q = \la \, \dot\ga(s) \txt{if $z = \ga(s)$, $s \in [0, 1]$, $\la \in \R$.}
    \end{equation*}
    Note that this rule univocally determines a one-dimensional linear subspace of $\R^N$ in correspondence of any $z \in \NN \setminus \VN$, and the union of finitely many one dimensional subspaces if instead $z$ is a vertex.

    We set
    \begin{equation*}
        \EN(z,q) = \{\ga \in \ \EN^z \mid \; \text{$q$ parallel to $\dot\ga(1)$}\} \txt{for $(z,q) \in T \NN$}.
    \end{equation*}

    In the definition of $L$ we distinguish:
    \begin{itemize}
        \item[$\bullet$] If $z \in \NN \setminus \VN$, $z \in \ga((0,1))$ then
        \begin{equation}\label{link1}
            L(z,q) := L_\ga \left ( \ga^{-1}(z), \frac{q \cdot \dot\ga(\ga^{-1}(z))}{|\dot\ga(\ga^{-1}(z))|^2} \right )
        \end{equation}
        \item[$\bullet$] If $z \in \VN$, $q \ne 0$ then
        \begin{equation*}
            L(z,q) := \min_{\ga \in \EN(z,q)} L_\ga \left ( 1, \frac{q \cdot \dot\ga(1)}{|\dot\ga(1)|^2} \right )
        \end{equation*}
        \item[$\bullet$] If $z \in \VN$, $q = 0$ then
        \begin{equation}\label{link3}
            L(z,0) := \wha c_z.
        \end{equation}
    \end{itemize}
    In the above formulae $q \cdot q'$ stands for the scalar product in $\R^N$. Note that the formula \eqref{link3} provides the connection between the Lagrangian $L$ that will appear in the Lax--Oleinik formula and the flux limiter specified in the definition of solution to \eqref{eq:HJ}. See \cite{PozSic} for more comments on the above formulae.

    We can prove, see \cite[Proposition 4.1]{PozSic}:

    \begin{Proposition}
        The Lagrangian $L: T \NN \to \R$ is lower semicontinuous.
    \end{Proposition}

    \smallskip

    We record for later use. Given an interval $[t_1,t_2]$, $T_1 < T_ 2 \in [0,1]$ and an arc $\ga$ in $\NN$, we set:
    \begin{eqnarray*}
        \mathscr A &:=& \{ \text{curves $\xi: [t_1,t_2] \to [T_1,T_2]$} \; \mid \xi(t_1)= T_1, \; \xi(t_2)=T_2, \, \xi((t_1,t_2)) =(T_1,T_2)\}\\
        \mathscr B &:=& \{\text{curves } \zeta:[t_1,t_2] \to \NN \mid \zeta((t_1,t_2))= \ga((T_1,T_2)), \;\zeta(t_1)= \ga(T_1), \; \zeta(t_2)=\ga(T_2) \}.
    \end{eqnarray*}

    \medskip

    \begin{Proposition}\label{lemlemnew}
        Given an arc $\ga$ in $ \NN$, the map
        \begin{equation}\label{lemlemnew1}
            \Upsilon: \xi \mapsto \ga \circ \xi \txt{for $\xi \in \mathscr A$}
        \end{equation}
        is a bijection from $\mathscr A$ to $\mathscr B$ with
        \begin{equation}\label{lemlemnew0}
            \int_{t_1}^{t_2} L_\ga(\xi,\dot\xi) \, dt = \int_{t_1}^{t_2} L(\Upsilon(\xi), D_t (\Upsilon(\xi)) \, dt \quad\text{for any $\xi \in \mathscr A$.}
        \end{equation}
    \end{Proposition}
    \begin{proof}
        It is clear that $\Upsilon$ maps $\mathscr A$ into $\mathscr B$. The inverse $\Upsilon^{-1}$ is given by
        \begin{equation*}
            \Upsilon^{-1}(\zeta)(t) = {\ga}^{-1} \circ \zeta(t) \in \mathscr A \qquad t \in [t_1,t_2], \, \zeta \in \mathscr B
        \end{equation*}
        see \cite[Lemma 3.2]{PozSic}. Setting $\zeta = \Upsilon(\xi)$, we derive from \eqref{lemlemnew1}
        \begin{equation*}
            \dot\zeta(t) = \dot \xi(t) \, \dot{\ga}(\xi(t)),
        \end{equation*}
        which implies by \eqref{link1} and the condition $\xi((t_1,t_2)) =(T_1,T_2)$ that
        \begin{equation*}
            L(\zeta(t),\dot\zeta(t)) = L_{\ga} \left (\xi(t), \frac{\dot \xi(t) \, (\dot{\ga}(\xi(t)) \cdot \dot{ \ga}(\xi(t)))}{|\dot{\ga}(\xi(t))|^2} \right )= L_{\ga}(\xi(t),\dot \xi(t))
        \end{equation*}
        for a.e. $t \in (t_1,t_2)$. This in turn yields \eqref{lemlemnew0} and concludes the proof.
    \end{proof}

    \medskip

    \subsubsection{Action functional in \texorpdfstring{$\NN$}{N}}\label{action}

    \begin{Definition}
        For any curve $\xi:[0,T] \to \NN$ we define its {\it (Lagrangian) action}
        \begin{equation*}
            A_L(\xi):= \int_0^T L(\xi(\tau),\dot\xi(\tau)) \, d\tau.
        \end{equation*}
    \end{Definition}

    \begin{Definition}\label{defPhiLnet}
        Given $z_1$, $z_2$ in $ \NN$, $T >0$, we consider the minimization problem
        \begin{equation*}
            \inf_\xi A_L(\xi),
        \end{equation*}
        where the infimum is over the curves defined in $[0,T]$ linking $z_1$ to $z_2$. We call the corresponding {value function}, {\em minimal action functional} and denote it by $\Phi_L(z_1,z_2,T)$.
    \end{Definition}

    \smallskip

    We proceed defining a class of curves that will have a relevant role in the analysis, see \cite{PozSic}.

    \begin{Definition}\label{adcurve}
        We say that a curve $\xi:[0,T]\to \NN$ is {\it admissible} if there exists a finite partition $\{t_0,\dotsc,t_k\}$ of the interval $[0,T]$ with
        \begin{itemize}
            \item[--] $t_0=0$, $t_k=T$, {$t_i\le t_{i+1}$} for every $i=0,\ldots, k-1$;
            \item[--] $\xi(t_i)\in\VN$ for $i=1,\dotsc,k-1$;
            \item[--] for any $i$, either $\xi((t_i,t_{i+1}))\cap\VN=\emptyset$ and $\xi(t_i)\ne\xi(t_{i+1})$ or $\xi(t)\equiv z\in\VN$ for $t\in(t_i,t_{i+1})$.
        \end{itemize}
    \end{Definition}

    \smallskip

    The key property of admissible curves is:

    \begin{Proposition}\label{admi}
        Given $T > 0$, and a curve $\xi$ defined in $[0, T]$, we can find an admissible curve $\zeta$ defined in the same interval with $\xi(0) = \zeta(0)$, $\xi(T) = \zeta(T)$ and
        \begin{equation*}
            \int_0^T L(\xi (\tau),\dot\xi (\tau)) \, d\tau \ge \int_0^T L(\zeta (\tau),\dot\zeta (\tau)) \, d\tau.
        \end{equation*}
    \end{Proposition}
    \begin{proof}
        The proof goes as in \cite{PozSic} Proposition 4.5, exploiting in addition {\bf (P1)}.
    \end{proof}

    \medskip

    \begin{Remark}
        The uniform superlinearity of the $L_\ga$'s (condition {\bf (P3)}) implies that $L$ is superlinear as well, namely for any $a >0$ there exists $B_a >0$ with
        \begin{equation}\label{super}
            L(z,q ) \ge a \, |q| - B_a \qquad\text{for any $(z,q) \in T \NN$.}
        \end{equation}
        In the above formula it is actually enough to take
        \begin{equation*}
            B_a = \max \{- L(z,q) + a \, |q| \mid (z,q) \in T \NN\}.
        \end{equation*}
        Moreover, item {\bf (P4)} implies:
        \begin{equation*}
            \text{The Lagrangian $L$ is bounded by below in $T \NN$.}
        \end{equation*}
    \end{Remark}

    \bigskip

    \begin{Proposition}\label{curvacce}
        \hfill
        \begin{itemize}
            \item[{\bf (i)}] There exist admissible curves minimizing the action between two given points of $ \NN$ in a given time, and any of them is Lipschitz continuous;
            \item[{\bf (ii)}] Given $C >0$, the curves minimizing the action between $z_1$ and $z_2$ in a time $T$ are equi--Lipschitz continuous provided that
            \begin{equation}\label{enri0}
                d_{ \NN}(z_1,z_2) \le C \, T.
            \end{equation}
        \end{itemize}
    \end{Proposition}
    \begin{proof}
        The existence of a minimal curve and the fact that such curve is Lipschitz continuous, is obtained as in \cite[Theorems 5.2 and 5.3]{PozSic} through superlinearity of $L$ plus Dunford Pettis, see \cite{ButtazzoGiaquintaHildebrandt99}.

        As a first step for proving item {\bf (ii)}, we show that if \eqref{enri0} holds true, then for any $\zeta:[0,T] \to \NN$ minimizing the action between $z_1$ and $z_2$ there exists $l_0 >0$ such that
        \begin{equation}\label{enri01}
            |\dot \zeta(t)| \le l_0 \txt{for $t$ in a set of positive measure.}
        \end{equation}
        Let $\xi$ be a geodesic for $d_{ \NN}(z_1,z_2)$, parametrized with the Euclidean arc-length, then $\eta(t):= \xi \left (\frac{d_{ \NN}(z_1,z_2)}T t \right )$ is defined in $[0,T]$ and links $z_1$ to $z_2$; consequently
        \begin{equation*}
            A_L(\eta) = \int_0^T L \left ( \xi \left (\frac{d_{ \NN}(z_1,z_2)}T t \right ), \frac{d_{\NN}(z_1,z_2)}T \dot\xi \left (\frac{d_{ \NN}(z_1,z_2)}T t \right ) \right ) \, dt \ge \Phi_L(z_1,z_2,T).
        \end{equation*}
        If $R$ is an upper bound for $L(z,q)$ when $|q| \le C$, we get
        \begin{equation}\label{enri1}
            \Phi_L(z_1,z_2,T) \le R \, T.
        \end{equation}
        Using the fact the Lagrangian is superlinear, we find $l_0$ with
        \begin{equation*}
            |q| > l_0 \Rightarrow L(z,q) > R \txt{for any $z$}.
        \end{equation*}
        Taking into account \eqref{enri1}, we derive that any optimal curve $\zeta$ must indeed satisfy \eqref{enri01}. To get the assertion, we can now proceed as in the proof of \cite[Lemma A.3]{SicMTA}.
    \end{proof}

    \medskip

    \subsubsection{{Lax--Oleinik representation formula and comparison result}}

    We define
    \begin{equation}\label{lapulapu}
        w(z,t):= \inf \left \{g(\xi(0)) + A_L(\xi) \right \},
    \end{equation}
    where the infimum is over the curves defined in $[0,t]$ with $\xi(t)=z$.

    Recall that $g: \NN \to \R$ is uniformly continuous, we can prove the following.

    \smallskip

    \begin{restatable}{Proposition}{uffaz}\label{uffaz}
        The function $w$ defined in \eqref{lapulapu} is uniformly continuous in $\NN \times [0,+\infty)$.
    \end{restatable}

    The proof is in section \ref{uffazz}

    \smallskip

    \begin{Lemma}\label{nosleep}
        The infimum in the definition of $w$ is a minimum.
    \end{Lemma}
    \begin{proof}
        By Proposition \ref{admi}, there is a sequence $\xi_n$ of admissible curves defined in $[0,t]$ and ending at $z$ with
        \begin{equation*}
            \lim_{n\to +\infty} \left( g(\xi_n(0)) + \int_0^t L(\xi_n,\dot \xi_n) \, d\tau \right)= w( z,t).
        \end{equation*}
        If $|\xi_n(0)|$ is bounded, and consequently $g(\xi_n(0))$ is bounded by below, then the lengths of the $\xi_n$'s are equibounded by \eqref{super} and so the number of arcs visited by the $\xi_n$'s is equibounded because of {\bf (P1)}. By {\bf (P5)}, we therefore get the assertion as in the case of finite networks, see \cite[Theorem 5.4]{PozSic}. \\
        Consequently, we can reduce to the case when $ |\xi_n(0)| \to + \infty. $ We show that it is actually not possible. In fact, from \eqref{super} with an arbitrary positive $a$ and a related $B_a$, we deduce:
        \begin{eqnarray*}
            w(z,t) &\ge& \int_0^t L(\xi_n,\dot\xi_n) \, d\tau + g(\xi_n(0)) - \de_n \\
            &\ge& a \int_0^t |\dot \xi_n| \, d \tau - B_a t + g(\xi_n(0)) - \de_n\\
            &\ge& a \, d_{\NN}(z,\xi_n(0))- B_a t + g(\xi_n(0)) - \de_n,
        \end{eqnarray*}
        where $\de_n$ is a suitable infinitesimal positive sequence. Therefore
        \begin{equation*}
            \liminf_n \frac{- g(\xi_n(0))}{|\xi_n(0)|} \ge \liminf _n \frac{ a \, d_{\NN}(z,\xi_n(0)) - B_a t - w(z,t) - \de_n}{|\xi_n(0)|} \ge a,
        \end{equation*}
        which implies, because of the arbitrary choice of $a$,
        \begin{equation*}
            \lim_n \frac{- g( \xi_n(0))}{|\xi_n(0)|} = + \infty
        \end{equation*}
        in contradiction with the uniform continuity of $g$.
    \end{proof}

    \smallskip

    The same argument as in \cite[Theorem 6.7 ]{PozSic} finally yields:

    \begin{Theorem}[{\bf Existence of solutions}]
        The function $w$ defined in \eqref{lapulapu} is solution to \eqref{eq:HJ} according to Definition \ref{defsol}.
    \end{Theorem}

    \medskip

    We finally state a comparison result yielding that the function defined in \eqref{lapulapu} is actually the unique solution to \eqref{eq:HJ}.

    \begin{restatable}[{\bf Comparison}]{Theorem}{compa}\label{compa}
        Given a subsolution $u$ and a supersolution $v$ to \eqref{eq:HJ}, we have $u \le v$ on $\NN \times (0,+\infty)$.
    \end{restatable}

    The proof is given in Section \ref{compaz}

    \bigskip

    \section{Maximal topological crystal over \texorpdfstring{$\G_0$}{Γ₀}}\label{construens}

    As we have remarked above, the underlying abstract graph associated to our periodic network $\NN$ plays a crucial role in our strategy, particularly as far as the discretization process that we are going to implement is concerned. Moreover, it encodes the complexity of the network, which is a fundamental aspect of the asymptotic analysis of the limit problem.

    We recall that the maximal topological crystal $\G$ over a base graph $\Gamma_0$ is unique up to isomorphisms (see Remark \ref{remarkunicityuptoiso} ({\bf ii})). In this section, we describe how to construct a (unique up to isomorphisms) representation of $\G$ in terms of the base graph $\Gamma_0$ and $\Z^{b(\G_0)}$.

    We start by recalling some preliminary notions and properties of $\Gamma_0$. {In this respect the paper is self-contained, namely all the few concepts on graph theory and algebraic topological properties of graphs that we need, are preliminarily introduced and explained in the text.} {See \cite{Sunada} for a more comprehensive presentation.}

    \smallskip

    \subsection{Spanning trees, incidence and rotation vectors in \texorpdfstring{$\G_0$}{Γ₀}}

    {Let $\G_0=(\VV_0, \EE_0)$ be the base graph.}

    \begin{Definition}
        An {\it orientation} of $\G_0$ is a subset $\EE_0^+$ of the edges satisfying
        \begin{equation*}
            - (\EE_0^+) \cap \EE_0^+ = \emptyset \qquad{\rm and} \qquad - (\EE_0^+) \cup \EE_0^+ = \EE_0.
        \end{equation*}
        In other words, an orientation of $\G_0$ consists of a choice of exactly one edge in each pair $\{e,-e\}$.
    \end{Definition}

    \smallskip

    We fix an orientation $\EE_0^+$ on $\G_0$ with cardinality denoted by $|\EE_0^+| = \frac{1}{2}|\EE_0|$.

    \smallskip

    \begin{Definition}
        For any path $\xi$ in $\G_0$, we call {\it multiplicity} of an edge $e$ in $\xi$, and denote by $\#_{\xi}(e)$, the number of passages through $e$ needed to go from $\oo(\xi)$ to $\tt(\xi)$ following $\xi$. Note that it is vanishing if and only $e \not\in \xi$.
    \end{Definition}

    \smallskip

    We consider the $|\EE_0^+|$-dimensional vector space over $\R$, denoted by $\Cf_1(\G_0,\R)$, spanned by the edges of $\EE_0^+$ with the reversed edge identified to the opposite, as suggested by the notation. We further indicate by $\Cf_0(\G_0,\R)$ the $|\VV_0|$-dimensional vector space over $\R$ spanned by the vertices of $\G_0$.

    \begin{Definition}
        To any path $\xi$ of $\G_0$, we associate an integer vector belonging to $\Cf_1(\G_0,\R)$ given by
        \begin{equation*}
            [\xi]:= \sum_{e \in{ \EE_0}^+} ( \#_{\xi}(e)-\#_{\xi}(-e)) \, e.
        \end{equation*}
        $[\xi]$ is called {\em incidence vector} of $\xi$.

        We call {\it cycle} any path $\xi$ satisfying $\oo(\xi)= \tt(\xi)$. A cycle is said to be {\it trivial} if its incidence vector is vanishing.
    \end{Definition}

    \medskip

    We proceed considering the linear transformation $\partial: \Cf_1(\G_0,\R) \longrightarrow \Cf_0(\G_0,\R)$ given by
    \begin{equation*}
        \partial(e) := \tt(e) - \oo(e) \qquad\text{for every $e \in \EE_0^+$}
    \end{equation*}
    and then extended linearly. We focus on its kernel that we denote $\ker \, \partial =: H_1(\G_0,\R)$; it is called the {\it first homology group} of $\G_0$ with real coefficients. We will provide in what follows the few simple properties from homology theory we need, without mentioning homology any further.

    A relevant fact, that we are going to use in the following, is:

    \begin{Lemma}\label{cy}
        The vector subspace $H_1(\G_0,\R) \subset \Cf_1(\G_0,\R)$ is spanned by the incidence vectors of the cycles in $\G_0$.
    \end{Lemma}
    \begin{proof}
        See \cite[p. 40--41]{Sunada}.
    \end{proof}

    \begin{Definition}
        Given $\bar\VV \subset \VV_0$, $\bar \EE \subset \EE_0$, the pair $\bar\G= (\bar\VV,\bar\EE)$ is called a {\it subgraph} of $\G_0$ if
        \begin{eqnarray*}
            e \in \bar\EE &\Rightarrow& \oo(e) \in \bar\VV \\
            e \in \bar \EE & \Rightarrow& - e \in \bar\EE,
        \end{eqnarray*}
        this clearly implies that $\tt(e) \in \bar\VV$, whenever $e \in \bar\EE$. A subgraph is clearly a graph by itself.
    \end{Definition}

    \smallskip

    \begin{Definition}
        A {\em tree} $\TT =(\VV_{\TT},\EE_{\TT})$ in $ \G_0$ is a connected subgraph of $ \G_0$ without nontrivial cycles. A {\em spanning tree} is a subgraph of $\G_0$ which is maximal, with respect to the inclusion, among all the trees that are subgraphs of $\G_0$. Equivalently, it is a tree satisfying $\VV_{\TT}= \VV_0$. Any {connected} graph contains a spanning tree, even if it is in general not unique, see \cite[p. 31]{Sunada}.
    \end{Definition}

    We fix a spanning tree $\TT =(\VV_{\TT},\EE_{\TT})$ in $ \G_0$, and denote by $\EE_{\TT}^+$ an orientation of $\TT$ contained in $\EE_0^+$.

    \smallskip

    \begin{Definition}
        We call {\it simple} a path without repetition of vertices, except possibly the initial and terminal vertices; in other terms $\xi=(e_i)_{i=1}^M$ is simple if
        \begin{equation*}
            \tt(e_i) = \tt(e_j) \, \Rightarrow i=j.
        \end{equation*}
        Note there are finitely many simple paths in $\G_0$. We call {\it circuit}, a simple closed path.
    \end{Definition}

    \smallskip

    {It is easy to deduce the following: }

    \begin{Proposition}\label{cry}
        A path is simple if and only there is no circuit properly contained in it.
    \end{Proposition}

    \medskip

    \begin{Lemma}\label{crycry}
        Given $e_0 \in \EE_0 \setminus \EE_{\TT}$, there exists one and only one circuit $\xi$ containing $e_0$ and with all the other edges in $\EE_{\TT}$.
    \end{Lemma}
    \begin{proof}
        Since $\tt(e_0)$ and $\oo(e_0)$ are vertices of $\TT$, which is connected, there is a path $\xi$ in $\TT$ linking them, that must be simple by Proposition \ref{cry}. Then, the circuit $\zeta = \xi \cup e_0$ satisfies the assertion. If there were another circuit $\zeta_1 \ne \zeta$ satisfying the assertion as well, then we could link $\tt(e_0)$ to $\oo(e_0)$ through $\zeta$ and go back from $\oo(e_0)$ to $\tt(e_0)$ following $- \zeta_1$. In this way, we would construct a nontrivial cycle contained in $\TT$, contradicting the property of $\TT$ being a tree.
    \end{proof}

    \smallskip

    \begin{Definition}
        We call {\it fundamental} the circuits appearing in the previous statement.
    \end{Definition}

    \smallskip

    \begin{Remark}\label{electricity}
        There is an interpretation of the fundamental circuits in terms of electricity. In $\TT$ there is no flow of electricity since there are no nontrivial circuits. However, because of the maximality of $\TT$, the addition of any edge outside $\TT $ allows closing a uniquely determined circuit, along which electricity will flow.
    \end{Remark}

    \smallskip

    We define a map $\theta$ associating to any $e \in\EE^+_0\setminus \EE_{\TT}$ the incidence vector of the circuit provided by the Lemma \ref{crycry}, and $0$ to any $e \in\EE_{\TT}$. Taking into account Lemma \ref{cy}, we can extend it by linearity to a map
    \begin{equation*}
        \theta: \Cf_1(\G_0,\R) \to H_1(\G_0,\R).
    \end{equation*}

    The following key fact holds true, see \cite[Section 5.5]{Sunada}.

    \begin{Proposition}\label{cycy}
        The map $\theta$ defined above is surjective and, in addition, the incidence vectors $\{\theta(e)\}_{e \in \EE_0^+ \setminus \EE^+_{\TT}}$ form a basis of $H_1(\G_0,\R)$.
    \end{Proposition}

    This implies:

    \begin{Corollary}\label{incide}
        The dimension of $H_1(\G_0,\R)$ is given by
        \begin{equation*}
            |\EE_0^+| - |\VV_0| +1 = \frac12 |\EE_0| - |\VV_0| +1.
        \end{equation*}
    \end{Corollary}
    \begin{proof}
        We deduce from the previous proposition that
        \begin{equation*}
            \dim \, H_1(\G_0,\R) = |\EE_0^+| - | \EE^+_{\TT}|.
        \end{equation*}
        Our claim then follows from the identity $|\EE_{\TT}^+|=|\VV_0|-1$, which is well known in graph theory, see, e.g., \cite[Corollary 1.5.3]{Diestel17}.
    \end{proof}

    \smallskip

    \begin{Definition}
        $b(\G_0):= \frac{|\EE_0|}2 - |\VV_0| +1$ is called the Betti number of $\G_0$.
    \end{Definition}

    \begin{Remark}\label{Euler}
        $b(\G_0)$ can be also characterized in a combinatorial way via the so-called {\it cyclomatic number} introduced by Kirchhoff (see \cite[Formula (4.3)]{Sunada}). This is a graph version of the famous formula due to Euler (1750) which asserts that if $v$, $e$, $f$ are the number of vertices, edges and faces of a convex polyhedron, respectively, then $v - e + f = 2$.\\
        More generally, if $\G_0$ were not connected, one could show that $b(\G_0)$ is equal to $\frac{1}{2}|{ \EE_0}| - |{ \VV_0}| + {\mathbf c}(\G_0)$, where ${\mathbf c}(\G_0)$ denotes the number of connected components of $\Gamma_0$.
    \end{Remark}

    \begin{Definition}
        Given a path $\xi$ contained in $\G_0$, we call $\theta([\xi])$ the {\it rotation vector} of $\xi$.
    \end{Definition}

    \begin{Remark}
        Taking into account Remark \ref{electricity}, one can also interpret $\theta([\xi])$ as the number of different fundamental circuits one closes going through $\xi$. Since orientation is also taken into account, this yields an integer vector with dimension equal to the number of edges of $\EE_0^+$ not belonging to the maximal tree $\TT$.
    \end{Remark}

    \medskip

    One can prove the following:

    \begin{Lemma}
        If $\xi$ is a cycle, then its incidence and rotation vector coincide. Namely, $\theta$ is the identity when restricted $H_1(\G_0,\R)$.
    \end{Lemma}
    \begin{proof}
        The claim is a consequence of the fact that
        \begin{equation}\label{cycy01}
            \theta^2(e)=\theta(e) \txt{for any $e \in \EE_0^+ \setminus \EE_{\TT},$}
        \end{equation}
        as it follows from the definition of $\theta$. If $\xi$ is a cycle then $[\xi] \in H_1(\G_0,\R)$ by Lemma \ref{cy}; this implies by Proposition \ref{cycy} that
        \begin{equation*}[\xi]
            =\sum_{e \in \EE_0^+ \setminus \EE_\TT} l_e \theta(e)
        \end{equation*}
        for some integer coefficients $l_e$, which in turn yields by \eqref{cycy01}
        \begin{equation*}
            \theta([\xi]) = \sum_{e \in \EE_0^+ \setminus \EE_\TT} l_e \theta^2(e) = \sum_{e \in\EE^+_0\setminus \EE_\TT} l_e \theta(e) = [\xi].
        \end{equation*}
    \end{proof}

    \smallskip

    \begin{Remark}
        According to Proposition \ref{cycy}, we can fix a basis of $H_1(\G_0,\R)$ formed by $\theta(e)$ with $e$ varying in of $\EE_0^+ \setminus \EE_\TT$. This allows one to identify $H_1(\G_0,\R)$ and $H_1(\G_0,\Z)$, the first homology group with real and integer coefficients, to $\R^{b(\G_0)}$ and $\Z^{b(\G_0)}$, respectively. \\
        From now on, we suppose to have fixed a basis and we assume the above identifications. Same identification will apply for the dual spaces $H^1(\G_0,\R)$, $H^1(\G_0,\Z)$, which are called the {\it cohomology groups} with real and integer coefficients, respectively.
    \end{Remark}

    \medskip

    \begin{Example}\label{graphenebase}
        We present some graphs with their spanning trees in figure \ref{fig:graphenebase}. They are the base graphs of the periodic networks seen in Examples \ref{bouquetexe} and \ref{graphenecryexe}.
        \begin{itemize}
            \item Let $\G_1$ and $\TT_1$ be the graph and the spanning tree defined in figure \ref{fig:graphenebase}. We have that
            \begin{equation*}
                \theta(e_0)=0,\qquad\theta(e_1)=e_1-e_0,\qquad\theta(e_2)=e_2-e_0.
            \end{equation*}
            We know from Proposition \ref{cycy} that $H_1(\Gamma_1,\R)$ is generated by $\theta(e_1)$ and $\theta(e_2)$.

            \item Let $\G_2$ and $\TT_2$ be the graph and the spanning tree defined in figures \ref{subfig:bouquet} and \ref{subfig:bouquettree}. In this case we have that the spanning tree is just a point and $H_1(\G_2,\R)$ is generated by $\theta(f_1)=f_1$ and $\theta(f_2)=f_2$.
        \end{itemize}
    \end{Example}

    \begin{figure}[ht]
        \begin{subcaptiongroup}
            \phantomcaption\label{subfig:graphenebase}
            \includegraphics{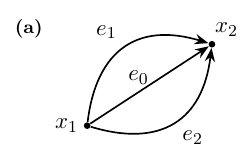}
            \qquad
            \phantomcaption\label{subfig:graphenetree}
            \includegraphics{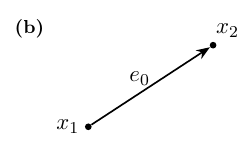}
            \\
            \hspace{-27mm}
            \phantomcaption\label{subfig:bouquet}
            \raisebox{-53.68419pt}{\includegraphics{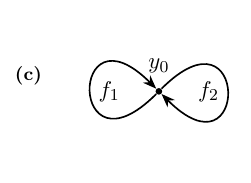}}
            \hspace{7mm}
            \phantomcaption\label{subfig:bouquettree}
            \raisebox{-15.67137pt}{\includegraphics{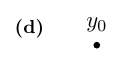}}
        \end{subcaptiongroup}
        \caption{We call $\G_1$, $\G_2$ the (oriented) graphs \subref{subfig:graphenebase}, \subref{subfig:bouquet} and $\TT_1$, $\TT_2$ their spanning trees in \subref{subfig:graphenetree}, \subref{subfig:bouquettree}, respectively.}\label{fig:graphenebase}
    \end{figure}

    \medskip

    \subsection{A representation of the maximal topological crystal}\label{murano}

    We recall that the maximal topological crystal over the base graph is unique up to isomorphism. In this section we describe a particular representative whose structure will be particularly convenient for our analysis.

    \begin{Definition}\label{defstandtopcry}
        We define a graph $\wtd\Gamma:=(\wtd\VV, \wtd\EE)$ with sets of vertices and edges given by
        \begin{equation*}
            \wtd\VV:= \VV_0 \times \Z^{b(\G_0)} \quad{\rm and} \quad \wtd\EE:= \EE_0 \times \Z^{b(\G_0)}
        \end{equation*}
        and the associated maps:
        \begin{eqnarray*}
            \oo: \wtd\EE &\longrightarrow & \wtd\VV \nonumber\\
            (e, h) &\longmapsto& (\oo(e), h),
        \end{eqnarray*}
        and
        \begin{eqnarray}\label{defbartopcry}
            -: \wtd \EE &\longrightarrow &\wtd\EE \nonumber\\
            (e, h) &\longmapsto& (-e, h + \theta(e)).
        \end{eqnarray}
    \end{Definition}

    Since $\theta: \Cf_1(\G_0,\Z) \to \Z^{b(\G_0)}$ is a group homomorphism, the map in \eqref{defbartopcry} is indeed a fixed point involution:
    \begin{eqnarray*}
        -(-(e,h)) &=& - (- e, h + \theta( e)) \\
        &=& (- (- e), h + \theta( e)+ \theta(- e) )= (e,h).
    \end{eqnarray*}
    We also have
    \begin{equation*}
        \tt( e,h)= \oo(- ( e,h)) = \oo(- e,h + \theta( e))= ( \tt( e),h + \theta( e)).
    \end{equation*}

    \smallskip

    \begin{Remark}\label{treecopy}
        Since $\theta \equiv 0$ on $\Cf_1( \TT,\Z)$, we obtain for any edge $e \in \EE_\TT$ and $h \in \Z^{b(\G_0)}$
        \begin{equation*}
            \oo(e,h)= (\oo(e),h) \quad\text{and} \quad \tt(e,h)= (\tt(e),h).
        \end{equation*}
        This shows that for any $h \in \Z^{b(\G_0)}$ the graph
        \begin{equation*}
            \TT \times \{h\} := ( \VV_\TT \times \{ h\}, \EE_\TT \times \{ h\})
        \end{equation*}
        is an isomorphic copy of $ \TT$ in $\wtd\G$.
    \end{Remark}

    \medskip

    \begin{Proposition}
        $\wtd\G = (\wtd\VV, \wtd\EE)$ is a maximal topological crystal over $\Gamma_0$, so it isomorphic to $\Gamma$.
    \end{Proposition}
    \begin{proof}
        We denote $\wtd\G = (\wtd\VV, \wtd\EE)$ as in Definition \ref{defstandtopcry}. For every $h_0 \ne 0 \in \Z^{b(\G_0)}$, we define the maps $\FN^{h_0}_{\wtd\VV}: \wtd\VV \to \wtd\VV$ and $\FN^{h_0}_{\wtd\EE}:\wtd\EE \to\wtd\EE $ by
        \begin{equation*}
            \FN^{h_0}_{\wtd\VV}(x,h) =(x,h+h_0) \quad\text{and}\quad \FN^{h_0}_{\wtd\EE}(e,h)=(e, h+h_0) \quad\text{for any $x \in \VV_0$, \,$e \in \EE_0$, $h \in \Z^{b(\G_0)}$.}
        \end{equation*}
        It is clear that $\FN^{h_0}_{\wtd\VV}$ and $\FN^{h_0}_{\wtd\EE}$ are both bijective. Moreover:
        \begin{eqnarray*}
            \oo(\FN^{h_0}_{\wtd\EE}(e,h)) &=& \oo(e,h+h_0) = \FN^{h_0}_{\wtd\VV}(\oo(e),h)= \FN^{h_0}_{\wtd\VV}(\oo(e,h)) \\
            - \FN^{h_0}_{\wtd\EE}(e,h) &=&-(e, h+h_0)= (-e,h+h_0+\theta(e))= \FN^{h_0}_{\wtd\EE}(-e,h+ \theta(e)) \\
            &=& \FN^{h_0}_{\wtd\EE}(-(e,h)).
        \end{eqnarray*}
        Therefore, $(\FN^{h_0}_{\wtd\VV},\FN^{h_0}_{\wtd\EE})$ is a graph automorphism according to Definition \ref{morph}. If we denote by
        \begin{equation*}
            \wtd G:= \left\{(\FN^{h_0}_{\wtd\VV},\FN^{h_0}_{\wtd\EE}): \; h_0\in \Z^{b(\G_0)}\right\}
        \end{equation*}
        it is easy to verify that $G$ is a group under composition, which is isomorphic to $\Z^{b(\G_0)}$; the isomorphism is given by:
        \begin{eqnarray*}
            \varphi: (\Z^{b(\G_0)},+) &\longrightarrow& (\wtd G, \circ)\\
            h_0 &\longmapsto& (\FN^{h_0}_{\wtd\VV},\FN^{h_0}_{\wtd\EE}).
        \end{eqnarray*}
        To conclude the proof, we need to show that $\wtd G$ acts freely on $\wtd \G$. In fact:
        \begin{itemize}
            \item It is clear from \eqref{defbartopcry} that $\FN^{h_0}_{\wtd\EE} (e,h) \ne -(e,h)$ for every $(e,h)\in \wtd \G$ and $h_0\in \Z^{b(\G_0)}$. So, $G$ acts without inversion.
            \item If for some $h_0\in \Z^{b(\G_0)}$ and $(x,h) \in \wtd \VV $ it happens that $\FN^{h_0}_{\wtd\VV} (e,h) = (e,h)$, then:
            \begin{equation*}
                (e, h+ h_0) = (e, h)
            \end{equation*}
            which implies that $h_0=0$ and therefore $(\FN^{0}_{\wtd\VV},\FN^{0}_{\wtd\EE})$ is the identity.
        \end{itemize}
        Moreover, the quotient graph is isomorphic to $\G_0$. This completes the proof of the assertion.
    \end{proof}

    \medskip

    We identify from now on the maximal topological crystal $\G$ over $\G_0$ with the graph $\wtd\G$ constructed in this section.

    \begin{Notation}\label{pi}
        Given an arbitrary periodic network $\NN =(\VN, \EN)$ with base graph $\G_0$, we fix bijections $\Psi_\VN$, $\Psi_\EN$ relating it to its underlying graph (the maximal crystal) $\G$, see Definition \ref{sottostante}, and set to ease notations
        \begin{equation*}
            \pi(z) = \Psi_\VN(z) \txt{for any $z \in \VN$}.
        \end{equation*}
        Taking into account that $\pi(z) \in \VV_0 \times{\Z}^{b(\G_0)}$, we further define $\pi_1(z)$, $\pi_2(z)$, for $z \in \VN$, as the projection of $\pi(z)$ into $\VV_0$, ${\Z}^{b(\G_0)}$, respectively.
    \end{Notation}

    \medskip

    Given a path $ \xi = ( e_i,h_i)_{1=1}^m$ in $\G$, the concatenation principle reads
    \begin{equation}\label{prepatch}
        \oo( e_{i+1},h_{i+1})=(\tt( e_i),h_i+\theta( e_i))= \tt( e_i,h_i),
    \end{equation}
    or in other terms
    \begin{equation}\label{patch}
        h_{i+1}= h_i +\theta( e_i) \quad\text{and}\quad \oo( e_{i+1})= \tt( e_i) \quad\text{for $i=1,\ldots,m$.}
    \end{equation}
    Therefore $ \xi_0= ( e_i)_{i=1}^m$ is a path in $\G_0$ that we call the {\it projected path} of $\xi$. \\
    Conversely, it is clear from \eqref{patch} that if we take a path $\xi_0= ( e_i)_{i=1}^m$ in $\G_0$ and choose $h \in \Z^{b(\G_0)}$, then there exists one and only one path $\xi$ in $\G$ starting at $(\oo( e_1),h)$ with projection on $\G_0$ equal to $\xi_0$, and it is given by
    \begin{equation}\label{eq:pathlift}
        \xi:= \left ( e_i, h + \sum_{j= 1}^{i} \theta( e_j) \right )_{i=1}^m.
    \end{equation}
    This is called the {\it unique path-lifting property}, see \cite[page 54]{Sunada}. Note that the terminal vertex of $\xi$ is
    \begin{equation}\label{patchbis}
        \tt(\xi)= \left (\tt( \xi_0), h + \theta \left (\sum_{j= 1}^{m} e_j \right ) \right ).
    \end{equation}
    We see from \eqref{patchbis} that $\xi$ is a cycle, not necessarily trivial, in $\G$ if and only if the projected path $\xi_0$ is a trivial cycle in $ \Gamma_0$, namely with $\theta([\xi_0])=0$.

    \smallskip

    We have:

    \begin{Proposition}\label{luppolo}
        The graph $ \G$ does not contain any self-loop, see Definition \ref{loop}.
    \end{Proposition}
    \begin{proof}
        Assume by contradiction that the assertion is false, then, according to \eqref{patchbis}, the projection of a self-loop in $\G$ is a self-loop in $\G_0$, denoted by $e_0$. Then, it follows from the definition of $\theta$, that we must have $\theta(e_0)=0$, which would $e_0 \in \EE_\TT$, which is impossible since $T$ is a tree.
    \end{proof}

    \smallskip

    \begin{Remark}\label{periodicgraph}
        Taking in mind Remark \ref{treecopy}, we see that $\G$ can be thought as composed by infinite copies of $\TT$, indexed by $h\in \Z^{b(\G_0)}$ and linked together by the lifts of the edges $e\in\EE_0\setminus\EE_\TT$. {The $\Z^{b(\G_0)}$ components of the elements of $\Gamma$ are related to which different copies of $\TT$ these elements belong to; in particular, thanks to the unique path-lifting property \eqref{eq:pathlift} and Proposition \ref{cycy}, the difference in these components indicates how these copies are connected to each other.} Namely a path $\xi$ in $\G$, with projection $\xi_0$ in $\G_0$, links a vertex in $\TT\times\{h\}$ to one in $\TT\times\{h'\}$ if and only if $h'=h+\theta([\xi_0])$.
    \end{Remark}

    Remark \ref{periodicgraph} provides a simple way to identify a vertex in $\G$ using its position relative to a given reference vertex, as it can be seen in the example below.

    \begin{Example}
        Let $\G'_1$ be the maximal topological crystal of the graph $\G_1$ given in figure \ref{fig:graphene}. We recall that in Example \ref{graphenebase} it is given a spanning tree $\TT_1$ of $\G_1$ and a basis of $H_1(\G_1,\R)$. It is apparent, in view of Remark \ref{periodicgraph}, that $\G'_1$ is made of infinite copies of the spanning tree $\TT_1$ linked together by the lifts of the edges $e_1$, $e_2$.\\
        Now let $\wha x$ and $\wha y$ be as in figure \ref{fig:graphene}. In view of figure \ref{subfig:graphenelift} there is an $h\in H_1(\G_1,\Z)$ such that $\wha x=(x_1,h)$. Then it is easy to check, using the unique path-lifting property, that
        \begin{equation*}
            \wha y=(x_2,h+2\theta(e_2)-3\theta(e_1)).
        \end{equation*}
    \end{Example}

    \smallskip

    \begin{figure}[ht]
        \begin{subcaptiongroup}
            \phantomcaption\label{subfig:graphenebasecov}
            \raisebox{-39.2475pt}{\includegraphics{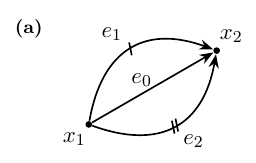}}
            \qquad
            \phantomcaption\label{subfig:graphenelift}
            \raisebox{-54.81616pt}{\includegraphics{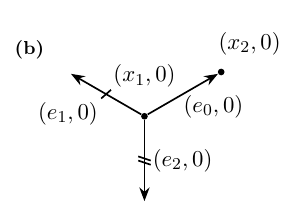}}
            \\
            \phantomcaption\label{subfig:graphene}
            \includegraphics{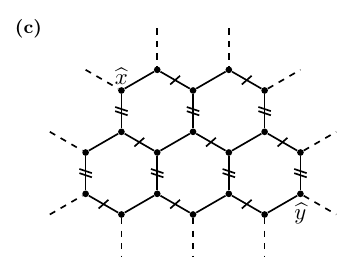}
        \end{subcaptiongroup}
        \caption{In \subref{subfig:graphene} is represented a maximal topological graph $\wha\G_1$ of the graph $\G_1$ in \subref{subfig:graphenebasecov}. In \subref{subfig:graphenelift} it is shown a lift of $\G_1$.}\label{fig:graphene}
    \end{figure}

    \bigskip

    \section{Discretization of the problem: from the network to the base graph}

    One of the main idea behind our strategy consists in {\it discretizing} the problem, posing it on the base graph associated to the network. In order to do this, we need to transfer the Hamiltonians and the Lagrangians to discrete objects $\{\HH(e, \cdot)\}_{e \in \EE_0}$ and $\{\LL(e, \cdot)\}_{e \in \EE_0}$ defined on the base graph $\G_0$, extrapolating all the information that are relevant to our problem. Our aim is to study the corresponding average action functional on $\G_0$, in particular its asymptotic behaviour as time goes to infinity.

    \medskip

    \subsection{Hamiltonians and Lagrangians on \texorpdfstring{$\G_0$}{Γ₀}}

    \begin{Notation}
        Let $\Psi_\EN: \EN \to \EE$ one of the bijections relating $\NN$ to its underlying graph $\G$, see Notation \ref{pi}. In order to ease notation, given $e \in \EE_0$, we denote by $H_e$ the Hamiltonian $H_\ga$ associated to the arcs $\ga$ such that
        \begin{equation*}
            \Psi_\EN(\ga)= (e,h) \txt{for some $h \in \Z^{b(\G_0)}$.}
        \end{equation*}
        The Hamiltonian does not depend on $h \in \Z^{b(\G_0)}$ because of the $G$-periodicity of $H$, see assumption {\bf (H$_\ga$4)}. Similarly, we use the same convention for Lagrangian $L_e$.
    \end{Notation}

    \smallskip

    Let us introduce the following {\it critical values}: for any $e \in \EE_0$, we set:
    \begin{eqnarray}
        a_e &:=& \max_{s\in[0,1]}\min_{\rho \in \R} H_e(s,\rho)\label{a00} \\
        a_0 &:=& \max_{e \in \EE_0} a_e.
    \end{eqnarray}

    \smallskip

    \begin{Remark}\label{minus}
        It is clear that
        \begin{equation*}
            a_e= a_{-e} \qquad\text{for any $e \in \EE_0$.}
        \end{equation*}
    \end{Remark}

    \medskip

    Given $a \in [a_e, + \infty)$, we set
    \begin{equation}\label{deffi}
        \si^+_a(e,s)= \max \{\rho \mid H_e(s,\rho)=a\} \qquad{\rm and} \qquad \quad \si(e,a) := \int_0^1 \si^+_{a}(e,s) \, ds.
    \end{equation}
    It is apparent that
    \begin{equation*}
        \si^+_a(-e,s) := - \min \{\rho \mid H_e(1-s,\rho)=a\}.
    \end{equation*}

    \medskip

    Before defining the discrete Hamiltonian $\HH(e,\cdot)$'s, we need a preliminary result.

    \begin{Lemma}\label{newborn}
        The function $a \mapsto \si(e,a)$ from $[a_e, + \infty)$ to $\R$ is continuous and strictly increasing. In addition, it is strictly concave, of class $C^1$ in $(a_e,+\infty)$ and satisfies
        \begin{equation*}
            \lim_{a \to + \infty} \frac{\si(e,a)}a =0.
        \end{equation*}
    \end{Lemma}
    \begin{proof}
        The monotonicity directly follows from the definition; the continuity is obtained exploiting the fact that $a \mapsto \si^+_a(e,s)$ is continuous and using the dominated convergence theorem to bring the limit inside the integral. By the strict convexity assumption on $H_e$, we derive for any $s \in [0,1]$, $\rho_0 \in (0,1)$, $b > a$ in $[a_e,+\infty)$:
        \begin{eqnarray}
            H_e\left (s, \si^+_{(1-\rho_0)a +\rho_0 b}(e,s) \right ) &=& (1-\rho_0) \, a +\rho_0 \, b = (1-\rho_0) \, H_e(s,\si^+_a(e,s)) +\rho_0 \,H_e(s, \si^+_b(e,s)) \nonumber\\
            &>& H_e(s, (1-\rho_0) \, \si^+_a(e,s)+\rho_0\, \si^+_b(e,s)).\label{newborn0}
        \end{eqnarray}
        Taking into account that $H_e(s,\cdot)$ is increasing in the interval $(\si^+_a(e,s), + \infty)$, the inequality in \eqref{newborn0} yields
        \begin{equation*}
            \si^+_{(1-\rho_0)a +\rho_0 b}(e,s) > (1-\rho_0) \, \si^+_a(e,s)+\rho_0\, \si^+_b(e,s).
        \end{equation*}
        By integrating the above relation over $[0,1]$, we finally get
        \begin{equation*}
            \si(e,(1-\rho_0)a +\rho_0 b) > (1-\rho_0) \,\si(e,a)+\rho_0\, \si(e,b),
        \end{equation*}
        which accounts for the strictly concave character of $\si(e,\cdot)$. We proceed showing that $ a \mapsto \si^+_a(e,s)$ is differentiable in $(a_e,+\infty)$ for any $s \in [0,1]$. Given $s \in [0,1]$ and an increment $r \in \R$, we have by Lagrange Theorem
        \begin{equation*}
            H_e(s,\si^+_{a+r}(e,s)) - H_e(s,\si^+_{a}(e,s)) = \frac \partial{\partial \rho} H_e(s,\rho_0) \, \big ( \si^+_{a+r}(e,s) - \si^+_{a}(e,s) \big )
        \end{equation*}
        for a suitable $\rho_0 \in [\si^+_{a}(e,s), \si^+_{a+r}(e,s)]$. Taking into account that $\frac \partial{\partial \rho} H_e(s,\rho_0) > 0$, we derive that
        \begin{equation*}
            \frac{\si^+_{a+r}(e,s) - \si^+_{a}(e,s)}r = \frac 1{\frac \partial{\partial \rho} H_e(s,\rho_0)}
        \end{equation*}
        and, sending $r$ to $0$, we get
        \begin{equation*}
            \lim_{r \to 0} \frac{\si^+_{a+r}(e,s) - \si^+_{a}(e,s)}r = \frac 1{\frac \partial{\partial \rho} H_e(s,\si^+_a(e,s))}.
        \end{equation*}
        Then, we prove that $a \mapsto \si(e,a)$ is of class $C^1$ by differentiating under the integral sign. Finally, to show the limit relation in the statement, we exploit the uniform superlinearity assumption on $H_e$, which reads
        \begin{equation}\label{newborn1}
            \lim_{r \to + \infty} \min \left \{ \frac{H_e(s,\rho)}\rho \mid \rho > r, \, s \in [0,1] \right \} = + \infty.
        \end{equation}
        It is clear from the definitions in \eqref{deffi} and the superlinearity of $H_e(s,\cdot)$ that $\frac{\si(e,a)}a$ is positive for $a$ large enough. Assume by contradiction that there exist a sequence $a_n \to + \infty$ and a positive $M$ with
        \begin{equation*}
            \lim_n \frac{\si (e,a_n)}{a_n} > M.
        \end{equation*}
        Consequently, for any $n$, there are $s_n \in [0,1]$, $\rho_n \in \R$ with
        \begin{equation*}
            H_e(s_n,\rho_n)=a_n \quad\text{and} \quad \frac{ \rho_n }{a_n}> M;
        \end{equation*}
        we deduce that
        \begin{equation*}
            \rho_n \to + \infty \quad\text{and} \quad \frac{H_e(s_n,\rho_n)}{\rho_n} < \frac 1M,
        \end{equation*}
        which is in contradiction with \eqref{newborn1}.
    \end{proof}

    \smallskip

    \begin{Definition}
        For every $e\in \EE_0$, we define the {\it discrete Hamiltonian} $\HH(e,\cdot)$ associated to $H_e$ as the inverse, with respect to the composition, of the map $a \mapsto \si(e,a)$:
        \begin{equation*}
            \HH(e, \cdot): [b_e,+ \infty) \to [a_e,+\infty),
        \end{equation*}
        where $b_e= \si(e,a_e)$.
    \end{Definition}

    We directly derive from Lemma \ref{newborn}:

    \begin{Proposition}\label{acca}
        $\HH(e, \cdot)$ is of class $C^1$ in $(b_e,+\infty)$, strictly increasing, strictly convex and superlinear.
    \end{Proposition}

    \smallskip

    Similarly, we can define the discrete Lagrangian.

    \begin{Definition}
        For every $e\in \EE_0$, we define the {\it discrete Lagrangian} $\LL(e,\cdot)$ associated to $H_e$ by Fenchel duality:
        \begin{equation}\label{deflag}
            \LL(e,\la) := \max_{\rho \in [b_e, + \infty)} \left(\rho \, \la - \HH(e,\rho)\right) \qquad\text{for $\la \ge 0$}
        \end{equation}
    \end{Definition}

    \medskip

    \begin{Remark}
        Note that
        \begin{equation*}
            \LL(e,0) = - a_e \qquad\text{for all $e \in \EE_0$},
        \end{equation*}
        and consequently by Remark \ref{minus}

        \begin{equation}\label{superminus}
            \LL(e,0)= \LL(-e,0) =- a_e= -a_{-e} \qquad\text{for all $e \in \EE_0$.}
        \end{equation}
    \end{Remark}

    \medskip

    We can put in relation the maximal flux limiter $\wha c_z$, for $z \in \VN$, and $\LL(e,\cdot)$, for $e \in \EE_0$:

    \begin{Lemma}\label{flusso}
        We have
        \begin{equation*}
            \wha c_z= \min - a_e= \min -a_{-e}= \min \LL(e,0) =\min \LL(-e,0) \txt{for any $z \in \VN$,}
        \end{equation*}
        where the minima in the above formula are over the edges $e$/$-e$ in $\EE_0$ with $(\Psi_\EN^{-1}((e,h)) \in \NN^z)$/$(\Psi_\EN^{-1}((-e,h)) \in \NN^z)$ for some $h \in \Z^{b(\G_0)}$, see Notation \ref{pi}.
    \end{Lemma}
    \begin{proof}
        According to \eqref{a00}, we have for any $e \in \EE_0$
        \begin{eqnarray*}
            -a_{-e} &=& - \left [ \max_s \, \min_\rho H_{-e}(s,\rho) \right ] = \min_s - \left [\min_\rho H_{-e}(s,\rho) \right] \\ &=& \min_s \max_\rho - H_{-e}(s,\rho) = \min_s L_{-e}(s,0).
        \end{eqnarray*}
        Minimizing the above quantity over the edges $-e$ with $\Psi_\EN^{-1}((-e,h)) \in \NN^z$ for some $h \in \Z^{b(\G_0)}$ is equivalent to consider
        \begin{equation*}
            \min_{\ga \in \EN^z} \, \left [\min_s L_\ga(s,0) \right ]
        \end{equation*}
        which is in turn equal to $\wha c_z$ thanks to \eqref{maxxi}.
    \end{proof}

    \medskip

    As a consequence of Proposition \ref{acca} and the properties of Fenchel transform (see for instance \cite{Rockafellar}), we have:

    \begin{Proposition}
        Given $e \in\EE_0$, the function $\la \mapsto \mathcal L(e,\la)$ from $[0,+\infty)$ to $\R$ is strictly convex, of class $C^1$ in $(0,+\infty)$ and superlinear as $\la$ goes to $+ \infty$.
    \end{Proposition}

    \smallskip

    We record for later use, the following property:

    \begin{Lemma}\label{inutile}
        For any $\la > 0$ we have
        \begin{equation*}
            \LL(e,\la)= \max_{a\in [a_e,+\infty)} \left( \la \,\si(e,a) - a \right).
        \end{equation*}
    \end{Lemma}
    \begin{proof}
        By the very definition of $\HH$, we have
        \begin{equation*}
            \HH(e,\rho)= a \quad \Longleftrightarrow \quad \rho= \si(e,a).
        \end{equation*}
        Therefore, given $\la > 0$, we get
        \begin{equation}\label{inutile1}
            \rho \, \la - \HH(e,\rho)= \si(e,a) \, \la - a \txt{for a suitable $a \in [a_e,+\infty)$.}
        \end{equation}
        Conversely, given $a\in [a_e,+\infty)$, \eqref{inutile1} is satisfied for a suitable $\la >0$. This proves the assertion.
    \end{proof}

    \medskip

    The next result will {provide a relation between the Lagrangian on the network and the discrete Lagrangian on the base graph, which can be interpreted as a sort of {\it minimal average action} on the arc corresponding to the edge}. This observation will be relevant for studying the action on $\G_0$ and to relate it to the action on the periodic network $ \NN$. We will crucially use it in the proof of Proposition \ref{lemnew}.

    \begin{Theorem}\label{babasic}
        We fix an edge $e$ in $\G_0$. Given any positive time $T$, we have
        \begin{equation}\label{basic}
            T \, \LL(e,1/T) = \min \left \{ \int_{0}^{T} L_e(\xi, \dot\xi) \, d\tau \mid \xi(0)=0, \, \xi(T)=1 \right \},
        \end{equation}
        where the minimum is over the curves $\xi: [0,T] \to [0,1]$.
    \end{Theorem}
    \begin{proof}
        According to \cite[Theorem 3.16]{Davini07}, there is an $a\in[a_e,\infty)$ such that
        \begin{equation*}
            \min\left\{\int_0^TL_e(\zeta,\dot\zeta)\,d\tau\right\}=\sigma(e,a)-aT,
        \end{equation*}
        where the minimum is over the curves $\zeta:[0,T]\to[0,1]$ with $\zeta(0)=0$, $\zeta(T)=1$ and $\dot\zeta\ge0$ a.e.. Lemma \ref{inutile} then yields
        \begin{equation}\label{eq:babasic1}
            \min\left\{\int_0^TL_e(\xi, \dot\xi)\,d\tau\,|\,\xi(0)=0,\,\xi(T)=1\right\}\le T\LL(e,1/T).
        \end{equation}
        Next we observe that by definition, for any $\lambda\in\R$ and $s\in[0,1]$,
        \begin{equation*}
            L_e(s,\la)=\max_\rho(\rho\la-H_e(s,\rho))=\max_{b\in[a_e,\infty)}\left(\max\left\{\la\sigma^+_b(e,s),-\la\sigma^+_b(-e,1-s)\right\}-b\right).
        \end{equation*}
        Therefore, for a fixed curve $\xi$ as in the statement and $b\in[a_e,\infty)$, we get
        \begin{equation}\label{eq:babasic2}
            \int_0^TL_e(\xi,\dot\xi)d\tau\ge\int_0^T\dot\xi\sigma^+_b(e,\xi)d\tau-bT.
        \end{equation}
        It is known, see for instance \cite[Lemma 3.11]{Davini07}, that $\xi$ is the reparametrization of a curve with constant speed defined in $[0,T]$, that we denote by $\eta$. We thus get
        \begin{equation}\label{eq:babasic3}
            \int_0^T\dot\xi \, \sigma^+_b(e,\xi)d\tau=\int_0^T\dot\eta \,\sigma^+_b(e,\eta)\,d\tau.
        \end{equation}
        Since it is apparent that
        \begin{equation*}
            \int \dot\eta_0\,\sigma^+_b(e,\eta_0)d\tau=0
        \end{equation*}
        whenever $\eta_0$ is a closed curve with constant speed, we infer from \eqref{eq:babasic3}, possibly removing all the cycles in $\eta$, that
        \begin{equation}\label{eq:babasic4}
            \int_0^T\dot\xi\sigma^+_b(e,\xi)d\tau = \int_0^T\dot\eta\sigma^+_b(e,\eta)\, d\tau=\int_0^1\sigma^+_b(e,s)ds.
        \end{equation}
        Finally, we get from \eqref{eq:babasic2}, \eqref{eq:babasic3} and \eqref{eq:babasic4} that
        \begin{equation}\label{eq:babasic5}
            \int_0^T L_e(\xi,\dot\xi)d\tau\ge\sigma(e,b)-bT.
        \end{equation}
        The equality in \eqref{basic} is then a consequence of \eqref{eq:babasic1}, \eqref{eq:babasic5} and Lemma \ref{inutile}.
    \end{proof}

    \medskip

    \subsection{Discrete action functional on \texorpdfstring{$\G_0$}{Γ₀}}

    To define a {discrete} action functional on $\G_0$, we need introducing a notion of parametrized path, namely to assign to any edge of a given path a non-negative {\it average speed} and a corresponding positive {\it time} needed to go through it, see \cite{SicSor1}. To keep a relation between minimization problems involving the action in $\G_0$ and $ \NN$, we prescribe, in view of \eqref{basic}, that time is the inverse of the speed, if the latter is positive. It can be instead any positive number if the speed is vanishing.

    \begin{Definition}
        We say that $\xi=(e_i,q_i,T_i)_{i=1}^m$ is a {\it parametrized path} if
        \begin{itemize}
            \item[{\bf (i)}] $(e_i)_{i=1}^m$ is a family of concatenated edges in $\G_0$, which is called the {\em support} of $\xi$;
            \item[{\bf (ii)}] the $q_i$'s are non-negative numbers and
            \begin{equation*}
                T_i = \left \{
                \begin{array}{cc}
                    \frac 1{q_i} & \quad\text{if $q_i > 0$} \\
                    \text{a positive constant} & \quad\text{if $q_i = 0$};
                \end{array}
                \right .
            \end{equation*}
            we denote by $ T_\xi:=\sum_i T_i$ the {\em total time} of the parametrization of $\xi$;
            \item[{\bf (iii)}] if $q_i >0$ for every $i$, we call the parametrization {\em nonsingular}. We further call {\em equilibrium circuit} a path of the form
            \begin{equation*}
                (e, -e,0,0,T_1,T_2) \qquad\text{for any $e \in \EE_0$, $T_1$, $T_2$ positive constants;}
            \end{equation*}
            \item[{\bf (iv)}] a general parametrized path is any concatenation of equilibrium paths and paths with nonsingular parametrization.
        \end{itemize}
        By {\em parametrized cycle} we mean a parametrized path supported on a cycle, by {\em parametrized circuit} a parametrized path supported on a circuit.
    \end{Definition}

    \medskip

    \begin{Remark}
        An equilibrium circuit represents a steady state, it should be thought as a fluctuation with zero speed along an edge and its opposite, entering and exiting from the same starting point of the edge.
    \end{Remark}

    \medskip

    \begin{Definition}
        Given a parametrized path $\xi=(e_i,q_i,T_i)_{i=1}^m$, we define the {\it (discrete) action functional} on $\xi$ as
        \begin{equation*}
            A_\LL(\xi):= \sum_{i=1}^m T_i \, \LL(e_i,q_i).
        \end{equation*}
    \end{Definition}

    \begin{Definition}\label{defPhiT}
        We consider the problem of minimizing $\A_\LL$ among the parametrized paths with total time of parametrization $T$ and rotation vector $h \in \Z^{b(\G_0)}$ linking two given vertices $ x$, $ y$. We denote by $\Phi_\LL( x, y,T;h)$ the corresponding {\it (discrete) minimal action functional}.
    \end{Definition}

    {In the following, in the same spirit fo Aubry--Mather theory (see \cite{SicSor1}), we will associate to the discrete action functional a variational problem defined on a suitable space of measures, which explains the necessity of considering average quantities. To find minimizers, we pass to a relaxed problem that we are going to describe in the next section. Such minimizers are related to circuits in $\G_0$, up to convex combinations, to be understood in the sense of measures.}

    \bigskip

    \subsection{Stationary Hamilton--Jacobi equations on \texorpdfstring{$\G_0$}{Γ₀}}\label{stattHJ}

    Given $p \in H^1(\G_0,\R)= \R^{b(\G_0)}$, we consider the $p$--modified Hamiltonian
    \begin{equation*}
        \HH^p(e,\rho):=\HH(e,\rho + \langle p, \theta(e) \rangle) \quad\text{for $\rho \in [\si(e,a_e) - \langle p, \theta(e) \rangle, + \infty)$,}
    \end{equation*}
    where $\langle \cdot, \cdot \rangle$ denotes the standard scalar product in $\R^{b(\G_0)}$. We set
    \begin{equation*}
        \si^p(e,a) := \si(e,a) - \langle p, \theta(e) \rangle \qquad\text{for any $e \in \EE_0$, $a \ge a_0$}.
    \end{equation*}
    Given a path $\xi=(e_i)_{i=1}^M$ in $\G$, we define the intrinsic length of $\xi$ corresponding to $p$ and $a$ as
    \begin{equation*}
        \si_a^p(\xi):= \sum_{i=1}^M \si^p(e,a).
    \end{equation*}
    Note that the above definition does not require $\xi$ to have a parametrization.

    Given $p \in \R^{b(\G_0)}$, we further consider the following family of equations defined on $\VV_0$:
    \begin{equation}\label{HJE}
        \max_{- e\in\EE_x} \HH^p(e,u(\tt(e)) - u(\oo(e)))= a \qquad\text{for $a \ge a_0$}.
    \end{equation}

    \begin{Definition}
        A function $u:\VV_0 \to \R$ is a {\it solution} to \eqref{HJE} if the above equality holds for any vertex $x$. $u$ is instead a {\it subsolution} to \eqref{HJE} if $=$ is replaced by $\le$.
    \end{Definition}

    The analysis of \eqref{HJE} is performed in \cite{SicSor} with the extra assumption that
    \begin{itemize}
        \item[($H_0$)] For any edge $e$ with $a_e=a_0$ the map $s \mapsto \min_\rho H_\ga(s,\rho)$ is constant.
    \end{itemize}

    \smallskip

    The following characterization is independent of ($H_0$).

    \begin{Proposition}
        \emph{\cite[Propositions 6.14, 6.15]{SicSor}} Equation \eqref{HJE} admits solutions in $\VV_0$ if and only the corresponding intrinsic length is nonnegative for all cycles and vanishing for some of them.
    \end{Proposition}

    \smallskip

    \begin{Remark}\label{accazero}
        Condition ($H_0$) is effective only at $a=a_0$: it guarantees that the cycle $(e,-e)$ has vanishing intrinsic length $\si_{a_0}^p(\cdot)$ for any $p$, and any $e$ with $a_e=a_0$; this condition implies the existence of solutions thanks to the previous proposition, provided that the intrinsic length of the other cycles is nonnegative. We can exhibit an example, with ($H_0$) failing, where equation \eqref{HJE} has not solutions for any $a$. Take a finite network $ \NN_0$ in $\R^N$ and consider the following Hamiltonian on $ \NN_0$
        \begin{equation*}
            H_\ga(s,\rho)= |\rho|^2 -f(s) \qquad\text{for any arc $\ga$,}
        \end{equation*}
        where $f$ is a continuous potential, defined in $[0,1]$, strictly positive everywhere except at a point where it vanishes; it is clear that $a_0=0$. When we transfer the Hamiltonians to the underlying graph, we get that all the intrinsic lengths for $a=0$ and $p=0$ are strictly positive and the same clearly holds true for $a >0$, which proves the claim of nonexistence of solutions for $a \ge 0$.
    \end{Remark}

    \smallskip

    We can state the following result:

    \begin{Theorem}
        Given $ p \in \R^{b(\G_0)}$, there exists at most one value of $a$, called critical value, for which the equation \eqref{HJE} admits solutions in $\VV_0$.
    \end{Theorem}
    \begin{proof}
        The unique value $a$ where there could be solutions to \eqref{HJE} is
        \begin{equation*}
            \inf\{a \ge a_0 \mid \;\text{\eqref{HJE} admits subsolutions}\},
        \end{equation*}
        see \cite[Proposition 6.5]{SicSor}. If ($H_0$) holds then a solution for the equation corresponding to the critical value does exist according to \cite[Theorem 6.16]{SicSor}. If instead ($H_0$) is not assumed, this existence result holds only if the critical value is strictly greater than $a_0$.
    \end{proof}

    \smallskip

    \begin{Definition}\label{defHbar}
        We define a function $\bar \HH: H^1(\G_0,\R) \to \R$ associating to $p$ the unique value for which \eqref{HJE} admits solution if such a value does exist, or $a_0$ if no equation of the family has solution.
    \end{Definition}

    \bigskip

    \section{Relaxed action-minimizing problems}\label{secrelaxed}

    In this section we introduce a relaxed version of the action-minimizing problems defined on curves, namely we want to define related variational problems on the suitable sets of probability measures.

    \subsection{Probability measures on the tangent cone \texorpdfstring{$T^+\G_0$}{T⁺Γ}}

    We start by introducing the ambient space, namely the notion of tangent cone of $\G_0$.

    \begin{Definition}
        The {\it tangent cone} of $ \G_0$ is defined as
        \begin{equation*}
            T^+ \G_0 = \EE_0\times \R^+ / \sim,
        \end{equation*}
        where $\R^+=[0,+\infty)$ and $\sim$ is the identification $(e,0)\sim (-e,0)$.\\
        We denote each fiber over $e$ by $\R^+_e :=\{e\}\times \R^+$.
    \end{Definition}

    We endow $T^+ \G_0$ with a distance defined as:
    \begin{eqnarray*}
        d((e_1,q_1), (e_2,q_2)) :=\left\{
        \begin{array}{ll}
            q_1+q_2+1 & \qquad\text{if $e_1\ne \pm e_2$}\\
            q_1+q_2 & \qquad\text{if $e_1=-e_2$}\\
            |q_1-q_2| & \qquad\text{if $e_1=e_2$}.
        \end{array}
        \right.
    \end{eqnarray*}
    This makes $T^+ \G_0$ a Polish space. A set $A$ is open in $T^+ \G_0$, with respect to the induced topology, if and only if $A \cap \R^+_e$ is open in the natural topology of $\R^+$ for any $e\in \EE_0$. Accordingly, $F$ is a Borelian set of $T^+ \G_0$ if and only if $F \cap \R^+_e$ is Borelian in $\R_e^+$ for any edge $e\in \EE_0$.

    \begin{Definition}
        Given a Borel probability measure $\mu$ on $T^+ \G_0$, we define its {\it support } as the set
        \begin{equation*}
            \supp_{\EE_0} \mu= \{e \in \EE_0 \mid \mu(\R^+_e) >0 \}.
        \end{equation*}
    \end{Definition}

    \smallskip

    \begin{Remark}\label{postdef}
        Note that $\mu$ can be written in the form
        \begin{equation}\label{postdef1}
            \mu = \sum_{e \in \supp_{\EE_0} \mu} \la_e \mu_e,
        \end{equation}
        where $\la_e >0$, $\sum_e \la_e =1$, and $\mu_e$ is the restriction of $\mu$ to $\R^+_e$, normalized so to make it a probability measure. The $\mu_e$'s make up a family of Borel probability measures on $\R^+_e$, for $e$ varying in $\supp_{\EE_0} \mu$, satisfying the condition
        \begin{equation*}
            \pm e \in \supp_{\EE_0} \mu \quad \Longrightarrow \quad \mu_e(0)= \mu_{-e}(0).
        \end{equation*}
    \end{Remark}

    \medskip

    \begin{Notation}
        Sometimes we will write for short in place of \eqref{postdef1}
        \begin{equation*}
            \mu = \sum_{e \in \EE_0} \la_e \mu_e \txt{or} \qquad \mu = \sum_e \la_e \mu_e
        \end{equation*}
        meaning that all the terms corresponding to $e \not\in \supp_{\EE_0} \mu$ are vanishing.

        We will further write $\spt \mu_e$ to indicate the usual support of $\mu_e$ as measure on $\R^+$.

    \end{Notation}

    \medskip

    A Borel probability measure $\mu = \sum_{e\in \spt_{\EE_0} \mu} \la_e \, \mu_e$ has finite first momentum if and only such property holds for any $\mu_e$, namely
    \begin{equation*}
        \int_0^{+\infty} q \, d\mu_e < +\infty \txt{for any $e \in \EE_0$.}
    \end{equation*}

    \smallskip

    We denote by $\mathbb P$ the family of Borel probability measures on $T^+ \G_0$ with finite first momentum and we endow it with the (first) Wasserstein distance (see, for example, \cite{Villani}). The corresponding convergence of measures can be expressed in duality with continuous functions $F(e,q)$ on $T^+ \G_0$ possessing linear growth at infinity; namely, given a sequence $\{\mu_n\}_n$ and $\mu$ in $\mathbb P$
    \begin{equation*}
        \mu_n \to \mu \quad \Longleftrightarrow \quad \int F(e,q) \, d \mu_n \to \int F(e,q) \, d \mu
    \end{equation*}
    for any function $F$ continuous on $T^+\G_0$ such that for any $e\in \EE_0$ there exist $a_e, b_e \in \R$ such that
    \begin{equation*}
        \forall\, e\in \EE_0 \quad \exists\, a_e, b_e \in \R: \qquad |F(e,q)| \le a_e \, q + b_e \qquad \forall\, q \ge 0.
    \end{equation*}

    \medskip

    \subsubsection{Occupation and closed measures}

    \begin{Definition}
        Given a parametrized path $\xi=(e_i,q_i,T_i)_{i=1}^m$, we define a corresponding probability measure on $T ^+\G_0$ as
        \begin{equation*}
            \mu_\xi= \frac 1{T_\xi} \, \sum_{i=1}^m T_i \, \de(e_i, q_i),
        \end{equation*}
        where $T_\xi = \sum_{i=1}^m T_i$ is the total time of parametrization, and $\delta(e,q)$ denotes the Dirac delta concentrated on the point $(e,q)$. We call it the {\it occupation measure supported on $\xi$}.
    \end{Definition}

    \smallskip

    We recover the usual formula for measures in $\bP$
    \begin{equation*}
        \mu_\xi= \sum_e \la^\xi_e \mu^\xi_e,
    \end{equation*}
    see Remark \ref{postdef}, setting
    \begin{eqnarray}
        \la^\xi_e &=& \frac 1{T_\xi} \, \sum_{i \in I(e)} T_i\label{chiudi1} \\
        \mu^\xi_e &=& \frac 1{\sum_{i \in I(e)} T_i} \sum_{i \in I(e)} T_i \, \de(e_i, q_i),\label{chiudi2}
    \end{eqnarray}
    where
    \begin{equation*}
        I(e) = \{i = 1, \cdots m \mid e=e_i\} \qquad\text{for $e \in \EE_0$.}
    \end{equation*}

    \begin{Definition}\label{Defrhoxi}
        If $\xi$ is a cycle and $\mu_\xi$ is the associated occupation measure supported on $\xi$, we define the {\it rotation vector of $\mu_\xi$} as
        \begin{equation*}
            \rho(\mu_\xi) := \frac{[\xi]}{T_\xi},
        \end{equation*}
        where $\rho(\mu_\xi) \in H_1(\G_0,\R)= \R^{b(\G_0)}$ by Lemma \ref{cy}.
    \end{Definition}

    \begin{Remark}
        $\rho(\mu_\xi)$, as in Definition \ref{Defrhoxi}, admits the following integral representation
        \begin{equation}\label{ruota}
            \rho(\mu_\xi)= \sum_{e\in \EE_0^+} \left ( \la^\xi_e \int q \, d\mu^\xi_e - \la^\xi_{-e} \int q \, d \mu^\xi_{-e} \right ) \, e,
        \end{equation}
        \\
        where $\la^\xi_e$, $\mu^\xi_e$ are defined as in \eqref{chiudi1}, \eqref{chiudi2}.
    \end{Remark}

    \smallskip

    \begin{Notation}
        We denote by $\bM$ the closure of the set of occupation measures of cycles in $\mathbb P$ with the topology induced by the (first) Wasserstein distance. It is convex, according to \cite[Proposition 4.9]{SicSor1}. We call {\it closed} any measure belonging to it.
    \end{Notation}

    \smallskip

    Taking into account that $(e,0)$ and $(-e,0)$ are identified in $T^+\G_0$, $\de(e,0) = \de(-e,0)$ is the closed occupation measure corresponding to the equilibrium circuit based on $e$, for any $e \in \EE_0$.

    \smallskip

    If $\mu = \sum_e \la_e \mu_e$ is a closed measure we define its rotation vector $\rho(\mu)$ via the right hand-side of \eqref{ruota} with obvious replacements, namely
    \begin{eqnarray*}
        \rho(\mu)&=& \sum_{e\in \EE_0^+} \left ( \la_e \int_0^{+\infty} q \, d\mu_e - \la_{-e} \int_0^{+\infty} q \, d \mu_{-e} \right ) \, e \\
        &=& \sum_{e\in \EE_0^+} \left ( \la_e \int_0^{+\infty} q \, d\mu_e - \la_{-e} \int_0^{+\infty} q \, d \mu_{-e} \right ) \, \theta(e).
    \end{eqnarray*}

    \smallskip

    \smallskip

    \begin{Proposition}\label{proprho}
        \emph{\cite[Proposition 4.10]{SicSor1}} The map $\rho: \bM \to H_1(\G_0,\R)$ is continuous and affine (for convex combinations), {\it i.e.}, for every $\lambda\in [0,1]$ and $\mu_1, \mu_2 \in \bM$
        \begin{equation*}
            {\rho}\left(\la \mu_1 + (1-\la)\mu_2 \right) = \la{\rho}(\mu_1) + (1-\la) {\rho}(\mu_2).
        \end{equation*}
        In particular, it is surjective.
    \end{Proposition}

    \medskip

    Exploiting the continuity property asserted in Proposition \ref{proprho}, we see that
    \begin{equation*}
        \rho(\mu) \in H_1(\G_0,\R) = \R^{b(\G_0)} \txt{for any $\mu \in \mathbb M$}.
    \end{equation*}
    A converse property is true as well, we in fact have:

    \begin{Proposition}\label{proprhobis}
        A probability measure $\mu$ is closed if and only if
        \begin{equation*}
            \sum_{e\in \EE_0^+} \left ( \la_e \int_0^{+\infty} q \, d\mu_e - \la_{-e} \int_0^{+\infty} q \, d \mu_{-e} \right ) \, e \in H_1(\G_0,\R) = \R^{b(\G_0)}.
        \end{equation*}
    \end{Proposition}

    See \cite{SicSor}.

    \subsection{Action of a measure}

    \begin{Definition}
        Given a probability measure $\mu = \sum_{e \in \EE_0} \la_e \, \mu_e$, we define its {\it action} as
        \begin{equation*}
            A(\mu)= \int \LL \, d\mu = \sum_{e \in \spt_{\EE_0}\mu} \int_0^{+\infty} \LL(e,q) \, d\mu_e(q).
        \end{equation*}
        Note that it can be infinite.
    \end{Definition}

    \medskip

    \begin{Remark}
        If $\de(e,0) = \de(-e,0)$, for $e\in \EE_0$, denotes the closed occupation measure corresponding to the equilibrium circuit supported on $e$ (recall that $(e,0)$ and $(-e,0)$ are identified in $T^+\G_0$), then \eqref{superminus} implies
        \begin{equation}\label{uffa}
            A(\de(e,0)) = - a_e \qquad\text{for any $e \in \EE_0$,}
        \end{equation}
        where $a_e$ is defined as in \eqref{a00}.
    \end{Remark}

    \medskip

    Given $h \in \R^{b(\G_0)}$, we consider the variational problem
    \begin{equation}\label{relax}
        \inf \{A(\mu) \mid \mu \in \bM \;\text{with} \; \rho(\mu) =h\}
    \end{equation}
    which can be thought as a timeless version of the average minimal action problem for paths in $\G_0$ {(or equivalently $\Gamma$)} with rotation vector $h$ linking two given vertices. Note that in this relaxed variational problem, initial and final points are irrelevant, and only the rotation vector plays a role.

    \begin{Definition}\label{defbeta}
        We indicate by $\be: H_1(\G_0,\R) \to \R$, the {\it Mather's $\be$--function}, the value function of \eqref{relax}:
        \begin{eqnarray*}
            \be: H_1(\G_0,\R) &\longrightarrow& \R \\
            h &\longmapsto& \inf \{A(\mu) \mid \mu \in \bM \;\text{with} \; \rho(\mu) =h\}.
        \end{eqnarray*}
    \end{Definition}

    We have:

    \begin{Lemma}
        \emph{\cite[Section 5.2]{SicSor1}} The function $\be$ is convex and superlinear.
    \end{Lemma}

    \medskip

    Since $\beta$ is convex and superlinear, we can consider its convex conjugate by Fenchel duality, that will be also relevant in the upcoming discussion.

    \begin{Definition}\label{defalpha}
        We call {\it Mather $\al$--function} the Fenchel transform of $\beta$:
        \begin{eqnarray*}
            \al: H^1(\G_0,\R) &\longrightarrow& \R \\
            p &\longmapsto& \sup_{h\in H_1(\G_0,\R)} \left( \langle p, h\rangle - \beta(h)\right),
        \end{eqnarray*}
        where $\langle \cdot, \cdot \rangle$ indicates the scalar product in $H^1(\G_0,\R)= \R^{b(\G_0)}$.
    \end{Definition}

    \medskip

    \begin{Lemma}
        \emph{\cite[Section 5.2]{SicSor1}} The function $\al$ is also convex and superlinear.
    \end{Lemma}

    \smallskip

    \begin{Definition}
        We say that $h \in H_1(\G_0,\R)= \R^{b(\G_0)}$ and $p \in H^1(\G_0,\R)=\R^{b(\G_0)}$ are {\em{conjugate}} if
        \begin{equation*}
            \langle p,h\rangle= \al(p) + \be (h),
        \end{equation*}
        where $\langle \cdot, \cdot \rangle$ indicates the scalar product in $\R^{b(\G_0)}$.
    \end{Definition}

    \smallskip

    Problem \eqref{relax} admits minimizers thanks to the fact that $\bM$ is closed in $\mathbb P$ by construction and to the following crucial property:

    \begin{Lemma}\label{compatto}
        \emph{\cite[Proposition 5.1]{SicSor1}} Given $ a \in \R$, the set
        \begin{equation*}
            \{ \mu \in \bM \mid A(\mu) \le a\}
        \end{equation*}
        is compact in $\bP$.
    \end{Lemma}

    \begin{Notation}
        We denote the set of minimizers of \eqref{relax} by $\bM^h$, and call its elements {\em Mather measures} with rotation vector $h$.
    \end{Notation}

    Any element of $\bM^h$ is an occupation measure, up to a convex combination. More precisely:

    \begin{Proposition}\label{combi}
        \emph{\cite[Proposition 6.1, Theorem 6.6]{SicSor1}} Given $h \in \R^{b(\G_0)}$, any $\mu \in \bM^h$ is the convex combination of occupation measures supported on circuits.
    \end{Proposition}

    \medskip

    To characterize the measures in the above convex combination we introduce the family of variational problems, indexed by $p\in H^1(\G_0,\R)$, that are in some sense dual to \eqref{relax} (and can be thought as a Lagrange multiplier version of the previous constrained minimization problems). The corresponding value function coincides with $-\alpha$.

    For $p \in H^1(\G_0,\R)$ let us define the {\it $p$--modified Lagrangian} $\LL^p$ as
    \begin{equation*}
        \LL^p(e,\la):= \LL(e,\la) - \la \langle p, \theta(e) \rangle \qquad\text{for $\la \ge 0$,}
    \end{equation*}
    and consider the problem:
    \begin{equation}\label{relax0}
        \inf \left \{\int \LL^p \, d\mu \mid \mu \in \bM \right \}.
    \end{equation}

    \begin{Remark}
        Note that $\LL^p$ is the Fenchel--Legendre transform the $p$--modified Hamiltonian $\HH^p$ defined in Section \ref{stattHJ}.
    \end{Remark}

    \smallskip

    We notice that \eqref{ruota} yields
    \begin{align}
        -\inf_{\mu\in\bM}\int\LL^pd\mu=\,&\sup_{\mu\in\bM}\left(\langle p,\rho(\mu)\rangle-\int\LL d\mu\right)=\sup_{h\in H_1(\G_0,\R)}\left(\langle p,h\rangle-\inf_{\mu\in\bM^h}\int\LL d\mu\right) \nonumber\\
        =\,&\sup_{h\in H_1(\G_0,\R)}(\langle p,h\rangle-\be(h))=\al(p).\label{ciaociao}
    \end{align}

    \smallskip

    \begin{Proposition}\label{effetto}
        The function $\al$ coincides with the function $\bar \HH$ defined in Definition \ref{defHbar}.
    \end{Proposition}
    \begin{proof}
        If $\bar \HH(p) > a_0$ then the statement is proven in \cite[Theorem 8.1]{SicSor1}, see Remark \ref{accazero}. Note that by \eqref{uffa}
        \begin{equation*}
            A((e,\de_0)) = -a_0 \qquad\text{for any $e$ with $a_e=a_0$}
        \end{equation*}
        which implies that $\min \al = a_0 =\min \bar \HH$. This concludes the proof.
    \end{proof}

    \medskip

    We deduce from Lemma \ref{compatto} and \eqref{ciaociao} that \eqref{relax0} admits minimizers making up a convex compact set in $\bM$ denoted by $\bM_p$, see \cite[Proposition 5.7]{SicSor1}.

    We are now in position to complete the information provided in Proposition \ref{combi}.

    \begin{Proposition}\label{utilissima}
        \cite[Proposition 5.8, Theorems 6.10, 8.1]{SicSor1} Let $h \in H_1(\G_0,\R)$ and let $\mu \in \bM^h$, {\it i.e.}, $\mu$ is a Mather measure with rotation vector $h$. Then:
        \begin{equation*}
            \mu \in \bM_p \qquad \Longleftrightarrow \qquad \text{$p$ is conjugate to $h$}.
        \end{equation*}
        Moreover, if this is the case, $\mu$ is the convex combination of occupation measures supported on circuits belonging to $\bM_p$.
    \end{Proposition}

    We point out that
    \begin{equation}\label{eq:minbe}
        \beta(0)=-\min_{p\in H^1(\G_0,\R)}\al(p)=-a_0.
    \end{equation}

    \medskip

    We further have:

    \begin{Proposition}\label{alfinale}
        Let $\mu \in \bM_p$, for some $p \in H^1(\G_0,\R)$, and $e \in \supp_{\EE_0} \mu$. Then, $\mu_e= \de(e,\mathcal Q_p(e))$, where
        \begin{equation}\label{alfinale1}
            \mathcal Q_p(e):= \frac \partial{\partial\rho} \HH(e, \si(e,\al(p))) \ge 0.
        \end{equation}
        If $Q_p(e)=0$ for some $e \in \spt_{\EE_0} \mu$, then $p$ is a minimizer of $\al$ and $a_e = \al(p) =a_0$.
    \end{Proposition}
    \begin{proof}
        This is a consequence of \cite[Proposition 8.6 and Corollaries 8.2, 8.7]{SicSor1}.
    \end{proof}

    \smallskip

    We record for later use:

    \begin{Corollary}\label{coralfinale}
        Given $R >0$, there exists $b >0$ such that if $p \in H^1(\G_0,\R)=\R^{b(\G_0)}$ with $|p| \le R$, then
        \begin{equation*}
            \mathcal Q_p(e) \le b \qquad\text{for any $\mu \in \bM_p$, $ e \in \spt_{\EE_0} \mu$.}
        \end{equation*}
    \end{Corollary}
    \begin{proof}
        If $p$ is bounded then $\al(p)$ is bounded as well because it is continuous, and the same holds true for $\frac \partial{\partial \rho} \HH(e, \si(e,\al(p)))$, whenever $ e \in \spt_{\EE_0} \mu$ for $\mu \in \bM_p$. The estimates are uniform in $e$ because the edges are finitely many. This proves the assertion by means of \eqref{alfinale1}.
    \end{proof}

    \medskip

    \medskip

    \subsection{Fundamental Mather's asymptotic Theorems}

    We can now state two key convergence theorems, which relate the asymptotic behaviours of, respectively, the minimal average action functionals on the base graph $\Phi_\LL$ (see Definition \ref{defPhiT}) and on the periodic network $\Phi_L$ (see Definition \ref{defPhiLnet}, to Mather's $\beta$-function, {\it i.e.}, the average minimal action functional (see Definition \ref{defbeta}). {These crucial results, which are interesting by themselves, are based on non-trivial adaptations in this setting of \cite[Corollary p. 181]{Mather91}.}

    The first key convergence theorem, as $T$ goes to infinity, reads:

    \begin{restatable}{Theorem}{new}\label{new}
        For any positive constants $A$, $\de$, there exists $T_0=T_0(A,\de)$ such that
        \begin{equation*}
            \left | \frac 1T \, \Phi_\LL(x,y,T;h) - \be \left ( \frac{h}T \right ) \right | < \de
        \end{equation*}
        for $T \ge T_0$, $x$, $y$ in $\VV_0$, and $h \in \Z^{b(\G_0)}$ with $\left | \frac{h}T \right | < A $.
    \end{restatable}

    The proof is given in Section \ref{newnew}.

    \medskip

    A similar asymptotic result is true for the minimal action between the vertices of $\NN$.

    \begin{Theorem}\label{melania}
        Given any positive constants $A$, $\de$, there exists $T_0=T_0(A,\de)$ such that
        \begin{equation*}
            \left | \frac 1T \, \Phi_L(z_1,z_2,T) - \be \left ( \frac{\pi_2(z_2)-\pi_2(z_1))}T \right ) \right | < \de
        \end{equation*}
        for $T \ge T_0$, $z_1$, $z_2$ in $ \VN$ with $\left | \frac{\pi_2(z_2)-\pi_2(z_1)}T \right | < A $.
    \end{Theorem}

    This result is consequence of Theorem \ref{new} and the following fact:

    \begin{restatable}{Proposition}{lemnew}\label{lemnew}
        We have
        \begin{equation*}
            \Phi_{L}(z_1,z_2,T)= \Phi_\LL(\pi_1(z_1),\pi_1(z_2),T;\pi_2(z_2) - \pi_2(z_1))
        \end{equation*}
        for any pair of vertices $z_1$, $z_2$ in $ \VN$, $T > 0$.
    \end{restatable}

    We postpone the proof of Proposition \ref{lemnew} to Section \ref{lemmanuovo}.

    \bigskip

    \section{Epilogue: proof of the Main Theorem}\label{epi}

    In this section we show the main homogenization result, namely the convergence, in an appropriate sense, of the solutions to the problems \eqref{HJeps} to that of \eqref{HJlim} coupled with a datum $g$, at $t=0$, suitably related to the $g_\eps$'s. A crucial preliminary step consists in identifying the {\it limit space} on which the average equation and its solutions are defined. Then, a subsequent problem will be dealing with the convergence of functions defined on different spaces.

    As we have mentioned, the rescaling in \eqref{HJeps} can be interpreted as a resizing of the metric by a factor $\eps$. The limit space will be accordingly given as the limit, in the Gromov--Hausdorff sense, of the rescaled metric spaces as $\eps$ goes to $0$, and it can be identified as what is called the {\it asymptotic cone} of $\VN$ endowed with the metric $d_\VN$. We will show that this cone is the space $\R^{b(\G_0)}$ with the appropriate norm. This means, roughly speaking, that the two spaces look alike when watched from the distance. In addition, we will introduce a family of transferring maps from $\VN$ to $\R^{b(\G_0)}$ playing a key role in the definition of convergence that we adopt in the homogenization procedure.

    \medskip

    \subsection{Asymptotic cone}\label{conogelato}

    Before discussing the definition of asymptotic cone, let us recall some preliminary notion.

    We start by recalling the notion of {\it distortion} of a map between two metric spaces, see \cite[Definition 7.1.4]{Burago}

    \begin{Definition}
        Given two metric spaces $(X,d_X)$, $(Y,d_Y)$ and a map $f: X \to Y$, not necessarily continuous, we call {\it distortion of $f$} the quantity
        \begin{equation*}
            \sup_{x_1,x_2 \in X} |d_X(x_1,x_2) - d_Y(f(x_1),f(x_2))|.
        \end{equation*}
    \end{Definition}

    We proceed giving the definition of {\it pointed Gromov--Hausdorff convergence}, see \cite[Definition 8.1.1]{Burago}.

    \begin{Definition}\label{pointed}
        The pointed metric spaces $( X_\eps, d_\eps, p_\eps)$ converge in the Gromov--Hausdorff sense to the pointed metric space $(X,d,p_0)$, with $p_\eps\in X_\eps$ for any $\eps$, $p_0 \in X$, if there exist maps
        \begin{equation*}
            f_\eps: X_\eps \to X
        \end{equation*}
        with $f_\eps(p_\eps)=p_0$, satisfying the following conditions: for any $\de > 0$, $r > 0$, there exist $\eps_0=\eps_0(\de,r)$ such that for every $\eps<\eps_0$
        \begin{itemize}
            \item[{\bf (i)}] $ \sup \, \{ |d(f_\eps(p_1),f_\eps(p_2)) - d_\eps(p_1,p_2)| \mid p_1, \, p_2 \in B_{X_\eps}(p_\eps,r)\} < \de$;
            \item[{\bf (ii)}] $\{p \in X \mid d(p, f_\eps(B_{X_\eps}(p_\eps,r))) \le \de\} \supset B_X(p_0,r -\de)$
        \end{itemize}
        where $B_X$, $B_{X_\eps}$ stand for the metric balls in $X$, $X_\eps$, respectively, with given centers and radii.
    \end{Definition}

    \smallskip

    Note that {\bf (i)} is a condition on the local distortion of $f_\eps$, while {\bf (ii)} is a sort of uniform quasi-surjectivity condition.

    \smallskip

    We can now define the asymptotic cone of a metric space.

    \begin{Definition}
        The {\it asymptotic cone} of a metric space $(X, d_X)$ is defined as the pointed Gromov–Hausdorff limit as $\eps \to 0$, provided that it does exist, of $( X, \eps d_X, p_0)$, with $p_0 \in X$.
    \end{Definition}

    We remark that the limit, if it exists, is actually unique up to isometries and independent of the choice of $p_0$, see \cite[Section 8.2.2]{Burago}. Here cone must be understood in a metric sense, namely it is a pointed metric space $(Y,d_Y,q_0)$ such that for any $\eps >0$ there exists an isometry $f: (Y,d_Y) \to (Y, \eps d_Y)$ with $f(q_0)=q_0$.

    \smallskip

    \medskip

    We recall for later use:

    \begin{Proposition}\label{hihi}
        \cite[Exercise 8.2.10]{Burago} Given two metric spaces $X$, $Y$ with finite Gromov--Hausdorff distance, if $X$ possesses an asymptotic cone, the same property holds true for $Y$. In addition the two asymptotic cones are isometric.
    \end{Proposition}

    For the Definition of Gromov--Hausdorff distance see \cite[Definition 7.3.10]{Burago}.

    \smallskip

    We aim at determining the asymptotic cone of $(\VN, d_\VN)$. Since $(\VN, d_\VN)$ is isometric to $(\VV,d_\VV)$, we equivalently study the asymptotic cone of the latter.

    \smallskip

    We will use the following crucial property of $d_\VV$:

    \begin{Lemma}\label{inva}
        We have
        \begin{equation*}
            d_\VV((x_0,h),(x_0,\overline h)) = d_\VV((x_0,0),(x_0,\overline h-h)) \txt{for any $x_0 \in \VV_0$, $h$, $\overline h$ in $\Z^{b(\G_0)}$}.
        \end{equation*}
    \end{Lemma}
    \begin{proof}
        We take a path $\xi =(e_i,h_i)_{i=1}^m$ linking $(x_0,h)$ to $(x_0, \overline h)$ with length equal to $d_\VV ((x_0,h),(x_0,\overline h))$. By the concatenation principle
        \begin{equation}\label{inva1}
            \sum_{i=1}^m \theta(e_i) = \overline h-h.
        \end{equation}
        We thereafter project $\xi$ to $\G_0$ obtaining a cycle $\xi_0$ passing through $x_0$, and finally lift $\xi_0$ to $\G$ choosing as initial point $(x_0,0)$, which gives by \eqref{patchbis}, \eqref{inva1} a path linking $(x_0,0)$ to $(x_0,\overline h-h)$. We have therefore proven that
        \begin{equation*}
            d_\VV((x_0,h),(x_0, \overline h)) \ge d_\VV((x_0,0),(x_0, \overline h-h)).
        \end{equation*}
        The above argument can be easily modified to get the converse inequality as well.
    \end{proof}

    \medskip

    Henceforth we fix $x_O\in\VV_0$ and define
    \begin{equation*}
        d_O(h_1,h_2):=d_\VV((x_O,h_1),(x_O,h_2)),\qquad\text{for }h_1,h_2\in\Z^{b(\G_0)}.
    \end{equation*}
    It is easy to check that $d_O$ is a metric on $\Z^{b(\G_0)}$. In addition, thanks to Lemma \ref{inva}, we have that this metric is \emph{invariant}, {\it i.e.}, $d_O(h_1,h_2)=d_O(0,h_2-h_1)$ for any $h_1,h_2\in\Z^{b(\G_0)}$.

    \begin{Lemma}\label{stabisnorm}
        For each $h\in\Z^{b(\G_0)}$ the limit
        \begin{equation}\label{eq:stabisnorm.1}
            \|h\|:=\lim_{n\to\infty}\frac{d_O(0,nh)}n
        \end{equation}
        does exist. Furthermore, $\|\cdot\|$ can be uniquely extended to a norm on $\R^{b(\G_0)}$.
    \end{Lemma}
    \begin{proof}
        Exploiting \cite[Propositions 8.5.1 and 8.5.3]{Burago} we see that the limit in \eqref{eq:stabisnorm.1} exists and can be uniquely extended to a seminorm on $\R^{b(\G_0)}$. To show that it is actually a norm, we claim that
        \begin{equation*}
            d_O(0,h) \ge |h| \txt{for any $h \in \Z^{b(\G_0)}$.}
        \end{equation*}
        In fact, let $\xi=((e_i,h_i))_{i=1}^m$ be a geodesic for $d_\VV((x_O,0),(x_O,h))= d_O(0,h)$. We derive from \eqref{eq:pathlift} that
        \begin{equation*}
            \xi= \left ( e_i, \sum_{j= 1}^{i} \theta( e_j) \right )_{i=1}^m\!\!\!\!\!\!\!\!.
        \end{equation*}
        Since $|\theta(e_i)|$ can be either equal to $1$ or to $0$ and $\xi$ is a geodesic for $ d_O(0,h)$, we in turn deduce that
        \begin{equation*}
            |h| = \left | \sum_{j= 1}^{m} \theta( e_j) \right | \le m = d_O(0,h),
        \end{equation*}
        which proves the claim.
    \end{proof}

    \medskip

    We call $\|\cdot\|$ the {\it stable norm} of $d_O$. Lemma \ref{stabisnorm} implies:

    \begin{Lemma}\label{Zcone}
        \emph{\cite[Theorem 8.5.4]{Burago}} $(\R^{b(\G_0)},\|\cdot\|)$ is the asymptotic cone of $(\Z^{b(\G_0)},d_O)$.
    \end{Lemma}

    This in turn yields:

    \begin{Theorem}
        The asymptotic cone of $(\VN,d_{\VN})$ is $(\R^{b(\G_0)},\|\cdot \|)$.
    \end{Theorem}
    \begin{proof}
        We consider the map $f$ between $(\VV,d_\VV)$ and $(\Z^{b(\G_0)},d_O)$ given by $f((x,h)):= h$. We claim that it possesses finite distortion. In fact, by Lemma \ref{inva}, we have for any $(x_1,h_1),(x_2,h_2)\in\VV$,
        \begin{align*}
            |d_\VV((x_1,h_1),(x_2,h_2))-d_O(h_1,h_2)|=\,&|d_\VV((x_1,h_1),(x_2,h_2))-d_\VV((x_O,h_1),(x_O,h_2))|\\
            \le\,&|d_\VV((x_1,h_1),(x_O,h_1))+d_\VV((x_2,h_2),(x_O,h_2))|\\
            \le\,&2\max_{y_1,y_2\in\VV_0}d_\VV((y_1,0),(y_2,0))=:C,
        \end{align*}
        which shows the claim. Taking also into account that $f$ is surjective, we derive that it is a $C$--isometry, according to \cite[Definition 7.3.27]{Burago}. This implies by \cite[Corollary 7.3.28]{Burago} that $(\VV,d_\VV)$ and $(\Z^{b(\G_0)},d_O)$ have finite {Gromov-Hausdorff} distance. Therefore we get the assertion exploiting Proposition \ref{hihi}, Lemma \ref{Zcone}, plus the fact that $(\VV,d_\VV)$ and $(\VN, d_\VN)$ are isometric.
    \end{proof}

    \bigskip

    \subsection{Convergence result}

    We give a notion of convergence of a family of functions $g_\eps: \NN \to \R$ to a function $g: \R^{b(\G_0)} \to \R$ driven, so to say, by the maps $\eps \pi_2$. Compare with the convergence used in \cite{ConItSic}. Note that in the definitions below only the {\it trace} of the approximating functions on $\VN$ enters into play, {namely the restriction of these functions to $\VN$.}

    \begin{Definition}\label{semi}
        We say that $g_\eps$ is {\em $\eps \pi_2 $--convergent} to $g$ {\em locally uniformly} if for any subsequence $\eps_n \to 0^+$ as $n\to +\infty$, and $z_n \in \VN$ with $\eps_n \pi_2(z_n)$ converging to $h \in \R^{b(\G_0)}$ as $n\to +\infty$, one has
        \begin{equation*}
            \lim_{n\to +\infty} g_{\eps_n}(z_n)=g(h).
        \end{equation*}
    \end{Definition}

    \smallskip

    The above notion of convergence can be easily generalized to the case of functions $v_\eps: \NN \times [0,+\infty) \to \R$ with limit $v: \R^{b(\G_0)} \times [0,+\infty) \to \R$.

    \smallskip

    \begin{Definition}\label{semibis}
        We say that $v_\eps$ is {\em $\eps \pi_2$--convergent} to $v$ {\em locally uniformly} if
        \begin{equation*}
            \lim_{n\to +\infty} v_{\eps_n}(z_n,t_n) = v(h,t)
        \end{equation*}
        for any subsequences $\eps_n \to 0^+$ as $n\to +\infty$, and $z_n \in \VN$ with $(\eps_n \pi_2(z_n) ,t_n)$ converging to $(h,t)$.
    \end{Definition}

    \smallskip

    \begin{Definition}\label{semiter}
        We say that a family of functions $g_\eps: \NN \to \R$ is {\it $\eps d_{\NN}$--equi--uniformly continuous} if there exists a modulus of continuity $\om$ such that
        \begin{equation*}
            |g_\eps(z_1) - g_\eps(z_2)| < \om(\eps d_{\NN}(z_1, z_2)) \txt{for any $\eps >0$, $z_1,\,z_2 \in \NN$.}
        \end{equation*}
        Since $d_\NN$ and $d_\VN$ are equivalent on $\VN$ this also implies, up to a modification of $\om$, that
        \begin{equation*}
            |g_\eps(z_1) - g_\eps(z_2)| < \om(\eps d_{\VN}(z_1, z_2)) \txt{for any $\eps >0$, $z_1,\,z_2 \in \VN$.}
        \end{equation*}
    \end{Definition}

    \medskip

    Now that we have introduced all the needed notions, let us rewrite the statement of the Main Theorem given in Section \ref{secHam}.

    \smallskip

    \begin{MainThm}
        {Let $\NN =(\VN,\EN)$ a periodic network and let $\G=(\VV,\EE)$ denote its underlying abstract graph, with base graph $\Gamma_0=(\VV_0,\EE_0)$ and Bravais lattice $G\simeq \Z^{b(\G_0)}$}. Let ${H=}\{H_\gamma\}_{\gamma\in\EN}$ be a Hamiltonian on $\NN$, satisfying conditions {\bf(H$_\ga$1)--(H$_\ga$4)}. Let $u_\eps$ be the solutions to the rescaled time-dependent Hamilton--Jacobi equations:
        \begin{equation*}\tag{HJ$^\eps$}
            \left\{
            \begin{array}{lll}
                \partial_t u_\eps(x,t) + H\left(x, \frac{1}{\eps} \partial_x u_\eps(x,t)\right) = 0 && x\in \NN,\; t>0\\
                u_\eps(x,0) = g_\e (x) && x\in \NN
            \end{array}
            \right.
        \end{equation*}
        coupled with {maximal flux limiters}, $\eps d_\NN$--equi-uniformly continuous initial data $g_\eps: \NN \longrightarrow \R$ at $t=0$ and locally uniformly $\eps\pi_2$--convergent, as $\eps$ goes to $0^+$, to a function $g: \R^{b(\G_0)} \longrightarrow \R$.\\
        \noindent Then, there exists a convex and superlinear effective Hamiltonian $\bar{H}: \R^{b(\G_0)} \rightarrow \R$, such that the family $ u_\eps$ locally uniformly $\eps\pi_2$--converge to the unique viscosity solution $u: \R^{b(\G_0)} \times (0,+\infty) \longrightarrow \R$ to
        \begin{equation*}\tag{HJ$^{\rm lim}$}
            \left\{
            \begin{array}{lll}
                \partial_t u ( h,t) + \bar{H}(\partial _{h} u(h,t)) = 0 && (h,t) \in \R^{b(\G_0)} \times (0,\infty)\\
                u(h,0) = g(h) && h\in \R^{b(\G_0)}.
            \end{array}
            \right.
        \end{equation*}
    \end{MainThm}

    \medskip

    \begin{Remark}
        We recall that the limit solution to \eqref{HJlim} coupled with a uniformly continuous initial datum $g : \R^{b(\G_0)} \to \R$ is given by
        \begin{equation}\label{solim}
            u(h,t) := \inf \left \{ g(h_0) + t \, \be \left (\frac{h-h_0}t \right ) \mid h_0 \in \R^{b(\G_0)} \right \},
        \end{equation}
        where we used the fact that $\bar H= \al: \R^{b(\G_0)} \to \R$, see Proposition \ref{effetto}, and $\beta$ is its convex conjugate.
    \end{Remark}

    \medskip

    Before proving the Main Theorem we discuss three preliminary results:

    \smallskip

    \begin{Lemma}
        Assume that $g_\eps : \NN \to \R$ is an $\eps d_{\NN}$ ($\eps d_{\VN}$)--equi--uniformly continuous family of functions, with modulus of continuity $\om$, locally uniformly $\eps \pi_2$--convergent to a function $g: \R^{b(\G_0)} \to \R$. Then, $g$ is uniformly continuous, with respect to $\|\cdot\|$, with the same modulus of continuity.
    \end{Lemma}
    \begin{proof}
        Taking into account conditions {\bf (i)}, {\bf (ii)} of Definition \ref{pointed}, we can find, for any given $z_0\in\VN$, an infinitesimal sequence $\eps_n$ and $r >0$ satisfying
        \begin{equation*}
            \sup \left \{ |\eps_n \, d_{\VN}(z_1,z_2)- \eps_n \, \|\pi_2(z_1)-\pi_2(z_2)\| | \mid z_1 , \, z_2 \in B_{\eps_n}(z_0,r) \right \} < \frac 1n,
        \end{equation*}
        where $B_{\eps_n}$ stands for the ball in $\VN$ with respect to the metric $ \eps_n \, d_{\VN}$. In addition, for any $h$ with $\|h\| < r - \frac 1n$, we can find an element $z_n \in \VN$ with
        \begin{equation*}
            \eps_n \, d_{\VN}(z_n,z_0) < r
        \end{equation*}
        satisfying
        \begin{equation*}
            \|h -\eps_n \pi_2(z_n)\| < \frac 1n.
        \end{equation*}
        Given $h_1,h_2\in\R^{b(\G_0)}$ with $\|h_1\|, \, \|h_2\| < r - \frac 1n$, we therefore find, exploiting the above inequalities and the fact that the $g_\eps$'s $\eps \pi_2$--converge to $g$, up to extracting a subsequence of $\eps_n$, two sequences $\{z_n^1\}$, and $\{z_n^2\}$ in $\VN$ such that (for $i=1,\,2$)
        \begin{eqnarray*}
            \eps_n \, d_{\VN}(z^i_n,z_0) &<& r \\
            \|h_i- \eps_n \pi_2(z^i_n)\| &<& \frac 1n \\
            |g_{\eps_n}(z^i_n) - g(h_i)| &<& \frac 1n\\
            |\eps_n \, d_{\VN}(z_n^1,z_n^2)- \eps_n \|\pi_2(z_n^1)-\pi_2(z_n^2)\| | & <& \frac 1n.
        \end{eqnarray*}
        We derive
        \begin{eqnarray*}
            \eps_n \, d_{\VN}(z_n^1,z_n^2) &<& \eps_n \|\pi_2(z_n^1)-\pi_2(z_n^2)\| + \frac 1n \\
            & \le & \|h_1- \eps_n \pi_2(z^1_n)\| + \|h_2- \eps_n \pi_2(z^2_n)\| + \|h_1-h_2\| + \frac 1n \\ &<& \|h_1-h_2\| + \frac 3n
        \end{eqnarray*}
        and consequently
        \begin{eqnarray*}
            |g(h_1)-g(h_2)| &<& |g_{\eps_n}(z^1_n)- g_{\eps_n}(z^2_n)| + \frac 2n \\
            &< & \om( \eps_n \, d_{\VN}(z_n^1,z_n^2) ) + \frac 2n\\
            &<& \om \left (\|h_1-h_2\| + \frac 3n \right ) + \frac 2n.
        \end{eqnarray*}
        This shows that $\om$ is a uniform continuity modulus for $g$ as well.
    \end{proof}

    \bigskip

    We define for $\eps >0$, $(z,t) \in \NN \times (0, + \infty)$
    \begin{equation}\label{barato}
        \bar u_\eps (z,t) := \inf \left \{ g_\eps(\xi(0)) + \int_0^t L(\xi,\eps \dot \xi) \, d\tau \right \},
    \end{equation}
    where the infimum is taken over the curves in $\NN$ with $\xi(0) \in \VN$, $\xi(t) =z$. These functions differ from $u_\eps$ just because the initial point of the admissible curves is taken in $\VN$ and not at any point of $\NN$.

    \smallskip

    \begin{Lemma}\label{lekonve}
        Given a compact subset $K$ of $\R^{b(\G_0)}$, $A >0$ and a compact interval $I \subset [0,+\infty)$, we can find $R>0$ with
        \begin{equation*}
            \eps d_{ \NN}(z^*,\xi_\eps(0)) < R
        \end{equation*}
        for $\eps >0$ small enough, any $z^* \in \VN$ with $\eps \pi_2(z^*) \in K$, $t \in I$, any curve $\xi_\eps$ in $\NN$ satisfying $\xi_\eps(0) \in \VN$, $\xi_\eps(t)=z^*$ and
        \begin{equation}\label{lekonve00}
            g_\eps(\xi_\eps(0)) + \int_0^t L(\xi_\eps, \eps \, \dot \xi_\eps) \, d\tau < \bar u_\eps(z^*,t) + A.
        \end{equation}
    \end{Lemma}
    \begin{proof}
        We fix $z^* \in \VN$ with $\eps \pi_2(z^*) \in K$, $t \in I$. Since the $g_\eps$'s have the same uniform continuity modulus $\omega$, that we can assume concave, there exist $P >0$, $Q > 0$ with
        \begin{equation}
            g_\eps(z) \ge -P \, \eps d_{\VN}(z,z^*) - Q \qquad z \in \NN, \, \eps >0.\label{lekonve0}
        \end{equation}
        Up to increasing $Q$, if necessary, we have that
        \begin{equation*}
            g(h) \le Q \txt{for $h \in K$}
        \end{equation*}
        and we therefore claim that
        \begin{equation}\label{lekonve1}
            g_\eps(z) \le Q+1 \txt{for $z \in \VN$ such that $\eps \pi_2(z) \in K$,}
        \end{equation}
        when $\eps$ is suitably small. Were \eqref{lekonve1} not true, we could in fact find an infinitesimal sequence $\eps_n$ and $z_n \in \VN$ with $\eps_n \pi_2(z_n) \in K$ such that
        \begin{equation*}
            g_{\eps_n}(z_n) > Q+1.
        \end{equation*}
        Since $K$ is compact, $\eps_n \pi_2 (z_n)$ would converge, up to subsequence, to an element $h \in K$. By the $\eps \pi_2$--convergence of the $g_\eps$'s to $g$, we would find
        \begin{equation*}
            g(h)= \lim_n g_{\eps_n}(z_n) \ge Q+1,
        \end{equation*}
        which is impossible. The claim is therefore proven. Due to $L(z,\cdot)$ being superlinear, see \eqref{super}, there exists $B >0$ with
        \begin{equation*}
            L(z,q) \ge 2 \, P \, |q| - B \txt{for all $(z,q) \in T\NN$}
        \end{equation*}
        which implies
        \begin{equation}\label{lekonve3}
            \int_0^t \, L(\xi_\eps, \eps \dot\xi_\eps) \, d\tau \ge 2 P \, \eps d_{\NN}(z^*, \xi_\eps(0)) - B \, t,
        \end{equation}
        where $\xi_\eps$ is a curve satisfying \eqref{lekonve00} for some $A >0$. Exploiting \eqref{lekonve0}, \eqref{lekonve1}, \eqref{lekonve3} plus the formula representing $u_\eps$, we finally obtain
        \begin{align*}
            A+ Q + 1 + & t \, L(z^*,0) \ge \bar u_\eps(z^*,t) + A \\ & \ge g_\eps(\xi_\eps(0)) + \int_0^t L(\xi_\eps,\eps \dot\xi_\eps) \, d\tau \\ & \ge - P\, \eps d_{\NN}(\xi_\eps(0),z^*) - Q + 2 \, P \, \eps \, d_{\NN}(\xi_\eps(0),z^*) - B \, t,
        \end{align*}
        which in turn gives
        \begin{equation*}
            \eps \, d_{\NN} (\xi_\eps(0),z^*) \le \frac 1P \, \left (2Q + 1 + t \, L(z^*,0) + B \, t + A \right)
        \end{equation*}
        and proves the assertion.
    \end{proof}

    \medskip

    \begin{Lemma}\label{novia}
        Assume the $g_\eps$'s to be $\eps d_\NN$--equi--uniformly continuous with modulus $\om$. Let $\bar u_\eps$ be the functions defined in \eqref{barato}. If $\bar u_\eps$ is $\eps \pi_2$--locally uniformly convergent to $u$, then the same holds true for $u_\eps$.
    \end{Lemma}
    \begin{proof}
        Let $\eps_n$ be an infinitesimal subsequence and $(z_n,t_n)$ in $\VN \times (0,+\infty)$ with $(\eps_n \pi_2(z_n),t_n)$ converging to an element $(h_0,t_0) \in \R^{b(\G_0)}\times (0,+\infty)$; our task is to show that
        \begin{equation}\label{novia1}
            u_{\eps_n}(z_n,t_n) \longrightarrow u(h_0,t_0).
        \end{equation}
        We can write
        \begin{equation*}
            u_{\eps_n}(z_n,t_n) = g_{\eps_n}(z^*_n) + \int_0^{t_n} L(\xi_n,\eps_n \dot \xi_n) \, d\tau
        \end{equation*}
        for suitable curves $\xi_n$ in $\NN$ satisfying $\xi_n(0)= z^*_n$, $\xi_n(t)= z_n$. By changing the integral variable in the above formula from $t$ to $\frac t\eps$, we derive
        \begin{equation*}
            u_{\eps_n}(z_n,t_n) = g_{\eps_n}(z^*_n) + \eps_n \, \Phi_L(z^*_n,z_n,t_n/\eps_n).
        \end{equation*}
        By exploiting Proposition \ref{minactlip}, Lemma \ref{ester} and \textbf{(P1)}, we find $z'_n \in \VN$, $t'_n$ with
        \begin{equation*}
            d_\NN(z^*_n,z'_n) \le\sup_{\ga\in\EN}|\ga|=:M_\NN \txt{and} \qquad |t_n-t'_n| < \eps_n
        \end{equation*}
        satisfying
        \begin{equation*}
            \Phi_L(z^*_n,z_n,t_n/\eps_n) \ge \Phi_L(z'_n,z_n,t'_n/\eps_n) - M
        \end{equation*}
        for a suitable $M >0$. Also exploiting the uniform equicontinuity of the $g_\eps$'s, we derive
        \begin{eqnarray}
            \bar u_{\eps_n}(z_n,t_n) &\ge& u_{\eps_n}(z_n,t_n)\label{novia2} \\
            &=& g_{\eps_n}(z^*_n) + \eps_n \, \Phi_L(z^*_n,z_n,t_n/\eps_n) \nonumber \\ &\ge& g_{\eps_n}(z'_n) - \om(\eps_n d_\NN(z'_n, z^*_n)) + \eps_n \Phi_L(z'_n,z_n,t'_n/\eps_n)- \eps_n M \nonumber\\
            &\ge& \bar u_\eps(z_n,t'_n) - \om(\eps_n M_\NN) - \eps_n M. \nonumber
        \end{eqnarray}
        Since $\lim_{n \to \infty} t_n = \lim_{n \to \infty} t'_n=t_0$, we have by assumption
        \begin{equation*}
            \lim_{n \to \infty} \bar u_{\eps_n}(z_n,t_n) =\lim_{n \to \infty} \left(\bar u_\eps(z_n,t'_n) - \om(\eps_n M_\NN) - \eps_n M \right) = u(h_0,t_0),
        \end{equation*}
        which implies \eqref{novia1} by \eqref{novia2}.
    \end{proof}

    \medskip

    We proceed by proving the Main Theorem.

    \begin{proof}[{\bf Proof of Main Theorem:}]
        \, Given $(h_0,t_0) \in \R^{b(\G_0)} \times (0,+\infty)$, we can write by \eqref{solim}
        \begin{equation*}
            u(h_0,t_0) = g(h_1) + t_0 \, \be \left ( \frac{h_0-h_1}{t_0} \right )
        \end{equation*}
        for a suitable $h_1$. Given an infinitesimal sequence $\eps_n$, we take $(z_n,t_n) \in \VN \times (0,+\infty)$ with
        \begin{equation}\label{konve00}
            (\eps_n \,\pi_2(z_n),t_n) \to (h_0,t_0)
        \end{equation}
        and $\widehat z_n \in \VN$ with
        \begin{equation}\label{konve01}
            \eps_n \, \pi_2(\wha z_n) \to h_1.
        \end{equation}
        We aim at showing that
        \begin{equation}\label{konve0}
            \eps_n \, \Phi_L(\wha z_n, z_n,t_n/\eps_n) \to t_0 \, \be \left (\frac{h_0- h_1}{t_0} \right ) \txt{as $n \to +\infty$.}
        \end{equation}
        The first step is to estimate
        \begin{equation}\label{konve1}
            \left |\frac{\eps_n}{t_n} \, \Phi_L(\wha z_n, z_n,t_n/\eps_n) - \be \left (\eps_n \frac{\pi_2(z_n)- \pi_2(\wha z_n)}{t_n} \right ) \right |.
        \end{equation}
        Note that the quantity $\eps_n \, \frac{\pi_2(z_n)- \pi_2(\wha z_n)}{t_n}$ stays bounded for any $n$, because
        \begin{equation*}
            \eps_n \, (\pi_2(z_n) - \pi_2(\wha z_n)) \to h_0-h_1 \qquad\text{and} \qquad t_n \to t_0,
        \end{equation*}
        while $\frac{t_n}{\eps_n} \to + \infty$ as $n \to + \infty$. We then invoke Theorem \ref{melania} to get that the quantity in \eqref{konve1} becomes infinitesimal as $\eps$ goes to $0$, and consequently
        \begin{equation}\label{konve2}
            \left |t_0 \, \frac{\eps_n}{t_n} \, \Phi(\widehat z_n,z_n, t_n/\eps_n) - t_0 \, \be \left (\eps \frac{\pi_2(z_n) - \pi_2(\wha z_n)}{t_n} \right ) \right | \longrightarrow 0 \qquad{\rm as}\; n\to +\infty.
        \end{equation}
        Further, by exploiting the local Lipschitz continuity of $\be$ we obtain
        \begin{eqnarray}
            && \left | t_0 \, \be \left (\eps_n \frac{\pi_2(z_n) - \pi_2(\wha z_n)}{t_n} \right ) - t_0 \, \be \left (\frac{h_0-h_1}{t_0} \right ) \right |\label{konve3} \\
            && \quad \le\; t_0\, \ell\, \left | \eps_n \frac{\pi_2(z_n) - \pi_2(\wha z_n)}{t_n}- \frac{h_0-h_1}{t_0} \right | \nonumber
        \end{eqnarray}
        for a suitable Lipschitz constant $\ell$. Finally, since
        \begin{itemize}
            \item $\eps_n \frac{\pi_2(z_n)- \pi_2(\wha z_n)}{t_n}$ is bounded;
            \item $t_n \to t_0$
            \item the quantity in \eqref{konve1} becomes infinitesimal as $n \to + \infty$,
        \end{itemize}
        then $ \eps_n \, \Phi_L(\widehat z_n,z_n,t_n/\eps_n)$ is bounded and consequently
        \begin{equation}\label{konve4}
            \eps_n \, \left | \frac{t_0}{t_n} - 1 \right | \, \Phi_L(\widehat z_n,z_n,t_n/\eps_n) \longrightarrow 0 \qquad{\rm as}\; n\to +\infty,
        \end{equation}
        so that \eqref{konve0} can be obtained by combining \eqref{konve2}, \eqref{konve3}, \eqref{konve00}, \eqref{konve01}, \eqref{konve4}. This implies
        \begin{equation*}
            g_{\eps_n}(\widehat z_n) + \eps_n \, \Phi_L(\widehat z_n,z_n,t_n/\eps_n) \to u(h_0,t_0)
        \end{equation*}
        and consequently
        \begin{equation}\label{konve5}
            \limsup_{n} \bar u_{\eps_n} (z_n,t_n) \le u(h_0,t_0).
        \end{equation}
        Given $\de >0$, assume now $\wtd z_n \in \VN$ to be a $\de$--optimal point for $\bar u_{\eps_n}(z_n,t_n)$. Since $(\eps_n \pi_2(z_n),t_n) \to (h_0,t_0)$, see \eqref{konve00}, the elements $\eps_n \pi_2(z_n)$ are contained in a compact subset of $\R^{b(\G_0)}$, say $K$, and the $t_n$'s are contained in a compact interval, say $I$. We deduce by Lemma \ref{lekonve} and \textbf{(i)} in Definition \ref{pointed} that
        \begin{eqnarray*}
            \|\eps_n \pi_2(\wtd z_n)\| &\le& \|\eps_n \pi_2(\wtd z_n) - \eps_n \pi_2(z_n) \| + \eps_n \|\pi_2( z_n)\| \\ &\le& \eps_n \, d_{ \VN}(z_n, \wtd z_n) + Q + R' \le R + Q+ R' \txt{for suitable $Q$, $R$, $R'$. }
        \end{eqnarray*}
        Consequently the $ \eps_n \pi_2(\wtd z_n)$'s converge, up to subsequences, to some element $h_2 \in \R^{b(\G_0)}$. Arguing as in the first part of the proof, we show that
        \begin{equation*}
            \eps_n \, \Phi(\wtd z_n,z_n ,t_n/\eps_n) \longrightarrow t_0 \, \be \left (\frac{h_0-h_2}{t_0} \right ) \qquad{\rm as}\; n\to +\infty,
        \end{equation*}
        and consequently
        \begin{equation}\label{konve6}
            \liminf \bar u_{\eps_n} (z_n,t_n) + \de \ge u(h_0,t_0).
        \end{equation}
        The assertion is then obtained in the light of \eqref{konve5}, \eqref{konve6} and Lemma \ref{novia}, taking into account that $\de >0$ has been arbitrarily chosen.
    \end{proof}

    \bigskip

    \bigskip

    \section{Some proofs}\label{proofs}

    In this section, we provide the proofs of some of the more technical claims that we decided to postpone in order to make the presentation more accessible.

    \subsection{Proof of Proposition \ref{uffaz}}\label{uffazz}

    \uffaz*

    \smallskip

    \begin{Remark}\label{positive}
        If we fix arbitrarily a constant $b$ and consider the family of Hamiltonians
        \begin{equation*}
            H_\ga(s,\rho) + b \qquad\text{ for $\ga \in \EN$}
        \end{equation*}
        on $\NN$, they inherit from the $H_\ga$'s the properties {\bf (H$_\ga$1)}--{\bf (H$_\ga$4)}. The corresponding Lagrangians are
        \begin{equation*}
            L_\ga(s,\la) - b.
        \end{equation*}
        and the overall Lagrangian (see Section \ref{action}) becomes
        \begin{equation*}
            L(z,q) - b.
        \end{equation*}
        Since the structure of homogenization result does not change if we replace $H_\ga$ by $H_\ga\ +b$, (up to substituting $\bar{H}$ with $\bar{H}+b$), we can exploit the coercivity of the $L_\ga$'s, and choose the additive constant $b$ in such a way that the modified Lagrangians on any arc as well as the modified overall Lagrangian are {\it strictly positive}. In conclusion, we can assume the strict positivity of all such Lagrangians, without any loss of generality. We will make use of this property in the section.
    \end{Remark}

    \smallskip

    \begin{Notation}\label{inizio}
        Throughout the section we use the following notations:
        \begin{itemize}
            \item $\om_g(\cdot)$ is a concave uniform continuity modulus for the initial datum $g$;
            \item $\ell_C$ is, for $C >0$, the Lipschitz constant of any curve minimizing the action in a time $T$ between points $z_1$, $z_2$ in $\NN$ with $d_{\NN}(z_1,z_2) \le C\, T$, see Proposition \ref{curvacce}.
        \end{itemize}
        We further set
        \begin{eqnarray*}
            M[r] &:=& \max \{L(z,q) \mid z \in \NN, |q| \le r \} \\
            R_0 &:=& \max_z L(z,0).
        \end{eqnarray*}
    \end{Notation}

    \smallskip

    Some preliminary results are needed.

    \begin{Lemma}\label{ester}
        Given $t_0 >0$, there exists $C_0 >0$ such that for any $(z,t) \in \NN \times (t_0,+\infty)$, any optimal curve $\xi$ for $w(z,t)$ satisfies
        \begin{equation*}
            d_{\NN}(\xi(0),z) \le C_0 \, t.
        \end{equation*}
    \end{Lemma}
    \begin{proof}
        We have
        \begin{equation*}
            \om_g(r) \le a \, r +b \txt{for suitable $a$, $b$ positive.}
        \end{equation*}
        We further have by \eqref{super} that there exists $B$ such that
        \begin{equation*}
            L(z,q) > (a+1) \, |q| - B \txt{for any $(z,q) \in T\NN$}.
        \end{equation*}
        This implies
        \begin{equation*}
            g(z) + R_0 t - g(\xi(0)) \ge \int_0^t L(\xi,\dot \xi) \, d\tau \ge (a+1) \, d_{\NN}(\xi(0),z) - Bt,
        \end{equation*}
        where $\xi$ is an optimal curve for $w(z,t)$, and consequently
        \begin{equation*}
            a \, d_{\NN}(z,\xi(0)) + b + R_0 \, t \ge \om_g(d_{\NN}(z,\xi(0))) + R_0 \, t \ge (a+1) \, d_{\NN}(\xi(0),z) - B \, t,
        \end{equation*}
        which in turns yields
        \begin{equation*}
            (R_0 + B ) \, t + b \ge d_{\NN}(\xi(0),z).
        \end{equation*}
        The assertion is then proven by choosing
        \begin{equation*}
            C_0:= R_0 + B + \frac b{t_0}.
        \end{equation*}
    \end{proof}

    \smallskip

    \begin{Proposition}\label{minactlip}
        Given $C>0$, $\Phi_L$ is Lipschitz continuous on
        \begin{equation*}
            A_C:=\left\{(z_1,z_2,t)\in\NN^2\times\R^+:d_\NN(z_1,z_2)\le Ct\right\}.
        \end{equation*}
    \end{Proposition}
    \begin{proof}
        We start showing that, if $(z_1,z_2,t_1),(z_1,z_2,t_2)\in A_C$ with $t_2>t_1$, there is a constant $\bar\ell>0$ such that
        \begin{equation}\label{eq:minactlip1}
            |\Phi_L(z_1,z_2,t_1)-\Phi_L(z_1,z_2,t_2)|\le\ov\ell|t_1-t_2|,
        \end{equation}
        where $\bar\ell:=M[\ell_{2C}]+R_0$, see Notation \ref{inizio} for the definition of $R_0$, $\ell_C$ and $M[\ell_C]$. We first assume that $\Phi_L(z_1,z_2,t_2)\ge\Phi_L(z_1,z_2,t_1)$, then
        \begin{equation}\label{eq:minactlip2}
            |\Phi_L(z_1,z_2,t_1)-\Phi_L(z_1,z_2,t_2)|=\Phi_L(z_1,z_2,t_2)-\Phi_L(z_1,z_2,t_1)\le R_0|t_2-t_1|.
        \end{equation}
        If instead $\Phi_L(z_1,z_2,t_2)<\Phi_L(z_1,z_2,t_1)$, we break the argument according to whether $t_2 -t_1 \ge t_1$ or $t_2 -t_1 < t_1$. In the first instance, taking into account Remark \ref{positive}, we derive
        \begin{equation}\label{eq:minactlip3}
            \begin{aligned}
                |\Phi_L(z_1,z_2,t_1)-\Phi_L(z_1,z_2,t_2)|=\;&\Phi_L(z_1,z_2,t_1)-\Phi_L(z_1,z_2,t_2)\\
                \le\;&M[\ell_C]t_1\le M[\ell_C]|t_2-t_1|.
            \end{aligned}
        \end{equation}
        If instead $t_2 -t_1 < t_1$, we denote by $\xi$ an optimal curve for $\Phi_L(z_1,z_2,t_2)$ and select $\bar t$ with $t_1- \bar t = t_2 -t_1$, we then define a curve $\eta: [0,t_1] \to \NN$ via
        \begin{equation*}
            \eta(t) := \left \{
            \begin{array}{cc}
                \xi(t) & \quad\text{for $t \in [0,\bar t]$} \\
                \xi(at +b) & \quad\text{for $t \in (\bar t, t_1]$,}
            \end{array}
            \right .
        \end{equation*}
        where $a$, $b$ are determined by the conditions
        \begin{equation*}
            a \, \bar t + b = \bar t \qquad{\rm and} \qquad a \, t_1 + b = t_2,
        \end{equation*}
        which in particular gives
        \begin{equation*}
            a = \frac{t_2- \bar t}{t_1 - \bar t} = 2,
        \end{equation*}
        so that $\eta(t)$ is a curve with Lipschitz constant less than or equal to $\ell_{2C}$. Taking into account Remark \ref{positive}, we have
        \begin{equation}\label{eq:minactlip4}
            \begin{aligned}
                |\Phi_L(z_1,z_2,t_1)-\Phi_L(z_1,z_2,t_2)|=\;&\Phi_L(z_1,z_2,t_1)-\Phi_L(z_1,z_2,t_2)\le\int_{\bar t}^{t_1}L(\eta,\dot\eta)d\tau\\
                \le\;&M[\ell_{2C}](t_1-\bar t)=M[\ell_{2C}](t_2-t_1).
            \end{aligned}
        \end{equation}
        Then, \eqref{eq:minactlip1} follows from \eqref{eq:minactlip2}, \eqref{eq:minactlip3} and \eqref{eq:minactlip4}.\\
        We now consider two points $(z_1,z_2,t_1)$, $(z_1,z_2',t_2)$ in $A_C$ and assume that $\Phi_L(z_1,z_2',t_2)\ge\Phi_L(z_1,z_2,t_1)$. Given an optimal curve $\xi$ for $\Phi_L(z_1,z_2,t_1)$, we define a curve $\eta: [0,t_1 + d_{\NN}(z_2,z_2')] \to \NN$ as
        \begin{equation*}
            \eta(t) := \left \{
            \begin{array}{cc}
                \xi(t) & \quad\text{for $t \in [0,t_1]$} \\
                \zeta(t-t_1) & \quad\text{for $t \in (t_1, t_1 + d_{\NN}(z_2,z_2') ]$,}
            \end{array}
            \right .
        \end{equation*}
        where $\zeta: [0, d_{\NN}(z_2,z_2')]\longrightarrow \NN$ is a geodesic between $z_2$ and $z'_2$. We have
        \begin{equation}\label{eq:minactlip5}
            \begin{aligned}
                \Phi_L(z_1,z_2',t_1+d_{\NN}(z_2,z_2'))-\Phi_L(z_1,z_2,t_1)\le\;&\int_{t_1}^{t_1+d_{\NN}(z_2,z'_2)}L(\eta,\dot\eta)d\tau\\
                \le\;&M[1]d_{\NN}(z_2,z_2').
            \end{aligned}
        \end{equation}
        On the other side, by the first part of the proof, we have
        \begin{equation*}
            \Phi_L(z_1,z_2',t_2)-\Phi_L(z_1,z_2',t_1+d_{\NN}(z_2,z_2'))\le\bar\ell(d_{\NN}(z_2,z'_2)+(t_2-t_1))
        \end{equation*}
        so that, combining it with \eqref{eq:minactlip5}, we get
        \begin{align*}
            |\Phi_L(z_1,z_2',t_2)-\Phi_L(z_1,z_2,t_1)|=\;&\Phi_L(z_1,z_2',t_2)-\Phi_L(z_1,z_2,t_1)\\
            =\;&\Phi_L(z_1,z_2',t_2)-\Phi_L(z_1,z_2',t_1+d_{\NN}(z_2,z_2'))\\
            &+\Phi_L(z_1,z_2',t_1+d_{\NN}(z_2,z_2'))-\Phi_L(z_1,z_2,t_1)\\
            \le\;&(M[1]+\bar\ell)[(t_2-t_1)+d_{\NN}(z_2,z_2')].
        \end{align*}
        Similarly, we can show that, given $(z_1,z_2,t_1),(z_1',z_2,t_2)\in A_C$,
        \begin{equation*}
            |\Phi_L(z_1',z_2,t_2)-\Phi_L(z_1,z_2,t_1)|\le(M[1]+\bar\ell)[(t_2-t_1)+d_{\NN}(z_1,z_1')].
        \end{equation*}
        This completes the proof.
    \end{proof}

    \smallskip

    \begin{Proposition}\label{paola}
        The function $w$ is Lipschitz continuous on $\NN \times [t_0, + \infty)$ for $t_0>0$.
    \end{Proposition}
    \begin{proof}
        Using Lemma \ref{ester}, there is a constant $C_0>0$ such that
        \begin{equation*}
            w(z,t)=\inf\{g(z_0)+\Phi(z_0,z,t)\}\txt{for any }(z,t)\in\NN\times[t_0,\infty),
        \end{equation*}
        where the infimum is taken over the $z_0\in\NN$ with
        \begin{eqnarray*}
            d_\NN(z_0,z)\le C_0t.
        \end{eqnarray*}
        This proves our claim, since $w$, on $\NN\times[t_0,\infty)$, is the infimum of a family equiLipschitz continuous functions by Proposition \ref{minactlip}.
    \end{proof}

    \smallskip

    We can now prove the main proposition of this section.

    \begin{proof}[{\bf Proof of Proposition \ref{uffaz}:}]
        We assume for purposes of contradiction that there exist sequences $(z_n,t_n)$, $(z'_n,t'_n)$ in $\NN \times [0,+\infty)$, and a constant $\eps >0$ with
        \begin{equation}\label{paola3}
            d_{\NN}(z_n,z'_n) + |t_n - t'_n| \to 0, \qquad |w(z_n,t_n) - w(z'_n,t'_n)| > 3 \, \eps
        \end{equation}
        for a given $\eps >0$, which clearly implies the assertion to be false. We deduce from Proposition \ref{paola} that $\inf_n \{t_n, t'_n\}=0$ and consequently that both $t_n$ and $t'_n$ are infinitesimal, up to extracting a subsequence.

        We proceed estimating
        \begin{equation*}
            |w(z,t) - g(z)| \txt{for $z \in \NN$, $t >0$.}
        \end{equation*}
        If $w(z,t) \ge g(z)$, we get
        \begin{equation}\label{matilde0}
            |w(z,t) - g(z)| = w(z,t) - g(z) \le R_0 \, t.
        \end{equation}
        If instead $w(z,t) < g(z)$, we have
        \begin{equation}\label{matilde1}
            |w(z,t) - g(z)| = g(z) - w(z,t) = g(z) - g(\xi(0)) - \int_0^t L(\xi,\dot\xi) \, d\tau,
        \end{equation}
        where the curve $\xi$ is optimal for $w(z,t)$. Exploiting the concavity of $\om_g$, we find a constant $a_\eps$, with
        \begin{equation*}
            a_\eps \, r + \eps \ge \om_g(r),
        \end{equation*}
        by \eqref{super} there exists $B_\eps$ satisfying
        \begin{equation*}
            L(z,q) \ge a_\eps \, |q| - B_\eps.
        \end{equation*}
        We therefore obtain from \eqref{matilde1}
        \begin{eqnarray*}
            |w(z,t) - g(z)| &\le & \om_g(d_{\NN}(z,g(\xi(0))) - \int_0^t L(\xi,\dot\xi) \, d\tau \\ &\le& a_\eps \, d_{\NN}(z,\xi(0)) + \eps - a_\eps \, d_{\NN}(z,\xi(0)) + B_\eps t \le \eps + B_\eps t.
        \end{eqnarray*}
        Taking also into account \eqref{matilde0}, we get in the end
        \begin{equation*}
            |w(z,t) - g(z)| \le \eps + R_\eps \, t \txt{for any $(z,t) \in \NN \times (0,+\infty)$,}
        \end{equation*}
        where $R_\eps = \max\{R_0,B_\eps\}$. We consequently get
        \begin{eqnarray*}
            |w(z_n,t_n) - w(z'_n,t'_n)| &\le& |w(z_n,t_n) - g(z_n)| + |g(z_n) - g(z'_n)| + |g(z'_n) - w(z'_n,t'_n)| \\
            &\le& 2 \, \eps + R_\eps \, (t_n+t'_n) + \om_g(d_{\NN}(z_n,z'_n)),
        \end{eqnarray*}
        which contradicts \eqref{paola3}. This concludes the proof.
    \end{proof}

    \medskip

    \subsection{Proof of Theorem \ref{compa}}\label{compaz}

    \compa*

    \smallskip

    In the case of finite networks the above theorem is obtained, see \cite{Sic}, by first proving it for a semidiscrete problem and then exploiting the links between the two problems to transfer it to the continuous case. In the first step the compactness of the network is essentially exploited, and so a new argument is needed in our context.

    \smallskip

    \begin{Notation}
        We define
        \begin{eqnarray*}
            &Q:= (0,1) \times (0,+\infty) \quad \partial_p Q:= \big ((0,1) \times \{0\}\big ) \cup \big ( \{0,1\} \times [0, + \infty) \big )& \\
            &\partial^-_p Q:= \big ((0,1) \times \{0\}\big ) \cup \big ( \{0\} \times [0, + \infty) \big ). &
        \end{eqnarray*}

        We further denote by $UC (\cdot)$, $C(\cdot)$, respectively, the spaces of real valued uniformly continuous and continuous functions, defined on a given topological space.

    \end{Notation}

    \smallskip

    Following \cite{Sic}, we introduce the related semidiscrete problem.

    We define, for any $z \in \VN$, the operator
    \begin{equation*}
        F_z: UC \big ( (\NN \times \{0\}) \cup ( \VN \times (0,+\infty)) \big ) \to UC([0,+\infty))
    \end{equation*}
    through the following two steps:
    \begin{itemize}
        \item[--] Given $w \in C \big ( (\NN \times \{0\}) \cup ( \VN \times (0,+\infty)) \big )$, we indicate by $F_\ga[w]$ the maximal among the uniformly continuous solution to \eqref{HJg} in $Q$ with trace less or equal $w(\ga(s),0)$ on $[0,1]\times\{0\}$ and $w(\ga(0),t)$ on $\{0\}\times[0,\infty)$;
        \item[--] We set
        \begin{equation*}
            F_z[w]:= \min_{\ga \in \EN^z} F_\ga[w](1,\cdot).
        \end{equation*}
    \end{itemize}

    \smallskip

    We record for later use:

    \begin{Proposition}\label{unifcontloc}
        Given $g\in UC(\partial_p^-Q)$, we set
        \begin{equation}\label{eq:unifcontloc.1}
            w_0(s,t) =\min \left \{ g(s_0,t_0) + \int_{t_0}^t L_\ga(\eta,\dot \eta) \, d\tau \right \},
        \end{equation}
        where the minimum is over the elements $(s_0,t_0)$ varying in $ \partial^-_p Q$ with $t_0 <t$, and the curves $\eta:[t_0,t] \to [0,1]$ with $\eta(t_0)=s_0$ and $\eta(t)=s$. Then $w_0\in UC(\overline Q)$.
    \end{Proposition}
    \begin{proof}
        We set
        \begin{equation*}
            \Phi_\gamma(s_0,t_0,s,t):=\inf\left\{\int_{t_0}^tL_\gamma(\eta,\dot\eta)d\tau\mid\eta\text{ curve with }\eta(t_0)=s_0,\,\eta(t)=s\right\},
        \end{equation*}
        for $(s_0,t_0,s,t) \in \ov Q^2$ with $t >t_0$, so that if $(s,t)\in Q$
        \begin{equation*}
            w_0(s,t)=\inf\{g(s_0,t_0)+\Phi_\gamma(s_0,t_0,s,t)\},
        \end{equation*}
        with infimum taken over the $(s_0,t_0)\in\partial_p^-Q$. We recall for later use that by \cite[Theorem 3.1]{Davini07}, see also Proposition \ref{minactlip}, the map $\Phi_\gamma(s_0,t_0,s,t)$ is Lipschitz continuous on the set
        \begin{equation}\label{eq:unifcontloc8}
            \left\{(s_0,t_0,s,t)\in\overline Q^2\mid|s-s_0|\le C(t-t_0)\right\}
        \end{equation}
        whenever $C$ is a positive constant. For $\delta\in(0,1)$, we further set
        \begin{equation*}
            \overline Q_\delta:=\{(s,t)\in\overline Q|s\ge\delta,\,t\ge\delta\}.
        \end{equation*}
        We break the argument in three steps: in the first two we get some preliminary estimates, while the core of the proof is contained in the final step. \\
        \textbf{Step 1.}
        We claim that, given $\delta>0$, there is $C_\delta>0$ with
        \begin{equation}\label{eq:unifcontloc1}
            |s_0-\ov s|\le C_\delta(\ov t-t_0),
        \end{equation}
        for any $(s_0,t_0)\in Q$, $(\ov s,\ov t)\in Q_\delta$ such that
        \begin{equation*}
            w_0(\ov s, \ov t)= g(s_0,t_0) + \Phi_\gamma(s_0,t_0, \ov s, \ov t).
        \end{equation*}
        Compare with the proof of Lemma \ref{ester}. More precisely, we will prove that there is a constant $\eps_\delta$ with
        \begin{equation}\label{eq:unifcontloc2}
            \ov t-t_0\ge \epsilon_\delta
        \end{equation}
        which implies \eqref{eq:unifcontloc1} taking
        \begin{equation*}
            C_\delta:=\frac1{\epsilon_\delta}\ge\frac{|\ov s-s_0|}{\ov t-t_0}.
        \end{equation*}
        If $t_0=0$ then \eqref{eq:unifcontloc2}follows with $\epsilon_\delta=\delta$. Let us then assume $t_0>0$, and consequently $s_0=0$. We first discuss the case where
        \begin{equation}\label{grilli1}
            \ov t\ge\ov s,
        \end{equation}
        and set
        \begin{equation}\label{eq:unifcontloc5}
            R_1:=\max\limits_{s\in[0,1],|\lambda|\le1}L_\gamma(s,\lambda).
        \end{equation}
        By \textbf{(P3)}, there exists, for any $A>0$, a positive constant $B_A$ with
        \begin{equation}\label{eq:unifcontloc4}
            L_\gamma(s,\lambda)>A|\lambda|-B_A,\txt{for any }(s,\lambda)\in[0,1]\times\R.
        \end{equation}
        We denote by $\eta$ an optimal curve for $w_0(\bar s, \bar t)$, then the fact that $\eta(t_0)= s_0=0$ and \eqref{eq:unifcontloc4} imply
        \begin{equation}\label{eq:unifcontloc6}
            \begin{aligned}
                g(0,\ov t-\ov s)+R_1\ov s\ge\;& g(0,\ov t-\ov s)+\int_{\ov t-\ov s}^{\ov t}L_\ga(\tau,1)d\tau\ge g(0,t_0)+\int_{t_0}^{\ov t}L_\gamma(\eta,\dot \eta)d\tau\\
                \ge\;&g(0,t_0)+A\ov s-B_A(\ov t-t_0)\ge g(0,t_0)+A\delta-B_A(\ov t-t_0).
            \end{aligned}
        \end{equation}
        We denote by $\om_g$ a concave uniform continuity modulus for $g$, which therefore satisfies
        \begin{equation*}
            \om_g(r) \le a \, r +b \txt{for suitable $a$, $b$ positive}.
        \end{equation*}
        By combining the two above estimates, we get
        \begin{equation*}
            a|\ov t-\ov s-t_0|+b+R_1 \ov s\ge\om_g(|\ov t-\ov s-t_0|)+R_1\ov s\ge g(0,\ov t-\ov s)-g(0,t_0)+R_1\ov s\ge A\delta-B_A(\ov t-t_0).
        \end{equation*}
        Since
        \begin{equation*}
            a+a(\ov t-t_0)\ge a\ov s+a(\ov t-t_0)\ge a|\ov t-\ov s-t_0|,
        \end{equation*}
        we further get
        \begin{equation*}
            (a+B_A)(\ov t-t_0)\ge a|\ov t-\ov s-t_0|-a+B_A(\ov t-t_0)\ge A\delta-a-b-R_1.
        \end{equation*}
        Since the constant $A$ can be chosen so that the rightmost term of the above inequality is strictly positive, the claim \eqref{eq:unifcontloc2} is proved, under the extra assumption \eqref{grilli1}, with
        \begin{equation}\label{eq:unifcontloc7}
            \epsilon_\delta:=\frac{A\delta-a-b-R_1}{a+B_A}.
        \end{equation}
        Next we assume that
        \begin{equation}\label{grilli2}
            \ov t<\ov s.
        \end{equation}
        We have that
        \begin{equation*}
            g(\ov s-\ov t,0)+R_1\ov t\ge g(0,t_0)+A\delta-B_A(\ov t-t_0).
        \end{equation*}
        suitably modifying the estimates in \eqref{eq:unifcontloc6}. Similarly we get
        \begin{equation*}
            a+a(\ov t-t_0)+b+R_1\ge a(\ov s-\ov t+t_0)+b+R_1\add{\ov t}\ge A\delta-B_A(\ov t-t_0),
        \end{equation*}
        and consequently
        \begin{equation*}
            (a+B_A)(\ov t-t_0)\ge A\delta-a-b-R_1,
        \end{equation*}
        which finally proves \eqref{eq:unifcontloc2} in the case \eqref{grilli2} with the same $\epsilon_\delta$ as in \eqref{eq:unifcontloc7}.\\
        \textbf{Step 2.} Given $(\ov s,\ov t)\in Q$ and $\eps>0$, $\eta:[t_0,\ov t]\to[0,1]$ optimal in \eqref{eq:unifcontloc.1} for $w_0(\ov s,\ov t)$, we claim that there exists a constant $R_\eps>0$, independent of $\ov s$, $\ov t$, with
        \begin{align}
            g(0,\ov t)-w_0(\ov s,\ov t)\le\;&R_\eps(\ov t-t_0+\ov s)+\eps,\label{eq:unifcontloc9}\\
            g(\ov s,0)-w_0(\ov s,\ov t)\le\;&R_\eps\ov t+\eps.\label{eq:unifcontloc10}
        \end{align}
        Exploiting the concavity of $\om_g$, we find a constant $a_\eps$ with
        \begin{equation*}
            a_\eps r+\eps\ge\om_g(r),
        \end{equation*}
        in addition by \eqref{eq:unifcontloc4} there exists $B_\eps$ satisfying
        \begin{equation*}
            L_\gamma(s,\lambda)\ge a_\eps|\lambda|-B_\eps,\txt{for any }(s,\lambda)\in[0,1]\times\R.
        \end{equation*}
        We therefore have
        \begin{align*}
            g(0,\ov t)-w_0(\ov s,\ov t)=\;&g(0,\ov t)-g(\eta(t_0),t_0)-\int_{t_0}^{\ov t}L_\gamma(\eta,\dot\eta)d\tau\\
            \le\;&\om_g(\eta(t_0)+\ov t-t_0)-\int_{t_0}^{\ov t}L_\gamma(\eta,\dot\eta)d\tau\\
            \le\;&a_\eps\eta(t_0)+a_\eps(\ov t-t_0)+\eps-a_\eps|\ov s-\eta(t_0)|+B_\eps(\ov t-t_0)\\
            \le\;&(a_\eps+B_\eps)(\ov t-t_0)+a_\eps\eta(t_0)-a_\eps|\ov s-\eta(t_0)|+\eps\\
            \le\;&(a_\eps+B_\eps)(\ov t-t_0)+a_\eps\ov s+\eps,
        \end{align*}
        which yields \eqref{eq:unifcontloc9} with $R_\eps=a_\eps+B_\eps$. Similarly
        \begin{align*}
            g(\ov s,0)-w_0(\ov s,\ov t)=\;&g(\ov s,0)-g(\eta(t_0),t_0)-\int_{t_0}^{\ov t}L_\gamma(\eta,\dot\eta)d\tau\\
            \le\;&\om_g(|\ov s-\eta(t_0)|+t_0)-\int_{t_0}^{\ov t}L_\gamma(\eta,\dot\eta)d\tau\\
            \le\;&a_\eps|\ov s-\eta(t_0)|+a_\eps t_0+\eps-a_\eps|\ov s-\eta(t_0)|+B_\eps(\ov t-t_0)\\
            \le\;&B_\eps(\ov t-t_0)+a_\eps t_0+\eps,
        \end{align*}
        yielding \eqref{eq:unifcontloc10} with $R_\eps=a_\eps+B_\eps$. We further claim that
        \begin{align}
            w_0(\ov s,\ov t)-g(0,\ov t)\le\;&R'_\eps\ov s+\eps,&&\text{if }\ov t\ge\ov s,\label{eq:unifcontloc11}\\
            w_0(\ov s,\ov t)-g(\ov s,0)\le\;&R''_\eps\ov t.\label{eq:unifcontloc12}
        \end{align}
        for suitable positive constants $R'_\eps$, $R''_\eps$. We take $R_1$ defined as in \eqref{eq:unifcontloc5} and $\ov t\ge\ov s$ to get
        \begin{align*}
            w_0(\ov s,\ov t)-g(0,\ov t)\le\;&\int_0^{\ov s}L_\gamma(\tau,1)d\tau+g(0,\ov t-\ov s)-g(0,\ov t)\\
            \le\;&R_1\ov s+\om_g(\ov s)\le R_1\ov s+a_\varepsilon\ov s+\eps,
        \end{align*}
        which implies \eqref{eq:unifcontloc11} with $R'_\eps=R_1+a_\eps$. Finally, setting $R_0$ as in Notation \ref{inizio}, we have
        \begin{equation*}
            w_0(\ov s,\ov t)-g(\ov s,0)\le\int_0^{\ov t}L_\gamma(\ov s,0)d\tau+g(\ov s,0)-g(\ov s,0)\le R_0\ov t,
        \end{equation*}
        which proves \eqref{eq:unifcontloc12} with $R''_\eps=R_0$.\\
        \textbf{Step 4.} In order to prove the uniform continuity of $w_0$ and conclude the proof, we will show that if $(s_n,t_n)$, $(s'_n,t'_n)$ are sequences in $Q$ with
        \begin{equation*}
            |s_n-s'_n| + |t_n - t'_n|\longrightarrow0,
        \end{equation*}
        then
        \begin{equation}\label{eq:unifcontloc13}
            |w_0(s_n',t_n')-w_0(s_n,t_n)|\longrightarrow 0.
        \end{equation}
        Possibly interchanging the role of $(s_n,t_n)$ and $(s'_n,t'_n)$, we can assume that $w_0(s_n',t_n')\ge w_0(s_n,t_n)$, up to subsequences, so that
        \begin{equation}\label{grilli01}
            |w_0(s_n',t_n')-w_0(s_n,t_n)|=w_0(s_n',t_n')-w_0(s_n,t_n).
        \end{equation}
        We denote by $(s_{0,n},t_{0,n})$ elements such that
        \begin{equation*}
            w_0(s_n,t_n)=g(s_{0,n},t_{0,n})+\Phi_\gamma(s_{0,n},t_{0,n},s_n,t_n).
        \end{equation*}
        We focus on the case where
        \begin{equation}\label{grilli001}
            \lim_n d((s_n,t_n),\partial_p^-Q)= \lim_n d((s'_n,t'_n),\partial_p^-Q)=0,
        \end{equation}
        and $s_n$, $s'_n$, $t_n$, $t'_n$ have limit, maybe infinite, as $n \to +\infty$.

        We first consider the subcase where
        \begin{eqnarray}
            t_n-t_{0,n}& \longrightarrow& 0\label{eq:unifcontloc14}\\
            t_n &\ge& s_n \txt{up to subsequences} \nonumber
        \end{eqnarray}
        which implies that $s_n$, $s'_n$ are infinitesimal otherwise \eqref{grilli001} should be violated. Inequalities \eqref{eq:unifcontloc9}, \eqref{eq:unifcontloc11}and \eqref{grilli01} yield for an arbitrary $\eps>0$,
        \begin{align*}
            |w_0(s_n',t_n')-w_0(s_n,t_n)|=\;&[w_0(s_n',t_n')-g(0,t_n')]+[g(0,t_n')-g(0,t_n)]\\
            &+[g(0,t_n)-w_0(s_n,t_n)]\\
            \le\;&2\eps+R'_\eps s_n'+\om_g(|t_n-t_n'|)+R_\eps(t_n-t_{0,n}+s_n).
        \end{align*}
        Since $t_n-t_{0,n}$, $t_n-t'_n$ go to 0 as $n\to\infty$, and the same holds true for $s_n'$, we derive
        \begin{equation*}
            \limsup_n |w_0(s_n',t_n')-w_0(s_n,t_n)| \le \eps
        \end{equation*}
        which proves \eqref{eq:unifcontloc13} since $\eps$ has been arbitrarily chosen. Next, we focus on the subcase where
        \begin{eqnarray}
            t_n-t_{0,n}& \longrightarrow& 0\label{eq:unifcontloc15}\\
            t_n &< & s_n \txt{up to subsequences} \nonumber
        \end{eqnarray}
        which implies by \eqref{grilli001} that $t_n$ is infinitesimal. We take an arbitrary $\eps >0$, and exploit \eqref{eq:unifcontloc10}, \eqref{eq:unifcontloc12} and \eqref{grilli01} to get
        \begin{align*}
            |w_0(s_n',t_n')-w_0(s_n,t_n)|=\;&[w_0(s_n',t_n')-g(s'_n,0)]+[g(s_n',0)-g(s_n,0)]\\
            &+[g(s_n,0)-w_0(s_n,t_n)]\\
            \le\;&\eps+R''_\eps t_n'+\om_g(|s_n-s_n'|)+R_\eps t_n,
        \end{align*}
        which yields \eqref{eq:unifcontloc13} since $t_n$, $t'_n$, $|s_n-s_n'|$ go to 0 as $n\to\infty$, and $\eps$ is arbitrary. Now, still keeping the hypothesis \eqref{grilli001}, assume \eqref{eq:unifcontloc14} and \eqref{eq:unifcontloc15} to be false, which implies that, up to subsequences,
        \begin{equation}\label{eq:unifcontloc16}
            t_n-t_{0,n}> 2 \kappa,\txt{for $n$ big enough and some $\kappa >0$ },
        \end{equation}
        and consequently
        \begin{equation}\label{grilli02}
            t_n'-t_{0,n}> \kappa,\txt{for $n$ big enough.}
        \end{equation}
        We have by \eqref{grilli01}
        \begin{align}
            |w_0(s_n',t_n')-w_0(s_n,t_n)|=\;&w_0(s_n',t_n')-w_0(s_n,t_n) \nonumber\\
            \le\;&g(s_{0,n},t_{0,n})+\Phi_\gamma(s_{0,n},t_{0,n},s_n',t_n')\label{grilli007} \\
            &-g(s_{0,n},t_{0,n})-\Phi_\gamma(s_{0,n},t_{0,n},s_n,t_n) \nonumber\\
            =\;&\Phi_\gamma(s_{0,n},t_{0,n},s_n',t_n')- \Phi_\gamma(s_{0,n},t_{0,n},s_n,t_n). \nonumber
        \end{align}
        By \eqref{eq:unifcontloc16}, \eqref{grilli02}
        \begin{equation*}
            \kappa^{-1}(t_n-t_{0,n})\ge1\ge|s_n-s_{0,n}|\qquad\text{and}\qquad\kappa^{-1}(t_n'-t_{0,n})\ge1\ge|s_n'-s_{0,n}|,
        \end{equation*}
        therefore $(s_{0,n},t_{0,n},s_n',t_n')$, $(s_{0,n},t_{0,n},s_n,t_n)$ belong to a subset of $\ov Q^2$ of the form \eqref{eq:unifcontloc8} with $C=\kappa^{-1}$. Since $\Phi_\gamma$ is Lipschitz continuous in such a set, as recalled at the beginning of the proof, \eqref{eq:unifcontloc13} follows, and it is therefore fully proved under the extra assumption \eqref{grilli001}. It is left to discuss the case where
        \begin{equation*}
            (s_n,t_n), \; (s'_n,t'_n) \in Q_\de \txt{for a suitable $\de >0$.}
        \end{equation*}
        By \textbf{Step 1.} there is $C_\de$ with
        \begin{equation*}
            |s_n - s_{0,n}| \le C_\delta( t_n - t_{0,n})
        \end{equation*}
        and consequently
        \begin{equation*}
            |s'_n - s_{0,n}| \le \frac{C_\delta}2( t'_n - t_{0,n}) \txt{for $n$ large.}
        \end{equation*}
        Arguing as in \eqref{grilli007} we again deduce \eqref{eq:unifcontloc13} and conclude the proof.
    \end{proof}

    \begin{Lemma}\label{calmas}
        Given $w \in UC \big ( (\NN \times \{0\}) \cup ( \VN \times (0,+\infty)) \big )$, we have
        \begin{equation}\label{calmas1}
            F_\ga[w](s,t) = \min \left \{ w(\ga(s_0),t_0) + \int_{t_0}^t L_\ga(\eta,\dot \eta) \, d\tau \right \},
        \end{equation}
        where the minimum is over the elements $(s_0,t_0)$ varying in $ \partial^-_p Q$ with $t_0 <t$, and the curves $\eta:[t_0,t] \to [0,1]$ with $\eta(t_0)=s_0$ and $\eta(t)=s$.
    \end{Lemma}
    \begin{proof}
        It is straightforward to show that the minimum in \eqref{calmas1} is actually achieved. To ease notation, we denote by $w_0(s,t)$ the function defined by the right-hand side of \eqref{calmas1}. We have to show:
        \begin{itemize}
            \item[1.] The function $w_0$ is uniformly continuous in $Q \cup \partial^-_p Q$;
            \item[2.] the function $w_0$ is subsolution to \eqref{HJg};
            \item[3.] the function $w_0$ is the maximal subsolution not exceeding $w \circ \ga$ on $\partial^-_p Q$.
        \end{itemize}
        For the item 1, see Proposition \ref{unifcontloc}. The point 2 is a consequence of \cite[Lemma A.1]{PozSic}. The function $w_0$ coincides with $w \circ \ga$ in $[0,1] \times \{0\}$, while for any $t_0 >0$ we have by \eqref{calmas1}
        \begin{equation*}
            w_0(0,t_0) \le w(\ga(0),t_0-\de) + \de \, L_\ga(0,0)
        \end{equation*}
        for $\de >0$ small enough, which shows, sending $\de$ to $0$, that $w_0\le w \circ \ga$ at $(0,t_0)$, and so
        \begin{equation}\label{calmas2}
            F_\ga[w](s,t) \ge w_0(s,t).
        \end{equation}
        Finally, \cite[Proposition 5.9]{Sic} yields that there is a maximal trace $\ov w$ on $\partial^-_p Q$, not exceeding $w \circ \ga$, that can be continuously extended to a uniformly continuous subsolution to \eqref{HJg} on the whole of $Q$. Then, we have that (see \cite[Theorem 6.5]{PozSic}):
        \begin{equation*}
            F_\ga[w](s,t)= \min\left \{ \overline w(s_0,t_0) + \int_{t_0}^t L_\ga(\eta,\dot \eta) \, d\tau \right \},
        \end{equation*}
        where $(s_0,t_0)$ varies in $ \partial^-_p Q$ with $t_0 <t$, and $\eta:[t_0,t] \to [0,1]$ is any curve with $\eta(t_0)=s_0$ and $\eta(t)=s$. This shows \ that $F_\ga[w](s,t) \le w_0(s,t)$, and concludes the proof in view of \eqref{calmas2}.
    \end{proof}

    We proceed now by defining the operator
    \begin{equation*}
        G: C([0,+\infty)) \times \R \longrightarrow C([0,+\infty))
    \end{equation*}
    as
    \begin{equation*}
        G[\psi,a](t):= \min \{\psi(r) + a \, (t-r) \mid r \in [0,t]\}.
    \end{equation*}

    This definition is justified by the following characterization:

    \begin{Lemma}
        \emph{\cite[Lemma 4.4]{Sic}} $G[\psi,a]$ is the maximal continuous function in $[0,+\infty)$ dominated by $\psi$ with
        \begin{equation*}
            \frac d{dt} \varphi(t) \le a \qquad t \in (0,+\infty)
        \end{equation*}
        for any $C^1$ supertangent $\varphi$ to $G[\psi,a]$ at $t$.
    \end{Lemma}

    \smallskip

    We consider the problem
    \begin{equation}\label{discr}
        \tag{Discr} \left \{
        \begin{array}{cc}
            w(z,t)= G[F_z[w], \wha c_z] & \quad\text{for $(z,t) \in \VN \times (0,+\infty)$} \\
            w(z,0) = g(z) & \quad\text{for $z \in \NN$.}
        \end{array}
        \right .
    \end{equation}
    We say that
    \begin{equation*}
        w_0\in UC \big ( (\NN \times \{0\}) \cup ( \VN \times (0,+\infty)) \big )
    \end{equation*}
    is a subsolution to\eqref{discr} if
    \begin{eqnarray*}
        w_0(z, t) &\le& G[F_z[w_0], \wha c_z](t) \qquad \text{for any $z \in \VN, t \in (0, + \infty)$,} \\
        w_0(z,0) &\le& g(z) \qquad\text{for $z \in \NN$,}
    \end{eqnarray*}
    the definition of supersolution is given replacing $\le$ by $\ge$ in the above formulae. A solution is a super and a subsolution at the same time.

    The connection of this problem with \eqref{eq:HJ} is given by the following result.

    \begin{Proposition}\label{filone}
        \hfill
        \begin{itemize}
            \item[{\bf (i)}] If $u$ is a uniformly continuous subsolution to \eqref{eq:HJ} then the trace of $u$ on $(\NN \times \{0\}) \cup ( \VN \times (0,+\infty))$ is a subsolution to \eqref{discr}.
            \item[{\bf (ii)}] If $v$ is a uniformly continuous supersolution to \eqref{eq:HJ} then the trace of $v$ on $(\NN \times \{0\}) \cup ( \VN \times (0,+\infty))$ is a supersolution to \eqref{discr}.
        \end{itemize}
    \end{Proposition}

    The proofs given in \cite{Sic}, see Propositions 6.4 and 6.5, are obtained arguing on a single arc and so still hold in our context.

    \medskip

    \begin{Lemma}\label{subdued}
        Given a uniformly continuous subsolution $u$ to \eqref{eq:HJ}, and $(z,t) \in \VN \times (0,+\infty)$, for any admissible curve $\zeta:[0,t] \to \NN$ with $\zeta(t)= z$, we get
        \begin{equation*}
            g(\zeta(0)) + \int_0^t L(\zeta, \dot\zeta) \, d\tau\ge u(z,t).
        \end{equation*}
    \end{Lemma}
    \begin{proof}
        We denote by $\{t_i\}$, $i=0, \cdots, m$, for a suitable $m$, the finite partition of $[0,t]$ in correspondence with $\zeta$ given in Definition \ref{adcurve}, note that $t_0=0$ and $t_m=t$. It will be proved by finite induction on $i$ that the inequality
        \begin{equation*}
            g(\zeta(0)) + \int_0^{t_i} L(\zeta, \dot\zeta) \, d\tau\ge u(\xi(t_i),t_i)
        \end{equation*}
        holds for any $i \in \{1, \cdots, m\}$. The above relation is trivially true for $i=0$. Let us assume that it holds for some $i \in \{1, \cdots, m-1\}$ and prove it for $i+1$. Exploiting the inductive assumption, we can write
        \begin{eqnarray*}
            g(\zeta(0)) + \int_0^{t_{i+1}} L(\zeta, \dot\zeta) \, d\tau&=& g(\zeta(0)) + \int_0^{t_i} L(\zeta, \dot\zeta) \, d\tau+ \int_{t_i}^{t_{i+1}} L(\zeta,\dot\zeta) \, d\tau\\
            &\ge&u(\zeta(t_i),t_i) + \int_{t_i}^{t_{i+1}} L(\zeta,\dot\zeta) \, d\tau.
        \end{eqnarray*}
        It is therefore enough to prove
        \begin{equation}\label{subdued2}
            u(\zeta(t_i),t_i) + \int_{t_i}^{t_{i+1}} L(\zeta,\dot\zeta) \, d\tau\ge u(\zeta(t_{i+1}),t_{i+1}).
        \end{equation}
        We first assume that $\zeta$ restricted to $[t_i,t_{i+1}]$ is nonconstant and with support contained in an arc $ \ga$ of $ \NN$ with $ \ga(1)= \zeta(t_{i+1})$. We thus have by Proposition \ref{lemlemnew}
        \begin{equation*}
            \int_{t_i}^{t_{i+1}} L(\zeta,\dot\zeta) \, d\tau= \int_{t_i}^{t_{i+1}} L_\ga( \Upsilon^{-1}(\zeta),D_\tau( \Upsilon^{-1}(\zeta)) \, d\tau.
        \end{equation*}
        Hence, by exploiting Lemma \ref{calmas}, the definition of the semidiscrete problem and that $u$, restricted to $\VN \times (0,+\infty)$, is subsolution to \eqref{discr} by Proposition \ref{filone}, we get
        \begin{eqnarray*}
            u(\zeta(t_i),t_i) + \int_{t_i}^{t_{i+1}} L(\zeta,\dot\zeta) \, d\tau&=& u (\ga(\Upsilon^{-1} \zeta(t_i)),t_i) + \int_{t_i}^{t_{i+1}} L_\ga( \Upsilon ^{-1} (\zeta), D_\tau( \Upsilon^{-1}(\xi)) \, d\tau\\
            &\ge& F_{ \ga}[u](1,t_{i+1}) \ge F_{\zeta(t_{i+1})}[u](t_{i+1}) \\ &\ge& G[F_{\zeta(t_{i+1})}[u], \wha c_{\zeta(t_{i+1})}](t_{i+1}) \ge u({\zeta(t_{i+1})},t_{i+1}),
        \end{eqnarray*}
        which completes the induction argument in the case where $\zeta$ is nonconstant in $[t_i,t_{i+1}]$. We proceed assuming instead that
        \begin{equation*}
            \zeta(t) \equiv z \in \VN \qquad\text{for $t \in [t_i,t_{i+1}]$,}
        \end{equation*}
        by exploiting that $u$ is subsolution to \eqref{eq:HJ}, we get
        \begin{eqnarray*}
            u(\zeta(t_i),t_i) + \int_{t_i}^{t_{i+1}} L(\zeta,\dot\zeta) \, d\tau&=& u( z,t_i) + \wha c_{z} \, (t_{i+1}-t_i) \\
            &\ge& u(z,t_{i+1}).
        \end{eqnarray*}
        This shows \eqref{subdued2} and completes the argument.
    \end{proof}

    \smallskip

    \begin{Lemma}\label{superdued}
        Given a uniformly continuous supersolution $v$ to \eqref{eq:HJ}, and $( z_0,t_0)\in\VN \times (0,+\infty)$, there exists an admissible curve $\xi$ defined in $[0,t_0]$ with $\xi(t_0)= z_0$ such that
        \begin{equation*}
            v(z_0,t_0) \ge g(\xi(0)) + \int_0^{t_0} L(\xi, \dot \xi) \, d\tau.
        \end{equation*}
    \end{Lemma}
    \begin{proof}
        The curve $\xi$ will be constructed through a backward procedure starting from $(z_0,t_0)$. We select $r \in [0,t_0]$ with
        \begin{equation}\label{superdued1}
            G[F_{ z_0}[v], \wha c_{z_0}](t_0)= F_{z_0}[v](r) + \wha c_{z_0} (t_0-r).
        \end{equation}
        We proceed considering an arc $\ga_1$ ending at $z_0$ with
        \begin{equation*}
            F_{z_0}[v](r)=F_{\ga_1}[v](1,r).
        \end{equation*}
        By Lemma \ref{calmas} there is a time $t_1 <r$ and a curve $\eta_1: [t_1,r] \to [0,1]$ linking $\eta_1(t_1)$ to $1$ such that $( \eta_1(t_1),t_1) \in \partial^- Q$ and
        \begin{equation}\label{superdued2}
            F_{z_0}[v](r) = v( \ga_1(\eta_1(t_1)),t_1) + \int_{t_1}^r L_{ \ga_1}(\eta_1, \dot \eta_1) \, d\tau.
        \end{equation}
        We then define a curve $\zeta_1: [t_1, t] \to \NN$ setting
        \begin{equation*}
            \zeta_1(\tau) = \left \{
            \begin{array}{cc}
                \ga_1 \circ \eta_1(\tau) & \quad \text{if $\tau \in [t_1,r]$} \\
                z_0 & \quad \text{if $\tau \in (r, t]$}.
            \end{array}
            \right .
        \end{equation*}
        Taking into account \eqref{superdued1}, \eqref{superdued2} and that $v$ is a supersolution to \eqref{discr}, we have
        \begin{equation*}
            v( z_0,t_0) \ge G[F_{z_0}[v],\wha c_{z_0}](t) = v(\zeta_1(t_1),t_1) + \int_{t_1}^t L(\zeta_1,\dot\zeta_1) \, d\tau.
        \end{equation*}
        If $t_1 >0$, then $\zeta_1(t_1) = \ga_1(0):=z_1 \in \VN$ because $( \eta_1(t_1),t_1) \in \partial^- Q$, this implies that
        \begin{equation}\label{superdued22}
            |\zeta_1| \ge |\ga_1|.
        \end{equation}
        We iterate the above procedure starting from the vertex $z_1$, constructing in this way a decreasing sequence of times $t_n$ (finite or infinite), and a sequence of concatenated curves $\zeta_n:[t_n,t_{n-1}] \to \NN$ contained in arcs $\ga_n$ with
        \begin{equation}\label{superdued3}
            v(\zeta_n(t_{n-1}),t_{n-1}) \ge v(\zeta_n(t_n),t_n) + \int_{t_n}^{t_{n-1}} L(\zeta_n,\dot\zeta_n) \, d\tau.
        \end{equation}
        Arguing as for \eqref{superdued22} we further get that
        \begin{equation}\label{superdued33}
            |\zeta_n| \ge |\ga_n|.
        \end{equation}
        We claim that the sequence $t_n$ is finite. If this is not the case, we arbitrarily fix an index $n_0 >0$ and a time $t_{n_0}$, and define $\zeta:[t_{n_0},t_0] \to \NN$ via
        \begin{equation}\label{superdued35}
            \zeta(\tau)= \zeta_n(\tau) \qquad\text{for $\tau \in [t_n,t_{n-1}]$, $n \le n_0$.}
        \end{equation}
        We glue together the inequalities in \eqref{superdued3} to get
        \begin{equation}\label{superdued351}
            v( z_0,t_0) \ge v(\zeta (t_{n_0}),t_{n_0}) + \int_{t_{n_0}}^{t_0} L(\zeta,\dot\zeta) \, d\tau.
        \end{equation}
        We further derive from \eqref{superdued33}
        \begin{equation}\label{superdued333}
            |\zeta| \ge \sum_{i=1}^{n_0} |\ga_i|.
        \end{equation}
        We argue as in Lemma \ref{nosleep} exploiting the superlinearity assumption on the $ L_{ \ga}$'s, to obtain for any $k \in \N$ and suitable constants $B_k$
        \begin{equation*}
            \int_{t_{n_0}}^t L(\zeta,\dot\zeta) \, d\tau \ge k \, |\zeta| - B_k (t -t_{n_0}),
        \end{equation*}
        which implies by \eqref{superdued351} that
        \begin{eqnarray}
            v(z_0,t_0) - v(\zeta (t_{n_0}),t_{n_0}) &\ge& k \, |\zeta| -B_k (t_0 - t_{n_0})\label{superdued4} \\
            &\ge& k \, d_\NN(z_0,\zeta(t_{n_0})) -B_k (t_0- t_{n_0}). \nonumber
        \end{eqnarray}
        We derive from \eqref{superdued333} and condition \textbf{(P1)} that
        \begin{equation*}
            \lim_{n_0 \to +\infty} v(\zeta (t_{n_0}),t_{n_0}) = - \infty,
        \end{equation*}
        and consequently
        \begin{equation*}
            \lim_{n_0 \to +\infty} d_\NN(\zeta(t_{n_0}),z_0) = + \infty.
        \end{equation*}
        Inequality \eqref{superdued4} yields
        \begin{equation*}
            \frac{|v(\zeta (t_{n_0}),t_{n_0})|}{d_{\NN}(\zeta(t_{n_0}),z_0)} \ge \frac{-v( z_0,t_0) -B_k (t_0- t_{n_0})}{d_{\NN}(\zeta(t_{n_0}),z)} + k
        \end{equation*}
        and passing to the limit
        \begin{equation*}
            \liminf_{n_0 \to + \infty} \frac{|v(\zeta (t_{n_0}),t_{n_0})|}{d_{\NN}(\zeta(t_{n_0}),z_0)} \ge k,
        \end{equation*}
        and, since $k \in \N$ is arbitrary
        \begin{equation*}
            \liminf_{n_0 \to + \infty} \frac{|v(\zeta (t_{n_0}),t_{n_0})|}{d_{\NN}(\zeta(t_{n_0}),z_0)} = + \infty,
        \end{equation*}
        which is in contradiction with $v$ being uniformly continuous.

        We deduce that the sequence $\{t_n\}$ must be finite, with minimum, or last element, equal to $0$. Therefore the curve $\zeta$ defined as in \eqref{superdued35} with $0$ in place of $t_{n_0}$ is an admissible curve in $[0,t]$ satisfying the inequality in the statement.
    \end{proof}

    \smallskip

    We can now prove the main statement of this section.

    \begin{proof}[{\bf Proof of Theorem \ref{compa}:}]
        For any arc $\ga$, we have by Lemmata \ref{subdued} and \ref{superdued} that
        \begin{equation*}
            u \circ \ga \le v \circ \ga \qquad\text{on $\partial_p Q$}.
        \end{equation*}
        Then, by \cite{Sic} Theorem 5.2
        \begin{equation*}
            u \circ \ga \le v \circ \ga \qquad\text{on $\overline Q$.}
        \end{equation*}
        This concludes the proof.
    \end{proof}

    \medskip

    \bigskip

    \subsection{Proof of Theorem \ref{new}}\label{newnew}

    {In this section we provide the proof of one of the key results for our asymptotic analysis}

    \new*

    \smallskip

    Let us discuss first some preliminary results.

    \begin{Definition}
        A measure $\mu = \sum_{e\in \EE_0} \la_e \, \mu_e$ is said to be {\it nonsingular } if
        \begin{equation*}
            0 \not\in \spt \mu_e \txt{for any $ e \in \spt_{\EE_0} \mu$.}
        \end{equation*}
    \end{Definition}

    \medskip

    \begin{Proposition}\label{lubiana}
        Let $\mu \in \M^h \cap \M_p$, for some conjugate pair $h$, $p$. Then, up to rewriting the singular part (without loss of generality), it can be written in the form
        \begin{equation}\label{lubiana2}
            \mu = (1 -\bar \la) \, \nu + \bar \la \, \de(\bar e ,0)
        \end{equation}
        where $\bar e$ is any edge with $a_{\bar e}=a_0$, $\bar\la \in [0,1]$ and, if $\bar \la <1$, $\nu$ is nonsingular and
        \begin{equation}\label{lubiana3}
            \nu \in \M^{\frac h{1-\bar \la}} \cap \M_p.
        \end{equation}
    \end{Proposition}
    \begin{proof}
        Assume $\bar \la <1$. Then $\nu$ is a closed probability measure by Proposition \ref{proprhobis}. Formula \eqref{lubiana2} is then a direct consequence of Proposition \ref{alfinale}. We proceed proving \eqref{lubiana3}. Since $\rho$ is affine for convex combination and $\rho(\de(\bar e,0))=0$, we have
        \begin{equation*}
            \rho(\nu) = \frac h {(1 -\bar \la)} =:\bar h
        \end{equation*}
        and
        \begin{eqnarray*}
            (1 - \bar \la) \, \be (\bar h) + \bar \la \, \be (0) & \ge& \be(h) = A(\mu) = (1 - \bar \la) \, A( \nu) + \bar \la \, \be(0),
        \end{eqnarray*}
        which implies, see \eqref{eq:minbe}, that
        \begin{equation}\label{corlubiana3}
            \be(\bar h)= A (\nu) = \frac 1{1- \bar \la} \, \big ( \be(h) + \bar \la \,a_0 \big )
        \end{equation}
        so that $\nu \in \M^{\bar h}$. We therefore have by \eqref{corlubiana3} and Propositions \ref{utilissima}, \ref{alfinale}
        \begin{eqnarray*}
            \langle \bar h,p\rangle & =& \frac 1{1- \bar \la} \, \langle h,p \rangle = \frac 1{1- \bar \la} \, ( \be (h) + \al(p)) = \be (\bar h) - \frac{\bar\la}{1-\bar \la} \, \al(p) + \frac 1{1-\bar \la} \, \al(p) \\
            &=& \be (\bar h) + \al(p),
        \end{eqnarray*}
        which gives, by Proposition \ref{utilissima}, that $\nu \in \M_p$, and concludes the proof.
    \end{proof}

    \smallskip

    \begin{Proposition}
        (\cite[Proposition 6.2]{SicSor1})\label{lubianone} Let $\mu = \sum_e \la_e \, \mu_e$ be a Mather measure and $e_0$ an edge with $\pm e_0 \in \spt_{\EE_0} \mu$, then $\mu_{\pm e_0} = \de(\pm e_0,0)$.
    \end{Proposition}

    \smallskip

    We proceed by establishing a crucial lemma that allows us putting in relation the action of a measure with the average action of a family of circuits taken with given multiplicities, see the forthcoming Remark \ref{remlequile}.

    \begin{Lemma}\label{lequile}
        Given $h \in \Z^{b(\G_0)} \setminus \{0\}$, $T >0$, a measure $\mu$ with $\rho(\mu) = \frac hT$ and action equal to $\be \left (\frac hT \right )$, we have
        \begin{eqnarray}\label{lequile1}
            A(\mu) &=& \frac 1T \, \left ( \sum_i n_i \, A_\LL(\xi_i) + \sum_{e \in \cup_i \spt \xi_i} \frac{m_e}{\mathcal Q_p(e)} \, \LL(e,\mathcal Q_p(e)) \right ) \\
            &-& \bar \la \, a_0 \nonumber \\
            T&=& \frac 1{1-\bar \la} \, \sum_{e \in \cup_i \spt \xi_i} \left [ \left ( \sum_{i \in I(e)} n_i \right ) \, \frac 1{\mathcal Q_p(e)} + \frac{m_e}{\mathcal Q_p(e)} \right ]\label{lequile001}\\
            h &=& \sum_{ e \in ( \cup_i \spt \xi_i) \setminus \EE_\TT} \left [ \left ( \sum_{i \in I(e)} n_i \right ) \theta(e) + m_e \, \theta(e) \right ]\label{lequile01}
        \end{eqnarray}
        where $\bar \la \in [0,1)$, the $\xi_i$'s are (finitely many) nontrivial parametrized circuits, $n_i \in \N$, $p \in H^1(\G_0,\R)$ is conjugate with $\frac hT$,
        \begin{equation*}
            I(e) = \{i \mid e \in \xi_i\},
        \end{equation*}
        the $m_e$'s are nonnegative quantities bounded from above by $C_0$, the number of circuits in $\G_0$, and, in addition
        \begin{equation}\label{lequile10}
            m_e \in \N \qquad\text{for any $e \in ( \cup_i \spt \xi_i) \setminus \EE_\TT$.}
        \end{equation}
    \end{Lemma}

    \medskip

    \begin{Remark}\label{remlequile}
        The salient relation is given by formula \eqref{lequile1}, which somehow compares the action $A(\mu)$ to the average action of a family of circuits $\xi_i$ taken with multiplicity $n_i$ plus an error given by the rightmost term containing the coefficients $m_e$. Since they are bounded, the error becomes negligible as $T$ goes to infinity.
    \end{Remark}

    \begin{proof}
        By Propositions \ref{utilissima}, \ref{lubiana}
        \begin{equation}\label{lubianina}
            \mu = (1-\bar \la) \, \nu + \bar \la \, \de(\bar e,0)
        \end{equation}
        where $\bar e$ is such that $a_{\bar e}=a_0$, $\bar \la \in [0,1)$ and $\nu$ is a nonsingular measure belonging to $\M^{\frac h{T (1- \bar \la)}} \cap \M_p$ for $p \in H^1(\G_0,\R)$ conjugate to $\frac h{T (1-\bar \la)}$. We first assume $\bar \la =0$ or in other terms that $\mu$ is nonsingular. We know from Propositions \ref{utilissima} that
        \begin{equation*}
            \mu = \sum_{i=1}^m \si_i \, \mu_i \qquad \si_i >0 ,\, \sum \si_i =1
        \end{equation*}
        for some occupation measures $\mu_i$, based on circuits $\xi_i$, belonging to $\M_p$. By Proposition \ref{alfinale}
        \begin{equation*}
            \xi_i=(e_j^i,\mathcal Q_p(e_j^i), 1/\mathcal Q_p(e_j^i))_{j=1}^{m_i}
        \end{equation*}
        and the total time of parametrization of the $\xi_i$'s is given by
        \begin{equation*}
            S_i = \sum_{j= 1}^{m_i} \frac 1 {\mathcal Q_p(e_j^i)} \qquad\text{for any $i$,}
        \end{equation*}
        which implies
        \begin{equation}\label{zorro2}
            \sum_{i =1}^m \left ( \sum_{e \in \spt \xi_i} \frac 1{\mathcal Q_p(e)} \right ) \, \si_i \, S_i = \sum_{i=1}^m \si_i =1.
        \end{equation}
        Since the action functional on the space of measures is affine for convex combinations, we get
        \begin{equation}\label{lequile2}
            A(\mu)= \sum_i \si_i \, A(\mu_i) = \sum_i \frac{\si_i}{S_i} \, A_\LL(\xi_i) = \sum_i \, \frac{\si_i}{S_i} \, \sum_j \frac 1 {\mathcal Q_p(e_j^i)} \, \LL(e_j^i, \mathcal Q_p(e_j^i)).
        \end{equation}
        We focus on the relation
        \begin{equation*}
            A(\mu)= \frac 1T \, \sum_i \frac{T \si_i}{S_i} \, A_\LL(\xi_i).
        \end{equation*}
        If $\frac{T \si_i}{S_i}$ were a positive integer, for any $i$, then $A(\mu)$ should be equal to the average action of the union of the circuits $\xi_i$, each one taken with multiplicity $\frac{T \si_i}{S_i}$. However, in general, this is not the case, and an error will appear. In what follows we look into this issue.

        By multiplying and dividing by $T$, we get
        \begin{eqnarray}
            \mu &=& \frac 1T \, \sum_{e \in \spt_{\EE_0} \mu } \, \left ( \sum_{i \in I(e)} \frac{\si_i T}{S_i} \right ) \frac 1{\mathcal Q_p(e)} \,\de(e,\mathcal Q_p(e))\label{lequile030} \\
            h &=& \sum_{e \in \spt_{\EE_0} \mu \setminus \EE_\TT} \, \left ( \sum_{i \in I(e)} \frac{\si_i T}{S_i} \right ) \, \theta(e).\label{lequile30}
        \end{eqnarray}
        Note that by Proposition \ref{lubianone} for no edge $e$ it can happen that both $\pm e$ belong to $\spt_{\EE_0} \mu$. Taking into account that $h \in \Z^{b(\G_0)}$, we have
        \begin{equation}\label{lequile31}
            \sum_{i \in I(e)} \frac{\si_i T}{S_i} \in \N \qquad\text{for any $e \in (\cup_i \spt \xi_i) \setminus \EE_\TT$.}
        \end{equation}
        We set
        \begin{eqnarray}
            m_e &:=& \sum_{i \in I(e)} \frac{\si_i T}{S_i} - \left \lfloor \frac{\si_i T}{S_i} \right \rfloor \le C_0 \txt{for any $e \in \cup_i \spt \xi_i$}\label{lequile4} \\
            n_i &:=& \left \lfloor \frac{\si_i T}{S_i} \right \rfloor, \nonumber
        \end{eqnarray}
        where $\lfloor \cdot \rfloor$ in the above formula stands for the integer part. We derive
        \begin{equation}\label{zorro1}
            \sum_{i \in I(e)} \frac{\si_i T}{S_i} = m_e + \sum_{i \in I(e)} n_i \txt{for any $e \in \cup \spt \xi_i$.}
        \end{equation}
        By \eqref{zorro2},\eqref{lequile2}, the fact that $\mu$ is a probability measure, \eqref{lequile030}, \eqref{lequile30} and \eqref{zorro1}, we deduce that
        \begin{eqnarray*}
            T &=& \sum_{e \in \spt_{\EE_0} \mu } \left ( \sum_{i \in I(e)} n_i \right ) \, \frac 1{\mathcal Q_p(e)} + \frac{m_e} {\mathcal Q_p(e)}\\
            h &=& \sum_{ e \in \spt_{\EE_0} \mu \setminus \EE_\TT } \left [ \left ( \sum_{i \in I(e)} n_i \right ) \theta(e) + m_e \, \theta(e) \right ] \\
            A(\mu) &=& \frac 1T \, \sum_i \frac{\si_i T}{S_i} \, A_\LL(\xi_i) = \frac 1T \, \left [ \sum_i n_i \, A_\LL(\xi_i) + \sum_i \left (\frac{\si_i T}{S_i} - n_i \right ) \, A_\LL(\xi_i) \right ].
        \end{eqnarray*}
        The first two equalities above correspond to \eqref{lequile001}, \eqref{lequile01}. The third one yields \eqref{lequile1}, with the $m_e$'s bounded from above by \eqref{lequile4}, once we observe that by \eqref{lequile2}
        \begin{equation*}
            \sum_i \left (\frac{\si_i T}{S_i} - n_i \right ) \, A_\LL(\xi_i) = \sum_{e \in \spt_{\EE_0} \mu} \, \frac{m_e}{\mathcal Q_p(e)} \, \LL(e,\mathcal Q_p(e)).
        \end{equation*}
        Finally, \eqref{lequile10} is a consequence of \eqref{lequile31}. This proves the assertion if $\bar \la =0$ in \eqref{lubianina}. It is left the case where $\bar\la \in (0,1)$. We can apply the above part of the proof to $\frac h {T (1-\bar \la)}$ taking the nonsingular measure $\nu$ in $\rho^{-1} \left (\frac h {T (1-\bar \la)} \right )$. Using \eqref{lubianina} and the relations between the new and old times, the new and old rotation vectors, it is straightforward to show the statement and complete the proof.
    \end{proof}

    \medskip

    \begin{Remark}[Compare with Remark \ref{positive}]\label{positivebis0}
        We point out that if we arbitrarily fix a constant $b$ and consider the family of modified Hamiltonians
        \begin{equation*}
            H^*_\ga(s,\rho) = H_\ga(s,\rho) + b \qquad\text{ for $\ga \in \EN$},
        \end{equation*}
        then the corresponding Hamiltonians and Lagrangians on the underlying graph are $\HH^*(e,\rho) = \HH(e,\rho) + b$ and $\LL^*(e,\la) =\LL(e,\la) -b$, respectively. In fact, $a^*_e = a_e+b$ and
        \begin{equation*}
            \si^*(e,a) = \si(e,a-b) \qquad\text{ for any $a \in [a^*_e, +\infty)$.}
        \end{equation*}
        Consequently the inverse of $\si^*(e,\cdot)$ is $\HH(e,\cdot) + b$ defined in $[\si(e,a), + \infty)$, as was claimed. The assertion on $\LL^*$ is then direct consequence of \eqref{deflag}. Since the homogenization procedure does not change if we replace $H_\ga$ by $H_\ga +b$, (up to change $\bar{H}$ with $\bar{H}+b$) we can choose the additive constant $b$ in such a way that the Lagrangians $\LL^*(e,\cdot)$ are strictly positive for any $e \in \EE_0$. We will make use in what follows of this property.
    \end{Remark}

    \medskip

    Let us prove the main result of this section.

    \begin{proof}[{\bf Proof of Theorem \ref{new}}]
        \noindent{\em Preliminaries:} \quad By Remark \ref{positivebis0}, we can assume without loosing generality that $\LL(e, \cdot)$ is strictly positive for any $e \in \EE_0$, and
        \begin{equation}\label{new2001}
            \de < \frac 12, \quad A> 1.
        \end{equation}
        We fix $A$, $\de$ satisfying \eqref{new2001}. Given $T >0$, we consider a measure $\mu$ with rotation vector $\frac h T$ and action equal to $\be \left ( \frac hT \right )$. This implies that $\mu \in \bM_p$ for $p \in H^1(\G_0,\R)$ conjugate to $\frac hT$. Since $ \left |\frac hT \right | < A$, then $p \in \partial \be(h/T)$ is bounded as well by a constant depending on $A$. We know from Proposition \ref{lubiana} that
        \begin{equation}\label{lubianozzo}
            \mu = (1- \bar \la) \, \nu + \bar \la \, \de(\bar e,0),
        \end{equation}
        where $\bar \la \in [0,1]$, $\bar e$ satisfies $a_{\bar e}= a_0$ and, if $\bar\la \ne 1$, $\nu$ is a nonsingular measure belonging to $\bM^{\frac h{1-\bar\la}} \cap \bM_p$. We deduce by Corollary \ref{coralfinale} that
        \begin{equation}\label{new007}
            \frac 1{\mathcal Q_c(e)} \ge a \txt{for a suitable $a >0$, any $e \in \spt_{\EE_0} \mu \setminus\{\bar e\}$.}
        \end{equation}
        Taking into account that $\LL(e,\cdot)$ is strictly positive and $C^1$, we can in addition select $b > a$ with
        \begin{equation}\label{new201}
            \frac d{dt} \; [ t \, \LL(e, 1/t)] = \LL(e, 1/t) - \frac 1t \, \LL'(e,1/t)> 0 \quad\text{ for any $t > \frac b2$, $e \in \EE_0$.}
        \end{equation}
        We denote by $M > 1$ a quantity such that
        \begin{itemize}
            \item[--] is an upper bound of $\LL(e,1)$ for $e$ varying in $ \EE_0$;
            \item[--] is a Lipschitz constant for $\LL(e,\cdot)$ in $[0, 1/a]$, any $e \in \EE_0$;
            \item[--] is a Lipschitz constant for $\beta$ in $B(0,A)$;
            \item[--] is a Lipschitz constant in $[a/2,b]$ for $t \mapsto t \, \LL(e,1/t)$.
        \end{itemize}
        We require $T$ to satisfy
        \begin{eqnarray}
            \frac{|\EE_0| \be(0)}T & < & \frac \de 2\\
            \frac{|\EE_0| M A}{T} &<& \frac \de 4\label{new00}\\
            \frac{M(3 C_0+4) |\EE_0|^2}T &<& \frac \de 2,\label{new00bis}
        \end{eqnarray}
        where $|\EE_0|$ is the number of edges and $C_0$ is the number of circuits in $\G_0$.

        \medskip

        \noindent{\em First estimate:} \quad We show that
        \begin{equation*}
            \frac{\Phi_\LL(x,y,T;h)}T \ge \be \left ( \frac h T \right ) - \de.
        \end{equation*}
        We consider a parametrized path $\xi=(e_i, q_i, T_i)_{i=1}^m$ in $\G_0$ with $\sum_i T_i=T$, $ \theta([\xi])=h$, linking $ x$ to $y$ which is $\frac \de 4$--optimal for $\Phi_\LL(x,y,T;h)$. We further consider a parametrized path $\zeta = (f_j,1,1)_{j=1}^k$, contained in $\TT$ and linking $y$ to $ x$. The total length of the parametrization of $\zeta$, denoted by $T_\zeta$, is clearly less than or equal to $|\EE_0|$. By concatenating $\xi$ and $\zeta$, we get a parametrized cycle $\xi \cup \zeta$ in $\G_0$, with total length of parametrization $T+T_\zeta$, and
        \begin{equation*}
            \theta([\xi \cup \zeta]) = \sum_i \theta(e_i)= h.
        \end{equation*}
        The closed occupation measure corresponding to $ \xi \cup \zeta$ will be denoted by $\nu$. We have
        \begin{equation*}
            \rho(\nu) = \frac{\theta([\xi \cup \zeta])}{T+T_\zeta} = \frac h{T+ T_\zeta},
        \end{equation*}
        and
        \begin{equation}\label{new123}
            A(\nu) = \frac 1{T+ T_\zeta} \, \left [ \sum_i T_i \LL(e_i, q_i) + \sum_j \LL(f_j, 1) \right ] \ge \be \left ( \frac h {T+ T_\zeta} \right ).
        \end{equation}
        We derive from \eqref{new00}
        \begin{equation}\label{new0000}
            \left | \frac h T - \frac h{T+ T_\zeta} \right | = \frac{T_\zeta}{ T+ T_\zeta} \, \left | \frac h T \right | \le \frac{| \EE_0|}T \, A \le\frac \de 4 \, \frac 1M.
        \end{equation}
        Note that $\left |\frac hT \right | < A$ by assumption, and consequently
        \begin{equation*}
            \left | \frac h{T+ T_\zeta} \right | \le A.
        \end{equation*}
        Exploiting \eqref{new0000} and the Lipschitz character of $\be$, we get
        \begin{eqnarray}
            \be \left ( \frac h {T+T_\zeta} \right ) &=& \be \left ( \frac h {T+ T_\zeta} \right )- \be \left ( \frac h {T} \right ) + \be \left ( \frac h {T} \right )\label{new1}\\
            &\ge& \be \left ( \frac h{T} \right ) - \frac \de 4. \nonumber
        \end{eqnarray}
        Taking into account that $\LL$ is positive and \eqref{new00}, we further derive from the leftmost equality in \eqref{new123} and the $ \frac \de 4$--optimality of $\xi$
        \begin{eqnarray}
            A(\nu) &\le& \frac 1{T+T_\zeta} \, ( \Phi_\LL(x,y,T;h) + \de/4 + | \EE_0| \, M )\label{new2} \\
            &\le& \frac 1T \, \Phi_\LL(x,y,T;h) + \left ( \frac 1{T+ T_\zeta} - \frac 1T \right ) \, \Phi_\LL(x,y,T;h) + \frac \de 2 \nonumber\\
            & \le& \frac 1T \, \Phi_\LL(x,y,T;h) + \frac \de 2. \nonumber
        \end{eqnarray}
        By combining \eqref{new1}, \eqref{new2}, we obtain
        \begin{eqnarray*}
            \be \left ( \frac h T \right ) &\le& \be \left ( \frac h{T+ T_\zeta} \right ) + \frac \de 4 \\
            &\le& A(\nu) + \frac \de 4 \le \frac 1T \, \Phi_\LL(x,y,T;h) + \frac \de 4+ \frac \de 2,
        \end{eqnarray*}
        as it was claimed.

        \medskip

        \noindent{\em Second estimate:} \quad We proceed by showing that
        \begin{equation*}
            \be(h/T) \ge \frac 1T \, \Phi_\LL(x,y,T;h) - \de.
        \end{equation*}

        We first assume that $h=0$ or, in other terms that $\bar\la =1$ in \eqref{lubianozzo}. We recall that $\bar e$ satisfies $a_{\bar e}=a_0$. We consider a parametrized path $\xi= (e_i,1,1)_{i=1}^m$ made up by edges of $\EE_\TT$ linking $x$ to $y$ and passing through $\oo(\bar e)$; we denote by $\xi_1$ the subpath of $\xi$ between $x$ and $\oo(\bar e)$, and by $\xi_2$ the one connecting $\oo(\bar e)$ and $y$. We finally set $\eta = (\bar e,-\bar e,0,0, (T-m)/2,(T-m)/2)$. Note that the rotation vector of $\xi_1 \cup \eta \cup \xi_2$ is vanishing. We have
        \begin{equation*}
            A_\LL(\xi_1 \cup \eta \cup \xi_2)= \sum_{i=1}^m \LL(e_i,1) + (T-m) (-a_0).
        \end{equation*}
        We therefore exploit that $\be$ is positive, $M$ is an upper bound of $\LL(e,1)$, for $e$ varying in $\EE_0$ and use \eqref{eq:minbe}, \eqref{new00} to get
        \begin{eqnarray*}
            \be(0) - \frac 1T \, \Phi_\LL(x,y,T;0) &\ge&\be(0) - \frac 1T \, A_\LL(\xi_1 \cup \eta \cup \xi_2) \\
            &\ge& - \frac{M |\EE_0|} T +\frac mT \, \be(0) \ge - \de,
        \end{eqnarray*}
        as claimed.

        \smallskip

        For $h \ne 0$ we will essentially use Lemma \ref{lequile}. This estimate is the most delicate one. We recall the following relations from Lemma \ref{lequile}:
        \begin{eqnarray}
            \label{new903} A(\mu) &=& \frac 1T \, \left ( \sum_{i=1}^m n_i \, A_\LL(\xi_i) + \sum_{e \in \cup_i \spt \xi_i} \frac{m_e}{\mathcal Q_p(e)} \, \LL(e,\mathcal Q_p(e)) \right ) - \bar \la \, a_0\\
            T&=& \frac 1{1-\bar\la} \,\sum_{e \in \cup_i \spt \xi_i} \left [ \left ( \sum_{i \in I(e)} n_i \right ) \, \frac 1{\mathcal Q_p(e)} + \frac{m_e}{\mathcal Q_p(e)} \right ]\label{new903tris}\\
            h &=& \sum_{ e \in \cup_i \spt \xi_i \setminus \EE_\TT} \left [ \left ( \sum_{i \in I(e)} n_i \right ) \theta(e) + m_e \, \theta(e) \right ],\label{new903bis}
        \end{eqnarray}
        see the statement of Lemma \ref{lequile} for more details on the above notation. We recall that
        \begin{equation}\label{new917-1}
            \xi_i=(e_j^i,\mathcal Q_p(e_j^i), 1/\mathcal Q_p(e_j^i))_{j=1}^{m_i}
        \end{equation}
        is a family of circuits and, according to \eqref{new903tris}, the total time of parametrization is given by
        \begin{equation}\label{new917}
            \sum_{i=1}^m S_i = \sum_{i=1}^m \, \sum_{j= 1}^{m_i} \frac 1 {\mathcal Q_p(e_j^i)} = (1 - \bar \la) \, T.
        \end{equation}
        In order to get a path connecting $ x$ and $ y$, we need to modify the $\xi_i$'s as follows:
        \begin{itemize}
            \item[(i)] add connecting paths between $\tt(\xi_i)$ and $\oo(\xi_{i+1})$ belonging to $\TT$, for $i=1, \cdots m-1$;
            \item[(ii)] replace any $e \in (\cup_i \spt \xi_i) \setminus \EE_\TT$ such that $m_e \ne 0$ with the circuit given by Lemma \ref{crycry}, keeping the same multiplicity $m_e$, which is actually a positive integer by Lemma \ref{lequile};
            \item[(iii)] if $e \in (\cup_i \spt \xi_i) \cap \EE_\TT$ with $m_e \ne 0$ add an equilibrium circuit based on $e$ for a time $\frac{m_e}{\mathcal Q_p(e)}$;
            \item[(iv)] add a path in $\TT$ connecting $x$ to $\oo(\bar e)$, see \eqref{lubianozzo};
            \item [(v)] add a path in $\TT$ connecting $\oo(\bar e)$ to $x$;
            \item[(vi)] add a path in $\TT$ connecting $ x$ to $ \oo(\xi_1)$;
            \item[(vii)] add a path in $\TT$ connecting $\tt(\xi_m)$ to $ y$.
        \end{itemize}
        If $\bar \la =0$ items (iv), (v) should be removed. The concatenation of edges in the new path is as follows: we start at $ x$ and reach $\oo( \bar e)$ following the path in item (iv) then stay in the equilibrium circuit $(\bar e, -\bar e)$ at speed $0$ for a time $\bar \la \, T$ and finally go back to $x$ through the path in (v). We continue the trip going from $x$ to $\oo(\xi_1)$ through the path in item (vi). We proceed travelling the circuit $\xi_1$ for $n_1$ times, when we meet for the first time an edge $e$ as in item (ii), we pause the travel on $\xi_1$ and move along the circuit $\theta(e)$ for $m_e$ times, then we resume. When we meet for the first time an edge $e$ as in item (iii), we take a stopover at $e$ as indicated in (iii). When we reach through this procedure $\tt(\xi_1)$, we take the path between $\tt(\xi_1)$ and $\oo(\xi_2)$ introduced in item (i). We proceed travelling along $\xi_2$, when we meet an edge, as in item(ii), (iii), not previously seen, we apply the same procedure as above, and so on. Once travelled in this way all the $\xi_i$'s, we finally follow the path connecting $\tt(\xi_m)$ to $ y$ introduced in (vii), to conclude the journey from $ x$ to $ y$.

        The number of added edges in (i) can be estimated from above by $C_0 \, | \EE_0|$, those in (ii) by $| \EE_0|$ ( $\ge$ number of the edges in (ii)) times $C_0$ ($\ge$ the multiplicity $m_e$) times $|\EE_0|$ ($\ge$ number of edges in each added circuit, see (ii), which gives $C_0| \EE_0|^2$, no added edges in (iii), and finally those in (iv) plus (v) plus (vi) plus (vii) by $4 \,| \EE_0|$.

        We denote by $k$ the overall number of added edges counted with their multiplicity, and, according to what pointed out above, we obtain the estimate
        \begin{equation}\label{new501}
            k \le C_0 \, | \EE_0| + C_0| \EE_0|^2 + 4 \, | \EE_0| \le (2 C_0+ 4) \, |\EE_0|^2.
        \end{equation}
        We further impose that all the added edges are travelled with speed $1$ and consequently in a time $1$. Since all the extra edges belong to $\TT$ we get, by \eqref{new903bis}, a parametrized path $\zeta$ connecting $x$ to $y$ with
        \begin{equation*}
            \theta([\zeta])= h.
        \end{equation*}
        In addition, since the time spent by $\zeta$ on the original (non added) edges is given by
        \begin{equation}\label{newciaociao}
            \bar \la \, T + \sum_{e \in \cup_i \spt \xi_i} \left ( \sum_{i \in I(e)} n_i \right ) \frac 1{\mathcal Q_p(e)} + \frac{m_e}{\mathcal Q_p(e)} = T,
        \end{equation}
        see \eqref{new903tris}, the total time of parametrization of $\zeta$ is $T+k$. The action of $\zeta$ can be written as
        \begin{eqnarray*}
            A_\LL(\zeta)&=& \sum_i n_i \, A_\LL(\xi_i) + \sum_{e \in ( \cup_i \spt \xi_i) \setminus \EE_T } \frac{m_e}{\mathcal Q_p(e)} \, \LL(e,\mathcal Q_p(e)) \\
            &+& \sum_{e \in ( \cup_i \spt \xi_i) \cap \EE_T } \frac{m_e}{\mathcal Q_p(e)} \, \LL(e,0) - \bar \la \, T \, a_0 + \sum_{f \in F} k_f \, \LL(f,1),
        \end{eqnarray*}
        where $F$ denotes the family of edges in the added paths and $k_f$ their multiplicity, it is clear that $\sum_f k_f= k$. We then derive from \eqref{new903}, taking into account \eqref{new501} and observing that $M$ is a Lipschitz constant of $\LL(e,\cdot)$ in $[0, 1/a]$ for any $e$, that
        \begin{eqnarray*}
            \frac 1T \, A_\LL(\zeta) - \be( h/T) &=& \frac 1T \, \sum_{e \in (\cup_i \spt \xi_i) \cap \EE_\TT } \frac{m_e}{\mathcal Q_p(e)} \, [ \LL(e,0) -\LL(e,\mathcal Q_p(e))] \\ &+ & \frac 1T \, \sum_{f \in F} k_f \, \LL(f,1) \\
            &\le& \frac 1T \, [ C_0 |\EE_0| \, M + M \, k ] \le \frac MT \, (3 C_0+4) |\EE_0|^2
        \end{eqnarray*}
        and, by \eqref{new00bis}, we get
        \begin{equation}\label{new3000}
            \be( h/T) \ge \frac 1T \, A_\LL(\zeta) - \frac \de 2.
        \end{equation}
        The next step is to change the parametrization on $\zeta$ in order to get $T$ as new total time of parametrization. We define
        \begin{eqnarray*}
            \bar t_e &:=& \frac 1{\mathcal Q_p(e)} \, \left ( 1 - \frac{k} T \right ) \qquad\text{ for $e \in \cup_i \spt \xi_i$}\\
            \bar t_{\bar e} &:=& (\bar \la \, T) \, \left ( 1 - \frac{k} T \right );
        \end{eqnarray*}
        by \eqref{new2001}, \eqref{new00bis}, \eqref{new501} we have
        \begin{equation}\label{new500}
            1 - \frac{k}{T} > 1 - \frac{(2 C_0+4 ) \, |\EE_0|^2}T > 1 - \frac \de 2> \frac 12.
        \end{equation}
        We now consider a new path $\bar\zeta$ possessing the same edges of $\zeta$, with the same multiplicities, but with the following change of parametrization on the non added edges:
        \begin{itemize}
            \item[--] from $(e, \mathcal Q_p(e),1/\mathcal Q_p(e))$ to $(e, 1/\bar t_e, \bar t_e)$ for the $e$'s in the first term of the right hand-side of \eqref{new903tris};
            \item[--] again from $(e, \mathcal Q_p(e),1/\mathcal Q_p(e))$ to $(e, 1/\bar t_e, \bar t_e)$ for the $e$'s in the second term of the right hand-side of \eqref{new903tris}, {\it i.e.}, with $m_e >0$, not belonging to $\EE_\TT$;
            \item[--] from $(e, 0, m_e/\mathcal Q_p(e))$ to $(e,0, m_e \bar t_e)$ for the $e$'s in the second term of the right hand-side of \eqref{new903tris}, {\it i.e.}, with $m_e >0$, belonging to $\EE_\TT$;
            \item[--] from $(\bar e,0, \bar \la T)$ to $(\bar e,0, \bar t_{\bar e})$.
        \end{itemize}
        Taking into account \eqref{new917-1}, \eqref{new917}, the time spent on the $\xi_i$'s with the new parametrization is
        \begin{equation*}
            \sum_{i=1}^m \, \sum_{j= 1}^{m_i} \frac 1 {\mathcal Q_p(e_j^i)} \left ( 1 - \frac{k}{ T} \right ) = (1 - \bar \la) \, T \left ( 1 - \frac{k}{T} \right ) = (1 - \bar \la) \, T - (1 - \bar \la) \, k,
        \end{equation*}
        and the time spent on $\bar e$ is
        \begin{equation*}
            \bar \la \, T - \bar \la \,k.
        \end{equation*}
        Since the total time of parametrization of $\zeta$ is $T+k$, that of $\bar \zeta$ is actually $T$. By \eqref{new007} the edges in $\spt_\EE \mu \setminus \{\bar e\}$ can be divided in the following two classes:
        \begin{eqnarray*}
            \EE_1 &:=& \{ e \mid 1/\mathcal Q_p(e) \in [a,b]\} \\
            \EE_2 &:=& \{ e \mid 1/\mathcal Q_p(e) > b\},
        \end{eqnarray*}
        where the constants $a$, $b$ have been introduced at the beginning of the proof. By \eqref{new500}
        \begin{eqnarray}
            \frac 1{\bar t_e} & < & \frac2b \txt{for $e \in \EE_2$} \\
            \frac1b \; < \; \frac 1{\bar t_e} & < & \frac2a \txt{for $e \in \EE_1$.}\label{new20002}
        \end{eqnarray}
        Taking into account that $\frac 1{\mathcal Q_p(e)} > \bar t_e$ and \eqref{new201}, we get
        \begin{equation}\label{new20003}
            \frac 1{\mathcal Q_p(e)} \, \LL(e,\mathcal Q_p(e)) > \bar t_e \, \LL(e,1/\bar t_e) \txt{for any $e \in \EE_2$.}
        \end{equation}
        Since $M$ is a Lipschitz constant in $[a/2,b]$ for $t \mapsto t \, \LL(e,1/t)$, we further have by \eqref{new20002}, \eqref{new20003}, \eqref{newciaociao}
        \begin{eqnarray*}
            A_\LL(\zeta) - A_\LL ( \bar \zeta) &=& \sum_i n_i \left [\sum_{j=1}^{m_i} \left ( \frac 1{\mathcal Q_p(e_{ij})} \, \LL(e_{ij}, \mathcal Q_p{e_{ij}}) - \bar t_{e_{ij}} \, \LL(e_{ij}, 1/\bar t_{e_{ij}}) \right ) \right ] \\ &+& \sum_{e \in ( \cup_i \spt \xi_i) \setminus \EE_T } \left [ \frac{m_e}{\mathcal Q_p(e)} \, \LL(e, \mathcal Q_p(e)) - m_e \, \bar t_e \, \LL(e, 1/\bar t_e) \right ]\\
            &+& \sum_{e \in ( \cup_i \spt \xi_i) \cap \EE_T } \left [ \left ( \frac{m_e}{\mathcal Q_p(e)} - m_e \, \bar t_e \right ) \, \LL(e,0)\right ] + ( - \bar \la \, T + \bar t_{\bar e} ) \, a_0\\
            &\ge& \sum_i n_i \left [\sum_{e_{ij} \in \EE_1} \left ( \frac 1{\mathcal Q_{e_{ij}}} \, \LL(e_{ij}, \mathcal Q_p(e_{ij})) - \bar t_{e_{ij}} \, \LL(e_{ij}, 1/\bar t_{e_{ij}}) \right ) \right ] \\
            &+& \sum_{\substack{e \in ( \cup_i \spt \xi_i) \setminus \EE_T \\ e \in \EE_1 }} \left [ \frac{m_e}{\mathcal Q_p(e)} \, \LL(e, \mathcal Q_p(e)) - m_e \, \bar t_e \, \LL(e, 1/\bar t_e) \right ] \\
            &\ge& - \sum_i n_i \sum_{e_{ij} \in \EE_1} \left (\frac{M}{\mathcal Q_p(e_{ij})} \, \frac kT \right ) - \sum_{\substack{e \in ( \cup_i \spt \xi_i) \setminus \EE_T \\ e \in \EE_1 }} m_e \, \left ( \frac M{\mathcal Q_p(e)} \, \frac kT \right )\\
            &\ge& - M \, k.
        \end{eqnarray*}
        Consequently by \eqref{new00bis}
        \begin{equation*}
            \frac 1T \, (A_\LL(\zeta) - A_\LL ( \bar\zeta)) \ge - \frac{M \, k}T \ge \frac{M (2C_0+4) \, |\EE_0|^2}T \ge - \frac \de 2
        \end{equation*}
        and by \eqref{new3000}
        \begin{equation*}
            \be(h/T) \ge \frac 1T \, A_\LL(\zeta) - \frac \de 2\ge \frac 1T \, A_\LL ( \bar\zeta)- \de \ge \frac 1T \, \Phi_\LL(x,y,T;h) - \de.
        \end{equation*}
        This concludes the proof.
    \end{proof}

    \medskip

    \subsection{Proof of Proposition \ref{lemnew}}\label{lemmanuovo}

    \lemnew*

    \medskip

    We start with a definition:

    \begin{Definition}\label{defiuno}
        Let $\zeta:[0,T]\to\NN$ be an admissible curve, with corresponding decomposition $\{t_i\}_{i=1}^{m +1}$ of its interval of definition, such that
        \begin{equation*}
            \zeta([t_i,t_{i+1}]) \subset \spt \ga_i \qquad\text{for $i=1, \cdots, m$},
        \end{equation*}
        for some arc $\ga_i$ of $ \EN$. We call a {\em projection} of $\zeta$ on $\G_0$, any path made up
        \begin{itemize}
            \item[--] by edges $e_i$, in correspondence to intervals $[t_i,t_{i+1}]$ where $\zeta([t_i,t_{i+1}]) \subset \spt \ga_i $ is nonconstant, with
            \begin{equation*}
                \Psi_\EN(\ga_i) = (e_i,h_i) \in \EE \txt{for some $h_i \in\Z^{b(\G_0)}$,}
            \end{equation*}
            where $\Psi_\EN$ is one of the bijection relating $\NN$ to the underlying graph $\G$, see Notation \ref{pi};
            \item[--] by $\pm e$ for intervals $[t_i,t_{i+1}]$ where $\zeta$ is instead supported by a vertex $z$, where $\oo(e)= \pi_1(z)$, see Notation \ref{pi}.
        \end{itemize}
    \end{Definition}

    \smallskip

    \begin{Lemma}\label{stracuzzi}
        Any projected path on $\G_0$ of an admissible curve $\zeta$ in $\NN$ linking two vertices $z_1$, $z_2$, has rotation vector $\pi_2(z_2)-\pi_2(z_1)$.
    \end{Lemma}
    \begin{proof}
        Since $\zeta$ is admissible there is a decomposition of its interval of definition $\{t_i\}_{i=1}^{m +1}$ such that
        \begin{equation*}
            \zeta([t_i,t_{i+1}]) \subset \spt \ga_i \qquad\text{for $i=1, \cdots, m$},
        \end{equation*}
        where $\ga_i$ are arcs of $\NN$. If we remove the intervals of the decomposition where $\zeta$ is constant, we still obtain a curve linking $z_1$, $z_2$ and the rotation vector of the projection is not affected by Definition \ref{defiuno}. We can therefore assume that
        \begin{equation*}
            \zeta([t_i,t_{i+1}]) = \ga_i([0,1]) \txt{for any $i$.}
        \end{equation*}
        Consequently $(\Psi_\EN(\ga_i))_{i=1}^m$ is a path in $\G$ linking $\pi(z_1)$ to $\pi(z_2)$. This gives the assertion thanks to \eqref{prepatch}, \eqref{patch}.
    \end{proof}

    \smallskip

    We are ready to start the proof of Proposition \ref{lemnew}. We will essentially use Theorem \ref{babasic}.

    \begin{proof}[{\bf Proof of Proposition \ref{lemnew}}]
        According to Proposition \ref{curvacce}, we can take an optimal admissible curve $\zeta$ for $ \Phi_{L}(z_1,z_2,T)$, with corresponding partition $\{t_1,\dotsc, t_m\}$ of $[0,T]$. We denote by $\xi=(e_i)_{i=1}^m$ a projection of $\zeta$ on $\G_0$, according to Definition \ref{defiuno}. If $[t_i,t_{i+1}]$ is an interval where $\zeta$ is nonconstant and supported by an arc $\ga_i$ with $ \Psi_\EN (\ga_i)=(e_i,h_i)$, for some $h_i \in \Z^{b(\G_0)}$, we consider the parametrization
        \begin{equation*}
            (e_i, 1/(t_{i+1} -t_i), t_{i+1} -t_i)
        \end{equation*}
        and get by the optimality of $\zeta$, Proposition \ref{lemlemnew} and Theorem \ref{babasic}, that
        \begin{eqnarray}
            \int_{t_i}^{t_{i+1}} L(\zeta,\dot\zeta) \, dt &=& \int_{t_i}^{t_{i+1}} L_{\ga_i}( \Upsilon^{-1} (\zeta), D_t (\Upsilon^{-1}(\zeta)) \, dt\label{lemnew1} \\
            &\ge& (t_{i+1}-t_i) \, \LL(e_i, 1/(t_{i+1}-t_i)). \nonumber
        \end{eqnarray}
        If instead $\zeta$ is constant in $[t_i,t_{i+1}]$, supported by a vertex $\bar z_i$, we make correspond in $\xi$ the pair of parametrized edges $ \left (\pm e_0, 0,\frac{t_{i+1} -t_i}2 \right )$, with $\oo(e_0)= \pi_1(\bar z_i)$ and in addition
        \begin{equation}\label{lemnew11}
            \LL(e_0,0)= \wha c_{\bar z_i},
        \end{equation}
        see Lemma \ref{flusso}. We derive
        \begin{eqnarray}
            \int_{t_i}^{t_{i+1}} L(\zeta,\dot \zeta) &=& \wha c_{\bar z_i} \, (t_{i+1} - t_i)\label{lemnew2}\\
            &=& \LL(e_0,0) \, \frac{t_{i+1} - t_i}2 + \LL(-e_0,0) \, \frac{t_{i+1} - t_i}2. \nonumber
        \end{eqnarray}
        Due to \eqref{lemnew1}, \eqref{lemnew2}, we obtain
        \begin{equation*}
            A_{L}(\zeta) \ge A_\LL(\xi),
        \end{equation*}
        with the projected path $\xi$ endowed with the parametrization described above; taking into account that $\theta([\xi]) = \pi_2(z_2)-\pi_2(z_1)$ by Lemma \ref{stracuzzi}, we conclude that
        \begin{equation}\label{lemnew10}
            \Phi_{L}(z_1,z_2,T)= A_{L}(\zeta) \ge A_\LL(\xi) \ge \Phi_\LL(\pi_1(z_1), \pi_1(z_2),T; \pi_2(z_2)-\pi_2(z_1)).
        \end{equation}

        \smallskip

        Let us now start with an optimal parametrized path $\xi_0 =(e_i,q_i,t_{i+1} - t_i)_{i=1}^m$ for
        \begin{equation*}
            \Phi_\LL(\pi_1(z_1), \pi_1(z_2),T;\pi_2(z_2)- \pi_2(z_1))
        \end{equation*}
        with $t_1=0$, $t_{m+1}=T$. We fix $h_1= \pi_2(z_1)$, and denote by $\xi= ((e_i,h_i))_{i=1}^m$ the corresponding path in $\G$ given by the unique path-lifting property; thanks to \eqref{patchbis} we have
        \begin{equation*}
            \tt(\xi)= (\tt(\xi_0),h_1 + \theta(\xi_0)) = (\tt(\xi_0),\pi_2(z_2)) =\pi(z_2).
        \end{equation*}
        We thereafter set
        \begin{equation*}
            \ga_i= \Psi_\EN^{-1}((e_i,h_i)) \txt{for $i=1, \cdots, m$.}
        \end{equation*}
        By the optimality of $\xi_0$, using Theorem \ref{babasic} we can associate to the edge $e_i$, in the case where $q_i \ne 0$, a curve $\xi_i: [t_i,t_{i+1}] \to [0,1]$ with $\xi_i(t_i)=0$, $\xi_i(t_{i+1})=1$, with
        \begin{equation}\label{lemnew3}
            \int_{t_i}^{t_{i+1}} L_{e_i}(\xi_i, \dot\xi_i), \, dt = (t_{i+1}-t_i) \, \LL(e_i,q_i).
        \end{equation}
        If we set $\zeta_i:= \ga_i \circ \xi_i$, then according to Proposition \ref{lemlemnew}, \eqref{maxxi} and \eqref{lemnew3}, we have
        \begin{equation*}
            \int_{t_i}^{t_{i+1}} L(\zeta_i, \dot \zeta_i) \, dt \le\int_{t_i}^{t_{i+1}}L_{e_i}(\xi_i,\dot\xi_i)\,dt= (t_{i+1}-t_i) \, \LL(e_i,q_i).
        \end{equation*}
        If instead $(e_i,e_{i+1},0,0,t_{i+1}-t_i,t_{i+2}-t_{i+1})$ is an equilibrium circuit contained in $\xi_0$, with $e_{i+1}=-e_i$ and $\ga_i(0)=z_i$, we have by the minimality of $\xi_0$, see Lemma \ref{flusso}, that \eqref{lemnew11} holds true with $e_i$ in place of $e_0$. We then define
        \begin{eqnarray*}
            \zeta_i(t) &\equiv& z_i \quad\text{in $[t_i,t_{i+1}]$} \\
            \zeta_{i+1}(t) &\equiv& z_i \quad\text{in $[t_{i+1},t_{i+2}]$.}
        \end{eqnarray*}
        By concatenating all the $\zeta_i$'s, we get a curve $\zeta$ in $\NN$ linking $z_1$ to $z_2$ such that
        \begin{equation*}
            A_{L}(\zeta)\le A_\LL(\xi_0).
        \end{equation*}
        We derive
        \begin{equation*}
            \Phi_\LL(\pi_1(z_1), \pi_1(z_2),T;\pi_2(z_2)-\pi_2(z_1))= A_\LL(\xi_0)= A_{L}(\zeta) \ge \Phi_{L}(z_1,z_2,T),
        \end{equation*}
        which, together with \eqref{lemnew10}, gives the assertion.
    \end{proof}

    \bigskip

    \appendix

    \section{Construction of periodic networks}\label{appexnet}

    {}In this appendix we show how to explicitly construct a periodic network over any finite graph.

    \begin{Definition}\label{embe}
        Given a graph $\wha\G=(\wha\VV,\wha\EE)$ and a Euclidean space $\R^M$, we call {\em embedding} of $\wha\G$ in $\R^M$ a pair of maps
        \begin{equation*}
            \mathcal I_{\wha\VV}: \wha\VV \to \R^M, \qquad \mathcal I_{\wha\EE}: \wha\EE \to \mathcal K(0,1;\R^M),
        \end{equation*}
        where $\mathcal K(0,1;\R^M)$ stands for the space of regular simple curves of $\R^M$ defined in $[0,1]$, satisfying
        \begin{equation*}
            \mathcal I_{\wha\EE}(e)(0) = \mathcal I_{\wha\VV}(\oo(e)), \quad \mathcal I_{\wha\EE}(e)(1) = \mathcal I_{\wha\VV}(\tt(e)), \quad \mathcal I_{\wha\EE}(-e)(s)= \mathcal I_{\wha\EE}(e)(1-s) \quad\text{for $s \in [0,1]$}
        \end{equation*}
        and in addition
        \begin{equation*}
            \mathcal I_{\wha\EE}(e)([0,1]) \cap \mathcal I_{\wha\EE}(e')((0,1)) = \emptyset \qquad\text{ whenever $e'\ne \pm e$}.
        \end{equation*}
    \end{Definition}

    \medskip

    It is clear that the set
    \begin{equation*}
        \wha\NN :=\bigcup_{e \in \EE} \mathcal I_{\wha\EE}(e)([0,1]) \subset \R^M
    \end{equation*}
    is a network embedded in $\R^M$, the elements $\mathcal I_{\wha\VV}(x)$, for $x$ varying in $\wha\VV$, are the vertices of $\wha\NN$, and the curves $\mathcal I_{\wha\EE}(e)$, for $e$ varying in $\wha\EE$, the arcs. The initial graph $\wha \G$ is in turn the underlying graph of $\wha\NN$.

    \smallskip

    It is well known, (see \cite{Kainen74}) that every finite graph admits a (book) embedding in the three-dimensional Euclidean space. Note that the construction of periodic networks we are carrying out depends on the embedding of the base graph $\G_0$ we choose.

    \medskip

    Let $\G$ be the maximal topological crystal over $\G_0$ defined as in Section \ref{murano}. We choose a distinguished embedding of $\G$ in a Euclidean space to define a canonical representative in the equivalence class of the periodic networks over $\G_0$. We first denote by $(\mathcal I_{\VV_0}, \mathcal I_{\EE_0})$ an embedding of the finite graph $\G_0$ in $\R^K$, for some positive integer K, which does exist, as pointed above.

    \smallskip

    We set
    \begin{equation*}
        N_0 := b(\G_0) + K.
    \end{equation*}

    \smallskip

    \begin{Definition}
        We define an embedding $(\mathcal I_\VV, \mathcal I_\EE)$ of $ \G$ in $\R^{N_0}$ setting
        \begin{eqnarray*}
            \mathcal I_\VV((x,h)) &:=& (\mathcal I_{\VV_0}(x),h) \txt{for any $(x,h) \in \VV$} \\
            \mathcal I_\EE((e,h))(s) &:=& (\mathcal I_{\EE_0}(e)(s), h + s \theta(e)) \txt{for any $(e,h) \in \EE$, $s \in [0,1]$.}
        \end{eqnarray*}
    \end{Definition}

    \smallskip

    One can straightforwardly check that it is indeed an embedding, according to Definition \ref{embe}. We denote by $\NN'$ the corresponding network in $\R^{N_0}$ and by $ \VN'$, $ \EN'$ the sets of its vertices and arcs, respectively.

    \smallskip

    \begin{Proposition}
        The network $\NN'$ is periodic (in the sense of Definition \ref{vetro}).
    \end{Proposition}
    \begin{proof}
        Property \textbf{(i)} in Definition \ref{vetro} is immediate. Let $(e_i,h_i)$, $i=1, \, 2$, be arcs in the same orbit with respect to $\Z^{b(\G_0)}$, then $e_1=e_2=:e$ and the corresponding arcs on $\NN'$ are $(\mathcal I_{\EE_0}(e), h_1 + s \theta(e))$, $(\mathcal I_{\EE_0}(e), h_2 + s \theta(e))$, which possess the same Euclidean length in $\R^{N_0}$. This shows condition \textbf{(ii)} as well, and concludes the proof.
    \end{proof}

    It is easy to check:

    \begin{Lemma}
        An element $(x,h) \in \R^{N_0}$ belongs to $\NN'$ if and only if
        \begin{itemize}
            \item[--] $h$ has any component in $\Z$ except at most one which is of the form
            \begin{equation*}
                (1-s) \rho_1+ s \rho_2 \txt{for some $\rho_i$, $i=1,2$, in $\Z$ and $s \in [0,1]$;}
            \end{equation*}
            \item[--] there exist $e \in \EE_0$ and $s\in[0,1]$ with $\mathcal I_{\EE_0}(e)(s)= x$.
        \end{itemize}
    \end{Lemma}

    \smallskip

    \begin{Remark}\label{covnetbuild}
        According to Remark \ref{periodicgraph}, the network $\NN'$ is composed by infinite copies of $\TT$ (more precisely of embeddings of the maximal tree). The $(e,h)\in\EE$ with $\theta(e)\ne 0$, link them together. An example of this can be seen in figure \ref{fig:squarecov}.
    \end{Remark}

    \begin{figure}[ht]
        \begin{subcaptiongroup}
            \phantomcaption\label{subfig:square}
            \includegraphics{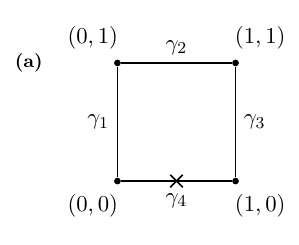}
            \qquad
            \phantomcaption\label{subfig:squaretree}
            \includegraphics{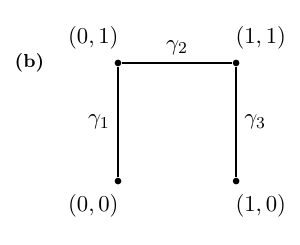}
            \\
            \phantomcaption\label{subfig:squarecov}
            \includegraphics{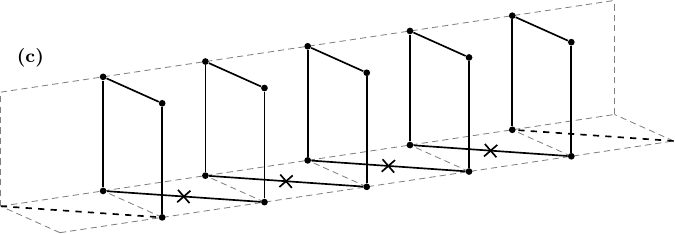}
        \end{subcaptiongroup}
        \caption{In \subref{subfig:square} it is represented an embedding in $\R^2$ of a finite graph $\G_0$ and in \subref{subfig:squaretree} the corresponding embedding of a spanning tree of $\G_0$, which we call $\NN_\TT'$. Using the process described above, we get in \subref{subfig:squarecov} the relative embedding in $\R^3$ of $\G$. As noted in Remark \ref{covnetbuild}, it is made of infinite copies of $\NN_\TT'$.}\label{fig:squarecov}
    \end{figure}


\smallskip

    \medskip

    \bigskip

    \begin{thebibliography}{99}
        {}
        \bibitem{AchdouCamilliCutriTchou12}
        Y. Achdou, F. Camilli, A. Cutrì and N. Tchou. ‘Hamilton–Jacobi equations constrained on networks’. In: \emph{Nonlinear Differential Equations and Applications NoDEA} 20.3 (Mar. 2012), pp. 413–445. \textsc{issn}: 1420-9004.
        \textsc{doi}: \href{https://doi.org/10.1007/s00030-012-0158-1} {\nolinkurl
            {10.1007/s00030-012-0158-1}}. {}
        \bibitem{AchdouLeBris23}
        Y. Achdou and C. Le Bris. ‘Homogenization of some periodic Hamilton-Jacobi equations with defects’. In: \emph{Communications in Partial Differential Equations} 48.6 (June 2023), pp. 944–986. \textsc{issn}: 1532-4133. \textsc{doi}:
        \href{https://doi.org/10.1080/03605302.2023.2238953} {\nolinkurl
            {10.1080/03605302.2023.2238953}}. {}
        \bibitem{AchdouOudetTchou16}
        Y. Achdou, S. Oudet and N. Tchou. ‘Effective transmission conditions for Hamilton–Jacobi equations defined on two domains separated by an oscillatory interface’. In: \emph{Journal de Mathématiques Pures et Appliquées} 106.6 (Dec. 2016), pp. 1091–1121. \textsc{issn}: 0021-7824. \textsc{doi}: \href{https://doi.org/10.1016/j.matpur.2016.04.002} {\nolinkurl{10.1016/j.matpur.2016.04.002}}. {}
        \bibitem{AchdouTchou15}
        Y. Achdou and N. Tchou. ‘Hamilton-Jacobi Equations on Networks as Limits of Singularly Perturbed Problems in Optimal Control: Dimension Reduction’. In: \emph{Communications in Partial Differential Equations} 40.4 (Jan. 2015), pp. 652–693. \textsc{issn}: 1532-4133. \textsc{doi}: \href{https://doi.org/10.1080/03605302.2014.974764} {\nolinkurl{10.1080/03605302.2014.974764}}. {}
        \bibitem{BarlesBrianiChasseigne14}
        G. Barles, A. Briani and E. Chasseigne. ‘A Bellman Approach for Regional Optimal Control Problems in $\mathbb{R}^N$’. In: \emph{SIAM Journal on Control and Optimization} 52.3 (Jan. 2014), pp. 1712–1744. \textsc{issn}: 1095-7138. \textsc{doi}: \href{https://doi.org/10.1137/130922288} {\nolinkurl{10.1137/130922288}}. {}
        \bibitem{BarlesBrianiChasseigne13}
        G. Barles, A. Briani and E. Chasseigne. ‘A Bellman approach for two-domains optimal control problems in $\mathbb{R}^N$’. In: \emph{ESAIM: Control, Optimisation and Calculus of Variations} 19.3 (June 2013), pp. 710–739.
        \textsc{issn}: 1262-3377. \textsc{doi}: \href
        {https://doi.org/10.1051/cocv/2012030} {\nolinkurl{10.1051/cocv/2012030}}. {}
        \bibitem{BarlesBrianiChasseigneTchou15}
        G. Barles, A. Briani, E. Chasseigne and N. Tchou. ‘Homogenization results for a deterministic multi-domains periodic control problem’. In: \emph{Asymptotic Analysis} 95.3–4 (Nov. 2015), pp. 243–278. \textsc{issn}: 0921-7134. \textsc{doi}:
        \href{https://doi.org/10.3233/asy-151322} {\nolinkurl{10.3233/asy-151322}}.
        {}
        \bibitem{BarlesChasseigne}
        G. Barles and E. Chasseigne. \emph{On Modern Approaches of Hamilton-Jacobi Equations and Control Problems with Discontinuities: A Guide to Theory, Applications, and Some Open Problems}. Progress in Nonlinear Differential Equations and Their Applications. Springer Nature Switzerland, 2024.
        \textsc{isbn}: 978-3-031-49371-3. \textsc{doi}: \href
        {https://doi.org/10.1007/978-3-031-49371-3} {\nolinkurl{10.1007/978-3-031-49371-3}}. {}
        \bibitem{BreH}
        A. Bressan and Y. Hong. ‘Optimal control problems on stratified domains’. In: \emph{Networks \& Heterogeneous Media} 2.2 (2007), pp. 313–331. \textsc{issn}: 1556-181X. \textsc{doi}: \href{https://doi.org/10.3934/nhm.2007.2.313} {\nolinkurl{10.3934/nhm.2007.2.313}}. {}
        \bibitem{Burago}
        D. Burago, Y. Burago and S. Ivanov. \emph{A Course in Metric Geometry}. Vol. 33. Graduate Studies in Mathematics. American Mathematical Society (AMS), June 2001. \textsc{isbn}: 978-1-4704-1794-9. \textsc{doi}: \href{https://doi.org/10.1090/gsm/033} {\nolinkurl{10.1090/gsm/033}}. {}
        \bibitem{ButtazzoGiaquintaHildebrandt99}
        G. Buttazzo, M. Giaquinta and S. Hildebrandt. \emph{One-dimensional Variational Problems. An Introduction}. Vol. 15. Oxford Lecture Series in Mathematics and Its Applications. Oxford University Press, 1999. \textsc{isbn}: 978-0-19-850465-8. {}
        \bibitem{CamilliMarchi13}
        F. Camilli and C. Marchi. ‘A comparison among various notions of viscosity solution for Hamilton–Jacobi equations on networks’. In: \emph{Journal of Mathematical Analysis and Applications} 407.1 (Nov. 2013), pp. 112–118.
        \textsc{issn}: 0022-247X. \textsc{doi}: \href
        {https://doi.org/10.1016/j.jmaa.2013.05.015} {\nolinkurl{10.1016/j.jmaa.2013.05.015}}. {}
        \bibitem{CamMar}
        F. Camilli and C. Marchi. ‘A continuous dependence estimate for viscous Hamilton–Jacobi equations on networks with applications’. In: \emph{Calculus of Variations and Partial Differential Equations} 63.1 (Dec. 2023). \textsc{issn}: 1432-0835. \textsc{doi}: \href{https://doi.org/10.1007/s00526-023-02619-y} {\nolinkurl{10.1007/s00526-023-02619-y}}. {}
        \bibitem{CCPS}
        E. Carlini, V. Coscetti, M. Pozza and A. Siconolfi. In preparation. {}
        \bibitem{ConItSic}
        G. Contreras, R. Iturriaga and A. Siconolfi. ‘Homogenization on arbitrary manifolds’. In: \emph{Calculus of Variations and Partial Differential Equations} 52.1–2 (Mar. 2014), pp. 237–252. \textsc{issn}: 1432-0835. \textsc{doi}: \href{https://doi.org/10.1007/s00526-014-0710-4} {\nolinkurl{10.1007/s00526-014-0710-4}}. {}
        \bibitem{Davini07}
        A. Davini. ‘Bolza Problems with Discontinuous Lagrangians and Lipschitz-Continuity of the Value Function’. In: \emph{SIAM Journal on Control and Optimization} 46.5 (Jan. 2007), pp. 1897–1921. \textsc{issn}: 1095-7138.
        \textsc{doi}: \href{https://doi.org/10.1137/060654311} {\nolinkurl
            {10.1137/060654311}}. {}
        \bibitem{Diestel17}
        R. Diestel. \emph{Graph Theory}. Graduate Texts in Mathematics. Springer Berlin Heidelberg, 2017. \textsc{isbn}: 978-3-662-53622-3. \textsc{doi}: \href{https://doi.org/10.1007/978-3-662-53622-3} {\nolinkurl{10.1007/978-3-662-53622-3}}. {}
        \bibitem{ForcadelSalazar20}
        N. Forcadel and W. Salazar. ‘Homogenization of a discrete model for a bifurcation and application to traffic flow’. In: \emph{Journal de Mathématiques Pures et Appliquées} 136 (Apr. 2020), pp. 356–414.
        \textsc{issn}: 0021-7824. \textsc{doi}: \href
        {https://doi.org/10.1016/j.matpur.2019.12.004} {\nolinkurl{10.1016/j.matpur.2019.12.004}}. {}
        \bibitem{ForcadelSalazarZaydan18}
        N. Forcadel, W. Salazar and M. Zaydan. ‘Specified homogenization of a discrete traffic model leading to an effective junction condition’. In: \emph{Communications on Pure \& Applied Analysis} 17.5 (2018), pp. 2173–2206.
        \textsc{issn}: 1553-5258. \textsc{doi}: \href
        {https://doi.org/10.3934/cpaa.2018104} {\nolinkurl{10.3934/cpaa.2018104}}. {}
        \bibitem{GaliseImbertMonneau15}
        G. Galise, C. Imbert and R. Monneau. ‘A junction condition by specified homogenization and application to traffic lights’. In: \emph{Analysis \& PDE} 8.8 (Dec. 2015), pp. 1891–1929. \textsc{issn}: 2157-5045. \textsc{doi}: \href{https://doi.org/10.2140/apde.2015.8.1891} {\nolinkurl{10.2140/apde.2015.8.1891}}. {}
        \bibitem{ImbMon}
        C. Imbert and R. Monneau. ‘Flux-limited solutions for quasi-convex Hamilton-Jacobi equations on networks’. In: \emph{Annales scientifiques de l{\rq }École normale supérieure} 50.2 (2017), pp. 357–448. \textsc{issn}: 1873-2151. \textsc{doi}: \href{https://doi.org/10.24033/asens.2323} {\nolinkurl{10.24033/asens.2323}}. {}
        \bibitem{ImbMonZid}
        C. Imbert, R. Monneau and H. Zidani. ‘A Hamilton-Jacobi approach to junction problems and application to traffic flows’. In: \emph{ESAIM: Control, Optimisation and Calculus of Variations} 19.1 (Mar. 2012), pp. 129–166.
        \textsc{issn}: 1262-3377. \textsc{doi}: \href
        {https://doi.org/10.1051/cocv/2012002} {\nolinkurl{10.1051/cocv/2012002}}. {}
        \bibitem{Kainen74}
        P. C. Kainen. ‘Some recent results in topological graph theory’. In: \emph{Graphs and Combinatorics}. Springer Berlin Heidelberg, 1974, pp. 76–108.
        \textsc{isbn}: 978-3-540-37809-9. \textsc{doi}: \href
        {https://doi.org/10.1007/bfb0066436} {\nolinkurl{10.1007/bfb0066436}}. {}
        \bibitem{LPV}
        P.-L. Lions, G. Papanicolaou and S. Varadhan. ‘Homogenization of Hamilton-Jacobi equation’. Jan. 1987. {}
        \bibitem{LionsSouganidis17}
        P.-L. Lions and P. Souganidis. ‘Well-posedness for multi-dimensional junction problems with Kirchoff-type conditions’. In: \emph{Rendiconti Lincei - Matematica e Applicazioni} 28.4 (Nov. 2017), pp. 807–816. \textsc{issn}: 1120-6330. \textsc{doi}: \href{https://doi.org/10.4171/rlm/786} {\nolinkurl{10.4171/rlm/786}}. {}
        \bibitem{LionsSouganidis16}
        P.-L. Lions and P. E. Souganidis. ‘Viscosity solutions for junctions: well posedness and stability’. In: \emph{Rendiconti Lincei, Matematica e Applicazioni} 27.4 (Oct. 2016), pp. 535–545. \textsc{issn}: 1720-0768. \textsc{doi}: \href{https://doi.org/10.4171/rlm/747} {\nolinkurl{10.4171/rlm/747}}. {}
        \bibitem{Mather91}
        J. N. Mather. ‘Action minimizing invariant measures for positive definite Lagrangian systems’. In: \emph{Mathematische Zeitschrift} 207.1 (May 1991), pp. 169–207. \textsc{issn}: 1432-1823. \textsc{doi}: \href{https://doi.org/10.1007/bf02571383} {\nolinkurl{10.1007/bf02571383}}. {}
        \bibitem{Morfe20}
        P. S. Morfe. ‘Convergence \& rates for Hamilton–Jacobi equations with Kirchoff junction conditions’. In: \emph{Nonlinear Differential Equations and Applications NoDEA} 27.1 (Jan. 2020), p. 10. \textsc{issn}: 1420-9004. \textsc{doi}: \href{https://doi.org/10.1007/s00030-020-0615-1} {\nolinkurl{10.1007/s00030-020-0615-1}}. {}
        \bibitem{Nash}
        J. Nash. ‘The Imbedding Problem for Riemannian Manifolds’. In: \emph{The Annals of Mathematics} 63.1 (Jan. 1956), pp. 20–63. \textsc{issn}: 0003-486X.
        \textsc{doi}: \href{https://doi.org/10.2307/1969989} {\nolinkurl
            {10.2307/1969989}}. {}
        \bibitem{PozzaLTB}
        M. Pozza. \emph{Large Time Behavior of Solutions to Hamilton-Jacobi Equations on Networks}. June 2023. arXiv: \href{https://arxiv.org/abs/2303.03872v2} {\nolinkurl{2303.03872v2}}. Submitted. {}
        \bibitem{PozzaSiconolfi21}
        M. Pozza and A. Siconolfi. ‘Discounted Hamilton-Jacobi equations on networks and asymptotic analysis’. In: \emph{Indiana University Mathematics Journal} 70.3 (2021), pp. 1103–1129. \textsc{issn}: 0022-2518. \textsc{doi}: \href{https://doi.org/10.1512/iumj.2021.70.8435} {\nolinkurl{10.1512/iumj.2021.70.8435}}. {}
        \bibitem{PozSic}
        M. Pozza and A. Siconolfi. ‘Lax–Oleinik Formula on Networks’. In: \emph{SIAM Journal on Mathematical Analysis} 55.3 (June 2023), pp. 2211–2237.
        \textsc{issn}: 1095-7154. \textsc{doi}: \href
        {https://doi.org/10.1137/21m1448677} {\nolinkurl{10.1137/21m1448677}}. {}
        \bibitem{RaoSiconolfiZidani14}
        Z. Rao, A. Siconolfi and H. Zidani. ‘Transmission conditions on interfaces for Hamilton–Jacobi–Bellman equations’. In: \emph{Journal of Differential Equations} 257.11 (Dec. 2014), pp. 3978–4014. \textsc{issn}: 0022-0396. \textsc{doi}: \href{https://doi.org/10.1016/j.jde.2014.07.015} {\nolinkurl{10.1016/j.jde.2014.07.015}}. {}
        \bibitem{Rockafellar}
        R. Rockafellar. \emph{Convex Analysis}. Princeton Landmarks in Mathematics and Physics. Princeton University Press, 1997. \textsc{isbn}: 978-0-691-01586-6. {}
        \bibitem{SchCam}
        D. Schieborn and F. Camilli. ‘Viscosity solutions of Eikonal equations on topological networks’. In: \emph{Calculus of Variations and Partial Differential Equations} 46.3–4 (Jan. 2012), pp. 671–686. \textsc{issn}: 1432-0835. \textsc{doi}:
        \href{https://doi.org/10.1007/s00526-012-0498-z} {\nolinkurl
            {10.1007/s00526-012-0498-z}}. {}
        \bibitem{SicMTA}
        A. Siconolfi. ‘Minimal action on the space of measures and {Hamilton}-{Jacobi} equations’. In: \emph{Minimax Theory and its Applications} 8.1 (2023), pp. 213–234. \textsc{issn}: 2199-1413. {}
        \bibitem{Sic}
        A. Siconolfi. ‘Time–dependent Hamilton–Jacobi equations on networks’. In: \emph{Journal de Mathématiques Pures et Appliquées} 163 (July 2022), pp. 702–738. \textsc{issn}: 0021-7824. \textsc{doi}: \href{https://doi.org/10.1016/j.matpur.2022.05.020} {\nolinkurl{10.1016/j.matpur.2022.05.020}}. {}
        \bibitem{SicSor1}
        A. Siconolfi and A. Sorrentino. ‘Aubry–Mather theory on graphs’. In: \emph{Nonlinearity} 36.11 (Sept. 2023), pp. 5819–5859. \textsc{issn}: 1361-6544.
        \textsc{doi}: \href{https://doi.org/10.1088/1361-6544/acf6ef} {\nolinkurl
            {10.1088/1361-6544/acf6ef}}. {}
        \bibitem{SicSor}
        A. Siconolfi and A. Sorrentino. ‘Global results for eikonal Hamilton–Jacobi equations on networks’. In: \emph{Analysis \& PDE} 11.1 (Jan. 2018), pp. 171–211. \textsc{issn}: 2157-5045. \textsc{doi}: \href{https://doi.org/10.2140/apde.2018.11.171} {\nolinkurl{10.2140/apde.2018.11.171}}. {}
        \bibitem{SorrentinoHomogenization}
        A. Sorrentino. \emph{On the Homogenization of the Hamilton-Jacobi Equation}. Apr. 2019. arXiv: \href{https://arxiv.org/abs/1904.01359v1} {\nolinkurl{1904.01359v1}}. {}
        \bibitem{Sunada}
        T. Sunada. \emph{Topological Crystallography}. Springer Japan, 2013.
        \textsc{isbn}: 978-4-431-54177-6. \textsc{doi}: \href
        {https://doi.org/10.1007/978-4-431-54177-6} {\nolinkurl{10.1007/978-4-431-54177-6}}. {}
        \bibitem{Villani}
        C. Villani. \emph{Optimal Transport}. Grundlehren der mathematischen Wissenschaften. Springer Berlin Heidelberg, 2009. \textsc{isbn}: 978-3-540-71050-9.
        \textsc{doi}: \href{https://doi.org/10.1007/978-3-540-71050-9} {\nolinkurl
            {10.1007/978-3-540-71050-9}}.
    \end{thebibliography}
\end{document}